\documentclass[12pt,reqno]{article}
\usepackage{amsmath,amssymb,amsthm}
\usepackage[colorlinks=true]{hyperref}
\usepackage{authblk}
\usepackage{enumitem}
\usepackage[margin=1in]{geometry}
\usepackage{courier}
\usepackage{graphicx}
\usepackage{amsthm} 
\usepackage{subfig}
\usepackage{xcolor}
\definecolor{newblue}{rgb}{0.0, 0.48, 0.65}
\usepackage{mathrsfs}
\usepackage{esint}  

\definecolor{light-gray}{gray}{0.95}
\usepackage{tcolorbox}

\usepackage{mathtools}

\DeclarePairedDelimiter\abs{\lvert}{\rvert}%
\DeclarePairedDelimiter\norm{\lVert}{\rVert}%
\DeclarePairedDelimiter\pair{\langle}{\rangle}%
\makeatletter
\let\oldabs\abs
\def\abs{\@ifstar{\oldabs}{\oldabs*}}
\let\oldnorm\norm
\def\norm{\@ifstar{\oldnorm}{\oldnorm*}}
\let\oldpair\pair
\def\pair{\@ifstar{\oldpair}{\oldpair*}}
\makeatother

\usepackage{ulem}

\usepackage{rotating}

\usepackage{lineno}

\usepackage{tikz}
\usetikzlibrary{arrows,chains,matrix,positioning,scopes}
\makeatletter
\tikzset{join/.code=\tikzset{after node path={%
\ifx\tikzchainprevious\pgfutil@empty\else(\tikzchainprevious)%
edge[every join]#1(\tikzchaincurrent)\fi}}}
\makeatother
\tikzset{>=stealth',every on chain/.append style={join},
         every join/.style={->}}
\tikzstyle{labeled}=[execute at begin node=$\scriptstyle,
   execute at end node=$]

\newtheorem{theorem}{Theorem} [section] 
\newtheorem{lemma}[theorem]{Lemma} 
\newtheorem{proposition}[theorem]{Proposition}

\theoremstyle{definition}
\newtheorem{definition}{Definition}[section]

\theoremstyle{remark}
\newtheorem{remark}{Remark}[section]

\newcommand{\re}{\mathbb{R}}

\newcommand{\dx}{\, \textnormal{d}x}
\newcommand{\ds}{\, \textnormal{d}s}

\newcommand{\dt}{\, \textnormal{d}t}

\newcommand{\eps}{\varepsilon}

\newcommand{\uoneeps}{u_1^{\varepsilon}}
\newcommand{\utwoeps}{u_2^{\varepsilon}}

\newcommand{\ueps}{u^{\varepsilon}}

\newcommand{\veps}{v^{\varepsilon}}

\newcommand{\vepshat}{\widehat{v}^{\varepsilon}}

\title{An aggregation model of cockroaches with fast-or-slow motion dichotomy}
\author[1]{J\'{a}n Elia\v{s}}
\author[2]{Hirofumi Izuhara}
\author[3]{Masayasu Mimura}
\author[1]{Bao Q. Tang}
\affil[1]{Institute of Mathematics and Scientific Computing, University of Graz, Heinrichstrasse 36, 8010 Graz, Austria}
\affil[2]{Faculty of Engineering, University of Miyazaki, 1-1 Gakuen Kibanadai Nishi, Miyazaki 889-2192, Japan}
\affil[3]{Graduate School of Integrated Sciences for Life, Hiroshima University, 1-3-1 Kagamiyama, Higashi-Hiroshima City, Hiroshima, Japan 739-8526}
\date{}

\newcommand{\mr}{\mathrm{mr}}
\newcommand{\yoneeps}{y_1^{\eps}}
\newcommand{\ytwoeps}{y_2^{\eps}}

\newcommand{\na}{\nabla}
\newcommand{\LQ}[1]{L^{#1}_{x,t}}

\newcommand{\intO}{\int_{\Omega}}
\newcommand{\bra}[1]{\left(#1\right)}
\newcommand{\sbra}[1]{\left[#1\right]}

\newcommand{\pa}{\partial}
\newcommand{\intQT}{\iint_{Q_T}}

\begin{document}



\maketitle

\begin{abstract} 
We propose a mathematical model, namely a reaction-diffusion system, to describe social behaviour of cockroaches. An essential new aspect in our model is that the dispersion behaviour due to overcrowding effect is taken into account {as a counterpart to commonly studied aggregation}. This consideration leads to an intriguing new phenomenon which has not been observed in the literature. Namely, due to the competition between aggregation towards areas of higher concentration of pheromone and dispersion avoiding overcrowded areas, 
the cockroaches aggregate more at the transition area of pheromone. Moreover, we also consider the fast reaction limit where the switching rate between active and inactive subpopulations tends to infinity. By utilising improved duality and energy methods, together with the regularisation of heat operator, we prove that the weak solution of the reaction-diffusion system converges to that of a reaction-cross-diffusion system.

\end{abstract}

\tableofcontents

\section{Introduction}

\subsection{Social behaviour of cockroaches}

Aggregation is a common behaviour of many species from bacteria to vertebrates which, among other benefits, increases the survival and reproduction success of a given individual \cite{Parrish-1999}. A typical example of such behaviour is the aggregation of soil-dwelling amoeba \textit{Dictyostelium discoideum}. In the absence of nutrients resource, the amoeba secretes a chemical signal, cAMP, attracting other amoebae to form a fruit body. 
This directional movement along a chemical concentration gradient is called chemotaxis and it became a basis for the famous Patlak-Keller-Segel model \cite{Patlak-1953, Keller-1970}. 

\medskip
Cluster formations are also typical for various cockroach species, such as the German cockroach, \textit{Blattella germanica}, and the American cockroach, \textit{Periplaneta americana}. Cockroaches respond to the airflow, as well as to the external objects (including predators) using a variety of mechanoreceptors located on their antennae \cite{Okada-2016}. The behaviour of cockroaches is further determined by a broad range of chemicals from toxins, food, pathogens, and pheromones which are sensed through chemoreceptors \cite{Harrison-2018}. Many cockroaches prefer dark, humid and enclosed spaces as shelters close to food resources. Notably, most cockroach species are gregarious, living in groups, which brings them many benefits ranging from reduction in predation risk, increased mating to humidity regulation and thermoregulation \cite{Dambach-1999}. A wide range of nonexclusive criteria such as mutual attraction, groups size, spatial and temporal variation in social structure, influences cockroach aggregation behaviour and to this day there is no definitive theory explaining the mechanism behind aggregation \cite[Chap 8]{Bell-2007}. Some cockroach species including \textit{Blattella germanica} are known to secrete an aggregation pheromone responsible for clustering
\cite{Ishii-1967, Ishii-1970}. 
 Indeed, the aggregation pheromone is shown to act initially as an attractant for cockroaches to group near the pheromone source, \cite{Bell-1972}, as well as to inhibit locomotion of individuals so that they  remain at the pheromone source \cite{Burk-1973}. Experimental and theoretical studies of shelter selection by cockroach populations further reveal that when the population size of cockroaches at the shelter increases, the individual probability of leaving it decreases and, consequently, newcomers are confined to the shelter. In particular, if the size of a shelter and thus the carrying capacity for the entire population is large enough, a single aggregation site is collectively selected \cite{Ame-2006, Sempo-2013}. The aggregation of cockroaches in this case likely results from a social amplification of a positive feedback via pheromone signal \cite{Sempo-2013}.

\medskip
A theory of the ideal free distribution however assumes that individuals compete for resources, e.g., shelters, and that each individual will go to a place where the chance of success is highest and, in particular, the individual will occupy a site among all sites with equally suitable conditions for survival where the number of competitors is lowest \cite{Fretwell-1969}. Sempo and coworkers \cite{Sempo-2013} show experimentally that a higher level of habitat fragmentation, which was achieved by a larger number of shelters of smaller sizes while the 
total carrying capacity of the habitat was maintained constant, cockroaches are randomly and equally distributed between the shelters.  
Thus, the smaller the shelter is, the fewer individuals it can contain and the probability of joining a shelter decreases with the group size occupying the shelter, 
leading to a negative feedback loop via crowding and, consequently, to dispersal of cockroaches. 

\medskip
Aggregation and dispersal, as two non-linear processes shaping the social behaviour of cockroaches, suggest the existence of a threshold which corresponds to a critical number of individuals occupying one shelter \cite{Sempo-2013}. Roughly speaking, whenever the population size of cockroaches is smaller than the critical threshold, cockroaches aggregate due to a positive feedback driven by the pheromone concentration. However, as soon as the critical threshold is reached dispersal prevails over aggregation despite high concentration of the chemoattractant.
In this paper we propose a mathematical model for the binary choice of cockroaches between aggregation and dispersal, as just described, based on a fast-or-slow motion dichotomy.

\subsection{Mathematical modelling}
In our modelling framework we assume that
\begin{itemize}
\item cockroaches are divided into two groups based on their ability to move: individuals in one group move actively (``fast"), while the mobility of the second group is inhibited (``slow")---hence, the fast-or-slow dichotomy; and
\item cockroaches can change their locomotive behaviour depending on the concentration of the aggregation pheromone.
\end{itemize}
 The spread of cockroaches in space and time and switching between the groups based on the pheromone concentration is precisely described by the following equations:
\begin{equation}\label{model1}  \left\{ \begin{aligned}
\partial_t u_{1} & =  d \Delta u_1  - \frac{1}{\eps}\left( r(v)u_1 - (1-r(v))u_2\right), \\
\partial_t u_{2} & = (d+D) \Delta u_2  + \frac{1}{\eps}\left( r(v)u_1 - (1-r(v))u_2\right), \\
\partial_t v & =  D_v \Delta v + \alpha(u_1+u_2) - \beta v, \\
\end{aligned} \right. \end{equation}
where $u_1$ and $u_2$ denote the population densities of slow and fast cockroach group, respectively, and $v$ is the concentration of aggregation pheromone. Individuals in both groups can move freely and randomly in space with constant rates of diffusion $d$ and $d+D$, respectively, where $d$ and $D$ are positive constants; in general $d \ll D$. Switching between the subpopulations $u_1$ and $u_2$ is determined by a probability distribution function $r$ which is a function of the pheromone concentration $v$ as shown in Figure~\ref{fig:01}. The rate of switching is given by $1/\eps$ for some $\eps >0$.

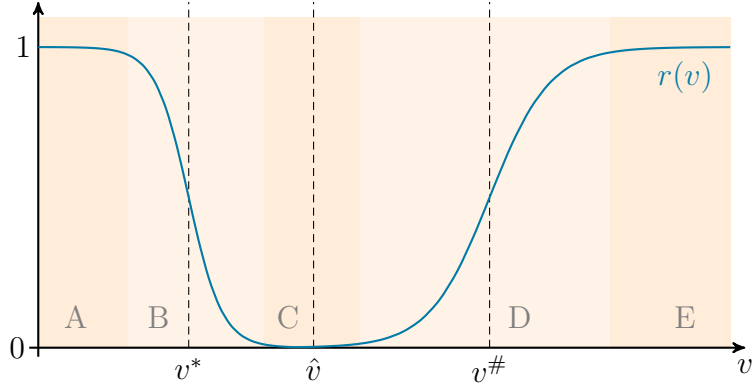
\begin{figure}[htbp]
\begin{center}
\begin{tikzpicture}[scale=4] 
\fill [orange!15] (0.0,0.0) rectangle (0.3,1.1); 
\fill [orange!10] (0.3,0) rectangle (0.75,1.1); 
\fill [orange!15] (0.75,0) rectangle (1.07,1.1); 
\fill [orange!10] (1.07,0) rectangle (1.9,1.1); 
\fill [orange!15] (1.9,0) rectangle (2.3,1.1); 
\draw[->, thick] (-0.03, 0) -- (2.35, 0) node[below] {$v$}; 
\draw[->, thick] (0, -0.03) -- (0, 1.15); 
\draw[densely dashed] (0.5, 0) -- (0.5, 1.15); \draw (0.5,-0.07) node {$v^\ast$};  
\draw[densely dashed] (1.5, 0) -- (1.5, 1.15); \draw (1.5,-0.07) node {$v^\#$};  
\draw[densely dashed] (0.915, 0) -- (0.915, 1.15); \draw (0.915,-0.07) node {$\hat{v}$};  
\draw[domain=0:0.85,smooth,variable=\x,newblue,thick] plot ({\x},  {0.5*(1-tanh(9*(\x-0.5)))});
\draw[domain=0.85:2.3,smooth,variable=\x,newblue,thick] plot ({\x},  {0.5*(1+tanh(5*(\x-1.5)))});
\draw[newblue] (2.15,0.9) node {$r(v)$};
\draw[black!50] (0.125,0.1) node {A};
\draw[black!50] (0.4,0.1) node {B};
\draw[black!50] (0.83,0.1) node {C};
\draw[black!50] (1.6,0.1) node {D};
\draw[black!50] (2.15,0.1) node {E};
\draw (-0.07,0) node {$0$};
\draw (-0.05,1) node {$1$};
\end{tikzpicture}
\caption{A typical example of the probability function $r$: $r(v) = \frac{1}{2}(1-\tanh(\gamma_1(v-v^\ast)))$ for $0 \le v \le \hat{v}$ and $r(v) = \frac{1}{2}(1+\tanh(\gamma_2(v-v^\#)))$ for $v > \hat{v}$, where $v^\ast, v^\#, \hat{v}, \gamma_1$ and $\gamma_2$ are positive constants. Depending on the concentration of  aggregation pheromone, $v$, we observe: dispersal of cockroaches (region A), change of cockroach motility from dispersal to aggregation (B), aggregation (C), change of cockroach motility from aggregation to dispersal (D), and dispersal again (E).}
\label{fig:01}
\end{center}
\end{figure}

The function $r$ is a real-valued function with values between zero and one and such that $r(0) = 1$, $r$ is decreasing for $0 < v \le \hat{v}$ and increasing for $v > \hat{v}$ for some $\hat{v} > 0$, and such that $r(v)$ tends to 1 as $v$ grows above all limits. If the concentration $v$ is small so that $r(v) \approx 1$ (region A in Figure~\ref{fig:01}), then $r(v)u_1 - (1-r(v))u_2 \approx u_1$, which together with the sign before $1/\eps$ in the equations for $u_1$ and $u_2$ implies switching of the mobility of cockroaches from slow, $u_1$, to fast, $u_2$. In this case, dispersal of cockroaches prevails over aggregation. If the concentration $v$ is such that $r(v) \approx 0$ (region C), then $r(v)u_1 - (1-r(v))u_2 \approx -u_2$, which implies that cockroaches switch their mobility from fast to slow. Thus aggregation dominates over dispersal. However, with further increasing $v$, $r$ may be close to one again (region E), which we associate with overcrowding, and thus dispersal dominates again. Regions B and D are transition regions in which change from dispersal to aggregation and vice-versa happens and the extent of switching relates to the concentration of the aggregation pheromone.


\medskip
In the model, the aggregation pheromone plays a role of signalling molecules called autoinducers in bacterial quorum sensing \cite{Rutherford-2012}: cockroaches sense the concentration of pheromone $v$ and perceive the crowding level in the neighbourhood, which then affects their locomotive behavior. The aggregation pheromone is spread in space with a diffusion rate $D_v$, which is secreted by every cockroach with a production rate $\alpha$ and degraded with a rate $\beta$.
 
 \medskip
We can further add (logistic) growth terms and terms for competition for resources between $u_1$ and $u_2$ in the model on top of the dichotomy dynamics. The resulting model equations are as follows:
\begin{equation}\label{model2} \left\{ \begin{aligned}
\partial_t u_{1} & =  d \Delta u_1 + a_1 (1 - u_1 - u_2)u_1 - \frac{1}{\eps}\left( r(v)u_1 - (1-r(v))u_2\right), \\
\partial_t u_{2} & = (d+D) \Delta u_2 + a_2 (1 - u_1 - u_2)u_2 + \frac{1}{\eps}\left( r(v)u_1 - (1-r(v))u_2\right), \\
\partial_t v & =  D_v \Delta v + \alpha(u_1+u_2) - \beta v, \\
\end{aligned} \right. \end{equation}
where $a_1$ and $a_2$ are the intrinsic growth rates of $u_1$ and $u_2$, respectively.
If $a_1=a_2=0$, the model \eqref{model2} reduces to the simple dichotomy dynamics model \eqref{model1}. In general, the growth rates $a_1$ and $a_2$ are two positive constants. We may, however, assume that $a_1\geq a_2$ since the growth rate of slow-motioned population is higher than that of the fast-motioned one because the environment where the cockroaches aggregate is more suitable for reproduction. 

\medskip
The switching between $u_1$ and $u_2$ in \eqref{model1} and \eqref{model2} is fully given by the difference of the terms  $r(v)u_1$ and $(1-r(v))u_2$. A more general case would be to consider two positive functions $p = p(v)$ and $q = q(v)$ and a mathematical model  
\begin{equation} \label{model3}
\left\{ 
\begin{aligned}
\partial_t u_{1} & =  d \Delta u_1 + a_1 (1 - u_1 - u_2)u_1 - \frac{1}{\varepsilon}\left( q(v)u_1 - p(v)u_2\right), \\
\partial_t u_{2} & = (d+D) \Delta u_2  + a_2 (1 - u_1 - u_2)u_2 + \frac{1}{\varepsilon}\left( q(v)u_1 - p(v)u_2\right), \\
\partial_t v & = D_v \Delta v + \alpha(u_1+u_2) - \beta v, \\
\end{aligned} 
\right. 
\end{equation}
where $q$ and $p$ satisfy the following modelling assumptions:
\begin{itemize}
\item If $v$ is small, $q(v)$ is close to a positive constant $r_1$ and $p(v)$ is close to zero. 
\item If $v$ is moderate, $q(v)$ is close to zero and $p(v)$ is close to a positive constant $r_2$. 
\item If $v$ is large, $q(v)$ is close to a positive constant $r_3$ and $p(v)$ is close to zero. 
\end{itemize}
The constants $r_1$, $r_2$ and $r_3$ do not necessarily have the same value. For the purpose of mathematical analysis we require from $p$ and $q$ to satisfy some smoothness and growth conditions specified later.

\medskip
The rate of conversion between $u_1$ and $u_2$ is denoted by $1/\varepsilon$, which is large whenever $\eps$ small. Since the time scale of the conversion of the states is in general much faster than the other time scales such as the dispersal and the growth of the populations, we may assume that $1/\eps$ is large (and thus $\eps$ small) compared with the diffusion rates $d$ and $d+D$ and the growth rates $a_1$ and $a_2$.
When we formally take the limit $\varepsilon\to 0$, we obtain that $q(v)u_1-p(v)u_2=0$. Here, noting that $u_1$ and $u_2$ are subpopulations, and denoting the total population as $u:=u_1+u_2$, we have $u_1=\frac{p(v)}{q(v)+p(v)}u$ and $u_2=\frac{q(v)}{q(v)+p(v)}u$. By adding the equations for $u_1$ and $u_2$, we obtain (formally) the following reduced model in the limit $\eps \to 0$: 
\begin{equation} \label{model4}
\left\{ 
\begin{aligned}
\partial_t u & =  \Delta \left( \left( d + D \frac{q(v)}{p(v)+q(v)} \right) u\right) + \frac{a_1 p(v) + a_2 q(v)}{p(v) + q(v)} (1 - u)u, \\
\partial_t v & =  D_v \Delta v + \alpha u - \beta v. 
\end{aligned} 
\right. 
\end{equation}
The feature of the reduced model is to naturally obtain a chemotactic movement. 
The nonlinear diffusion in the equation for $u$ indicates 
\[
\begin{aligned}
\partial_t u  &=  \Delta \left( \left( d + D \frac{q(v)}{p(v)+q(v)} \right) u\right) \\
&=\nabla \cdot \left( \left( d + D \frac{q(v)}{p(v)+q(v)} \right) \nabla u\right) + D \nabla\cdot \left( u \nabla \left(  \frac{q(v)}{p(v)+q(v)} \right) \right). 
\end{aligned}
\]
The first term means a density dependent diffusion and the second term describes a chemotactic effect which implies that the total population $u$ senses the chemical gradient $v$ on a macroscopic level. 
Hence, another aspect of the dichotomy dynamics is to approximate a two-component density dependent diffusion system \eqref{model4}, the so-called cross-diffusion system which naturally possesses a chemotactic movement, by a three-component reaction-diffusion system \eqref{model3}. The first result of this kind dates back to 2006 in the seminal work of Iida, Mimura and Ninomiya~\cite{iida-2006}.


\subsection{Contribution of our work}\label{contr}

In order to study aggregation of cockroaches, Funaki and coauthors \cite{Funaki-2012} propose both microscopic and macroscopic cross-diffusion model for the self-organised aggregation of cockroaches that includes directed movement due to an aggregation pheromone. A link between the micro- and macroscopic models is established by a reaction-diffusion system \eqref{model1} which is shown to approximate the macroscopic model, formally and rigorously, as a singular limit of problem~\eqref{model1} when $\eps \to 0$. In \cite{Funaki-2012} the model~\eqref{model1} with a strictly decreasing function $r$ is, however, considered and \textit{thus the effect of crowding is not taken into account}.
The model~\eqref{model2} with $a = a_1 = a_2 > 0$ and decreasing $r$ is further studied in \cite{Ei-2012}, in particular, the authors consider a reaction limit of problem \eqref{model2} for $a \to 0$ and every $\eps > 0$. 
Unlike in the papers \cite{Funaki-2012, Ei-2012}, the model equations \eqref{model3} proposed in this paper are more general and consider not only aggregation of cockroaches, which is \textit{only one side} of a complex and so far not fully understood behaviour of cockroaches, but also dispersal due to crowding by adapting a more realistic switching function such as the one depicted in Figure~\ref{fig:01}. 

\begin{figure}[tb!]
\begin{tabular}{p{0.46\textwidth} p{0.0\textwidth} p{0.46\textwidth}}
\includegraphics[width=0.46\textwidth]{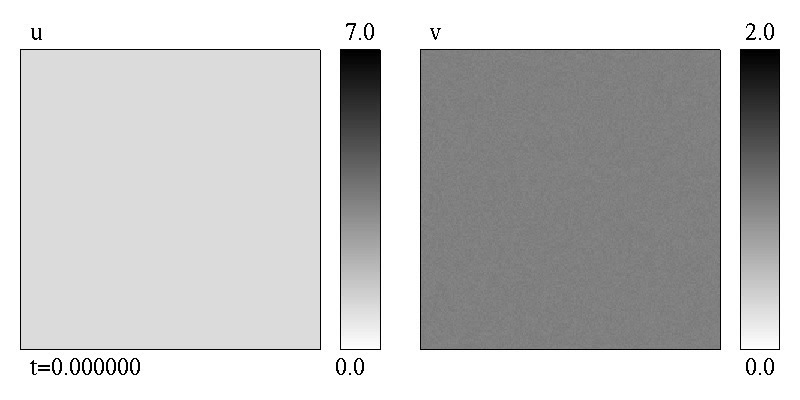} & & \includegraphics[width=0.46\textwidth]{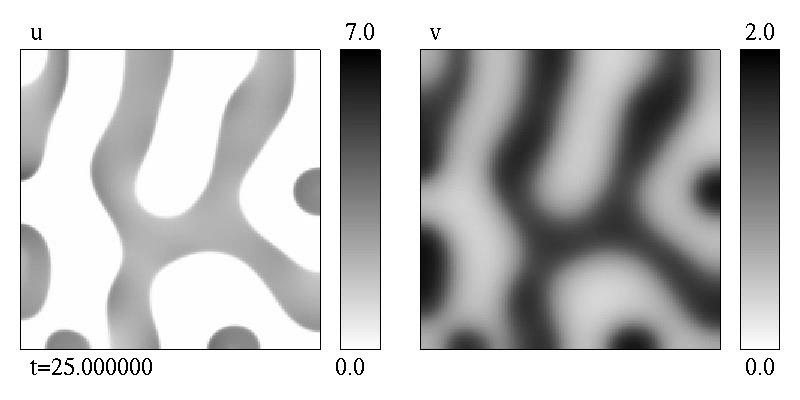}\\[-0.8em]
\includegraphics[width=0.46\textwidth]{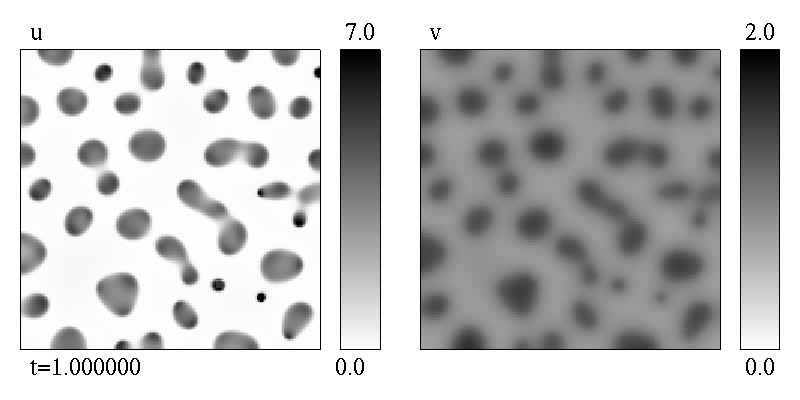} & & \includegraphics[width=0.46\textwidth]{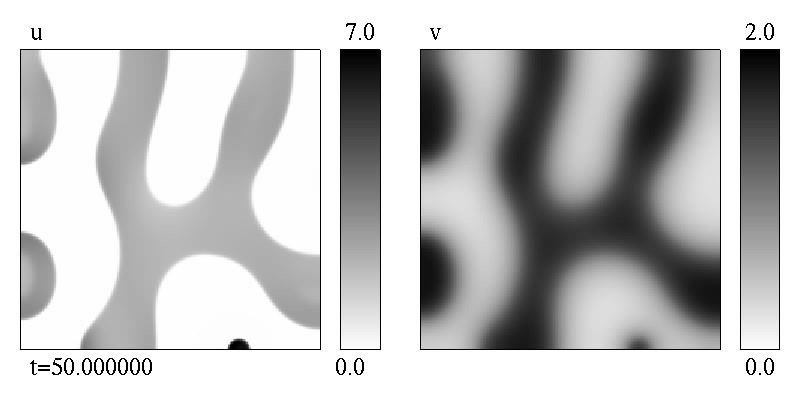}\\[-0.8em]
\includegraphics[width=0.46\textwidth]{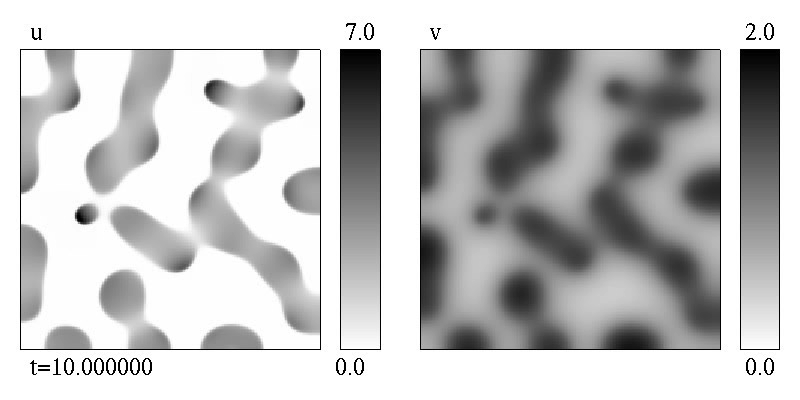} & & \includegraphics[width=0.46\textwidth]{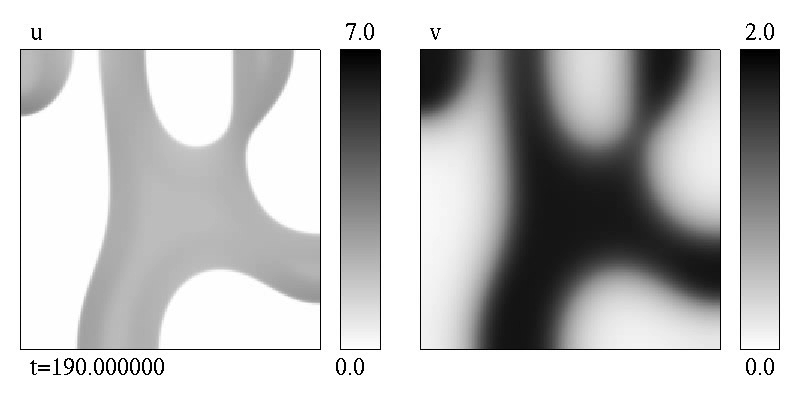}
\end{tabular} 
\caption{Aggregation of cockroaches for $v^\ast = 1$ and $v^\# = 2$: snapshots of spatial distribution of cockroaches and pheromone at different times. The total initial density of $u$ as well as the concentration of the pheromone $v$ is set to one. Moreover, additional noise is added to $v$.}
\label{fig:02}
\end{figure}

\medskip
With the probability function $r$ as in Figure~\ref{fig:01}, $q=r$ and $p = 1-r$, we see that the role of aggregation and dispersal in the model \eqref{model3} is shaped by two pheromone concentration values $v^\ast$ and $v^\#$ at which $r(v^\ast) = r(v^\#) = 1/2$ and the steepness of $r$ in the transition regions B and D. Let us consider $\gamma_1 = \gamma_2 = 20$, c.f. parameters in \eqref{param}. We remark that $r$ is rather steep in the transition regions B and D for chosen $\gamma_1$ and $\gamma_2$, hence the regions B and D are short intervals and we observe either dominating dispersal (regions A and E) or aggregation (region C) for most values of $v$. 
Different behaviour of the total population of cockroaches $u = u_1 + u_2$  can be observed depending on the distance between $v^\ast$ and $v^\#$. Figure~\ref{fig:02} shows spatio-temporal behaviour of cockroach density $u$ and the pheromone concentration $v$ from homogeneous initial states with some noise added to the initial pheromone concentration. In this simulation we set $v^\ast = 1$ and $v^\# = 2$ where $v^\# = 2$ is the maximum concentration of pheromone produced by cockroaches over time and space. For such high $v^\#$ the effect of dispersal due to crowding is minimal simply because $v$ does not get over $v^\#$. In this case cockroaches form large irregular aggregates
similar to the pattern obtained in \cite{Funaki-2012} where aggregation is only considered. 

\begin{figure}[tb!]
\begin{tabular}{p{0.46\textwidth} p{0.0\textwidth} p{0.46\textwidth}}
\includegraphics[width=0.46\textwidth]{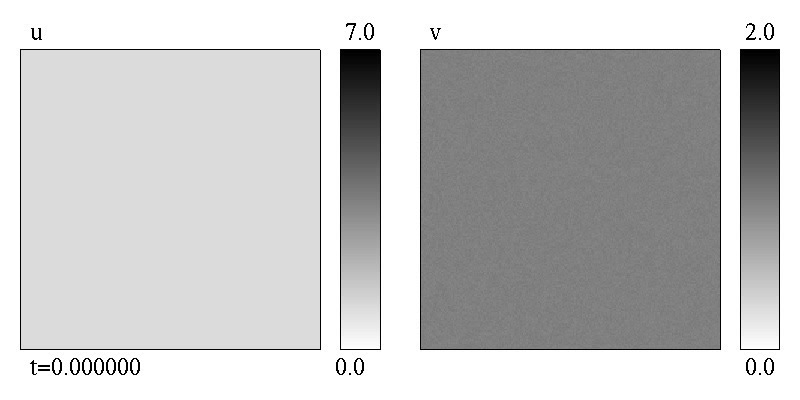} & & \includegraphics[width=0.46\textwidth]{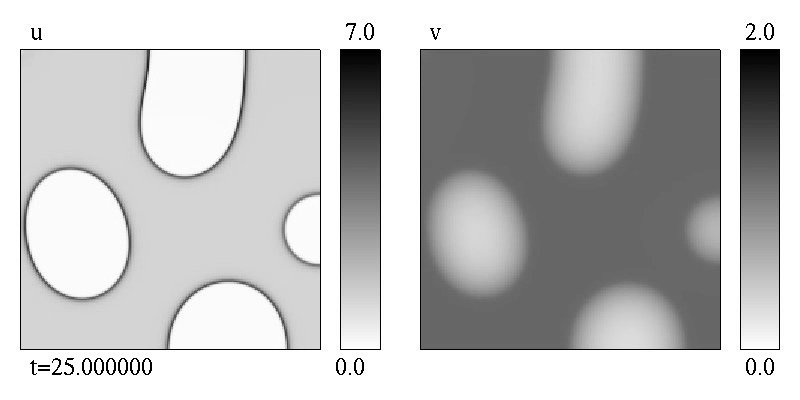}\\[-0.8em]
\includegraphics[width=0.46\textwidth]{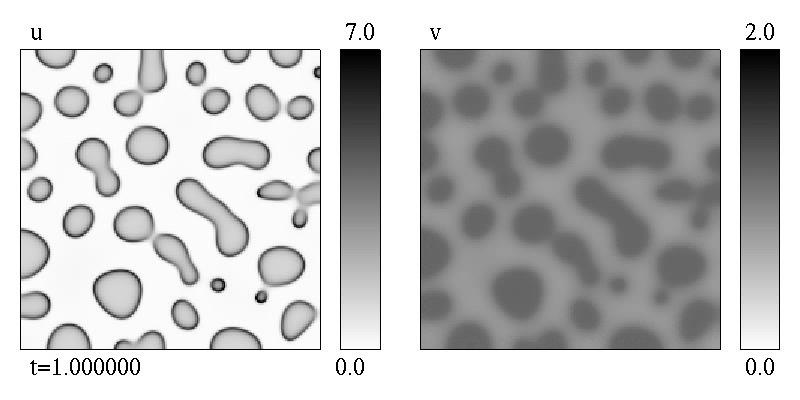} & & \includegraphics[width=0.46\textwidth]{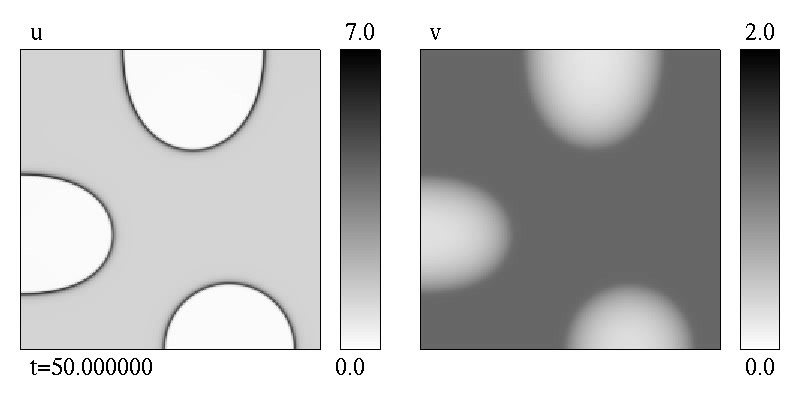}\\[-0.8em]
\includegraphics[width=0.46\textwidth]{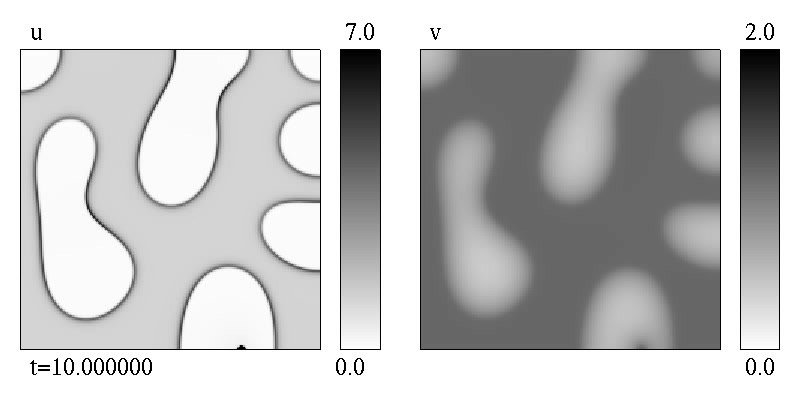} & & \includegraphics[width=0.46\textwidth]{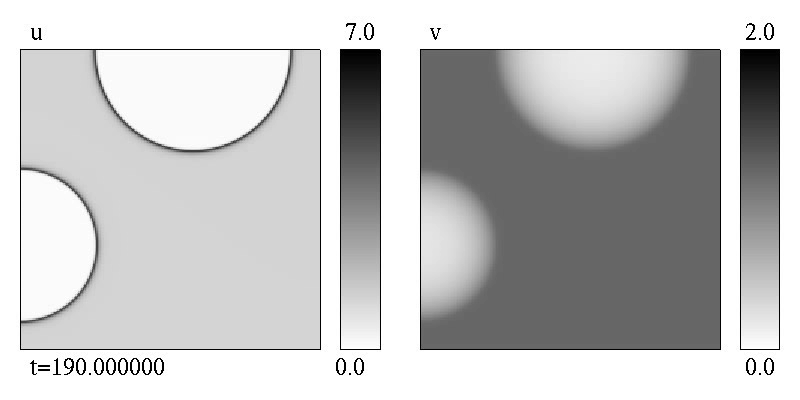}
\end{tabular} 
\caption{Aggregation and dispersal of cockroaches for $v^\ast = 1$ and $v^\# = 1.25$: snapshots of spatial distribution of cockroaches and pheromone at different times. The total initial density of $u$ as well as the concentration of the pheromone $v$ is set to one. Moreover, additional noise is added to $v$.}
\label{fig:03}
\end{figure} 

\medskip
However, if we set $v^\# = 1.25$, then the dispersal of cockroaches due to crowding kicks in. In this case cockroaches aggregate on the boundaries of rather uniformly shaped sites, later having form of disks, which divide the spatial domain into two regions, see Figure~\ref{fig:03}. The pheromone concentration is high outside the disks, therefore most cockroaches prefer these areas. Moreover, a low concentration of the aggregation pheromone inside the disks signalises a possible unfavourable environment, therefore cockroaches inside the disks move towards the places with higher concentration of the pheromone, hence they want to escape from the disks. Crowding outside disks, however, forces cockroaches to search for yet other places. 
Therefore, cockroaches get trapped to the narrow stripes which separate areas with high pheromone concentration from the areas in which the pheromone concentration is small.
Our newly discovered behaviour is intriguing and seems not to have been observed for cockroaches but apparently appeared in other similar biological models, such as chemotaxis models for E.coli or other bacteria (see e.g. \cite{aotani2010model}).

\medskip
Concerning the limit as $\eps\to 0$, our work is closely related to fast reaction limits towards a cross-diffusion-reaction system, which have been developed extensively in the last decade for many models in biology, chemistry, archeology, etc. Among those, let us mention the works \cite{eliavs2021singular,bothe2012cross,daus2020cross,desvillettes2015new,brocchieri2020evolution}. At first glance, our system \eqref{eq:01} looks similar to what have been considered in \cite{desvillettes2015new} or \cite{brocchieri2020evolution}. However, what differs our work from these two, and other related works in the literature, is the equation concerning the pheromone concentration $\veps$. More precisely, the equations of $\veps$ in \cite{desvillettes2015new,brocchieri2020evolution} possess a structure that implies the boundedness of $\veps$ in $L^{\infty}(\Omega\times(0,T))$ immediately by using the comparison principle. This in turn becomes an important ingredient to obtain other a-priori estimates, which ultimately allow to pass to the limit rigorously. This is not the case for our system at hand. Indeed, the regularity or the order of integrability of $\veps$ very much depends on that of $\uoneeps$ and $\utwoeps$. In fact, without further conditions on the functions $p$ and $q$, besides their smoothness, even the global existence of weak solutions to the system \eqref{eq:01} for fixed $\eps >0$ is unclear. In this paper, we resolve this issue by first using an improved duality method to obtain that $\uoneeps, \utwoeps$ belong to $L^{2+\delta_0}(\Omega\times(0,T))$ for some $\delta_0>0$. In dimensions one and two, utilising these estimates, the boundedness of $\veps$ follows from the regularisation of the heat operator. In higher dimensions $N\ge 3$, additional assumptions on either the boundedness of functions $p$ and $q$ or closeness the diffusion coefficients $d$ and $D$ are imposed. In case of the latter, we first use the duality method to obtain bounds of $\uoneeps$ and $\utwoeps$ in $L^r(\Omega\times(0,T))$ for some $r>(N+1)/2$, which, in combination with the regularisation of the heat operator, allow us to again obtain the boundedness of $\veps$ in $L^{\infty}(\Omega\times(0,T))$. It is worth mentioning that, while the structure of the equation for $\veps$ prevents us from obtaining the $L^{\infty}$-boundedness of $\veps$ in all dimensions (without further assumptions), it on other hand helps us to show that $\veps$ is bounded below away from zero (for any finite time), which in turn allows us to neglect the lower bounds assumption of $p$ and $q$ that were assumed in \cite[Assumption (12)]{desvillettes2015new}.

\medskip
Insufficient regularity for $\veps$, as well as for $\uoneeps$ and $\utwoeps$, is the main reason why the fast reaction limit of a system of equations for $\eps \to 0$ studied in \cite{Ei-2012} remained an open problem. However, our results apply to that model since it is a special case of \eqref{model3}. Therefore our paper completes and expands the works of Mimura and coauthors \cite{Ei-2012, Funaki-2012} in their modelling effort to reproduce social behaviour of cockroaches.

\medskip
\noindent\textbf{The rest of this paper is organised as follow}: In the next section, we discuss the numerical simulation of system \eqref{model3}. More precisely, bifurcation diagrams in both cases of without and with logistic growth are presented. We also give bifurcation diagram concerning the fast reaction limit $\eps\to 0$. The rigorous analysis of this fast reaction limit is investigated in Section \ref{analy}. Finally, the Appendix is devoted to two technical results.



\section{Numerical simulation of the model} \label{numer}

In this section, we numerically discuss the following model: 
\[
\left\{
\begin{aligned}
\partial_t u_1&=d\Delta u_1 + a_1(1-u_1-u_2)u_1-\frac 1 \varepsilon (q(v)u_1-p(v)u_2),\\
\partial_t u_2&=(d+D)\Delta u_2 + a_2(1-u_1-u_2)u_2+\frac 1 \varepsilon (q(v)u_1-p(v)u_2),\\
\partial_t v &= D_v \Delta v +\alpha (u_1+u_2)-\beta v.  
\end{aligned}
\right.
\]
In particular, we investigate the global structures of stationary solutions from the viewpoint of pattern formation by using a numerical bifurcation software AUTO (\cite{DOCDFKPSWZ}). 
Since the global structures of stationary solutions of the system with growth term ($a_1, a_2>0$) and of the model without growth term ($a_1=a_2=0$) somewhat differ, the two cases are discussed separately.  

Through this section, the functions of $q(v)$ and $p(v)$ are set as
\begin{equation}
\label{q and p}
\begin{aligned}
q(v)&=\frac{1}{2} \left( 1-\tanh(\gamma_1(v-v^*))\right)+\frac{1}{2} \left( 1+\tanh(\gamma_2(v-v^\#))\right)\\
& \text{  and} \\
p(v)&=1-q(v). 
\end{aligned}
\end{equation}
When the function $q(v)$ is divided into two part such as $q(v)=q_1(v)+q_2(v)$, where
$q_1(v)=\frac{1}{2} \left( 1-\tanh(\gamma_1(v-v^*))\right)$ and $q_2(v)=\frac{1}{2} \left( 1+\tanh(\gamma_2(v-v^\#))\right)$, 
note that the function $q_1(v)$ contributes to the activeness for low chemicals, on the other hand the function $q_2(v)$ contributes to the activeness for high chemicals. 
In addition, we use the following parameter setting unless otherwise stated: 
\begin{equation}
\label{param}
d=0.05,\, D_v=0.1,\, \alpha=1.0,\, \beta=1.0,\, \gamma_1=20.0,\, \gamma_2=20.0,\, v^*=1.0. 
\end{equation}

\subsection{System without growth term}
The case $a_1=a_2=0$ is studied in this subsection. 
Then, the system in one space dimension is given by 
\begin{equation}
\label{3RD1}
\left\{
\begin{aligned}
\partial_t u_1&=d \partial^2_x u_{1} -\frac 1 \varepsilon (q(v)u_1-p(v)u_2),\\
\partial_t u_2&=(d+D) \partial^2_x u_{2} +\frac 1 \varepsilon (q(v)u_1-p(v)u_2),\\
\partial_t v &= D_v  \partial^2_x v +\alpha (u_1+u_2)-\beta v,   
\end{aligned}
\quad t>0,\, 0<x<L. 
\right.
\end{equation}
We consider the above model under the Neumann boundary conditions
\[
\partial_x u_1=\partial_x u_2=\partial_x v=0 \quad \text{on } x=0,\, L
\]
and the initial conditions
\[
u_1(0,x)=u_{10}(x),\, u_2(0,x)=u_{20}(x), \, v(0,x)=v_0(x),\quad 0<x<L. 
\]
Adding the equations for $u_1$ and for $u_2$, we have 
\(
\partial_t (u_1+u_2)=\partial_x^2(du_1+(d+D)u_2). 
\)
Integrating the equation over $(0,L)$, and using integration by parts and the Neumann boundary conditions, 
\(  
\partial_t \int_0^L (u_1+u_2)dx = 0
\)
is obtained. 
This means that \eqref{3RD1} is a mass conservation system. 
When we define the spatial average of initial mass $u_{10}+u_{20}$ as $M:=\frac{1}{L}\int_0^L (u_{10}+u_{20})dx$, the constant steady state is expressed by 
$$
(\overline{u}_1,\overline{u}_2,\overline{v})=\left(\frac{p(\frac{\alpha}{\beta}M)}{p(\frac{\alpha}{\beta}M)+q(\frac{\alpha}{\beta}M)}M,\frac{q(\frac{\alpha}{\beta}M)}{p(\frac{\alpha}{\beta}M)+q(\frac{\alpha}{\beta}M)}M,  \frac{\alpha}{\beta}M \right), 
$$ 
which is parametrized by $M$. 
This constant steady state can be destabilized depending on the value of $M$ via Turing-like instability. 
In order to explain this, \eqref{3RD1} is linearized around the constant steady state. Then, the linearized matrix for $n$-Fourier cosine mode is obtained as follows: 
\[
A_n=
\begin{pmatrix}
-d\left(\frac{n\pi}{L}\right)^2-\frac 1 \varepsilon q(\overline{v}) & \frac 1 \varepsilon p(\overline{v}) & -\frac1 \varepsilon (q'(\overline{v})\overline{u}_1-p'(\overline{v})\overline{u}_2)\\
\frac1 \varepsilon q(\overline{v}) & -(d+D)\left(\frac{n\pi}{L}\right)^2-\frac 1 \varepsilon p(\overline{v}) & \frac 1 \varepsilon (q'(\overline{v})\overline{u}_1-p'(\overline{v})\overline{u}_2)\\
\alpha & \alpha & -D_v\left(\frac{n\pi}{L}\right)^2-\beta
\end{pmatrix}
\]
One can see from the linearized matrices and linear stability analysis that the neutral stability curve for $n$ ($n=1, 2, 3, \cdots$) is given by $\text{det}A_n=0$. For \eqref{q and p}, the stable and unstable regions of the constant steady state $(\overline{u}_1,\overline{u}_2,\overline{v})$, which depends on $M$, and the neutral stability curves in $(M,D)$-plane are shown in Figure \ref{NSC-3RD}. 
\begin{figure}[htbp]
\begin{center}
\includegraphics[width=80mm]{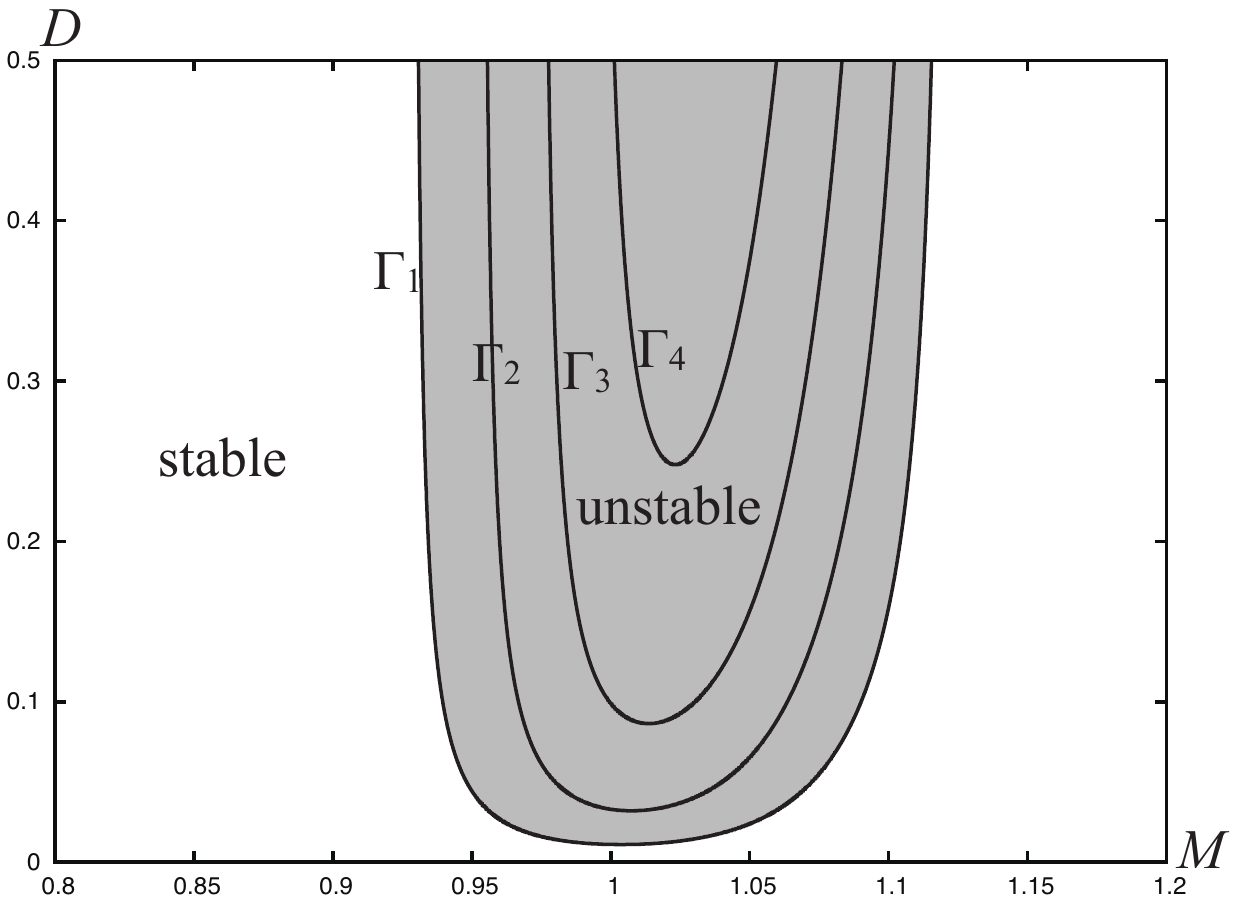}
\caption{The stable and unstable regions of the constant steady state $(\overline{u}_1,\overline{u}_2,\overline{v})$ and the neutral stability curves which defined by $\Gamma_n=\{(M,D)\in \mathbb{R}_+^2 : \text{det}A_n=0\}$ for each $n$. 
The parameter values are $v^\#=1.25$, $\varepsilon=0.001$ and $L=1.0$.}
\label{NSC-3RD}
\end{center}
\end{figure}
Assume that a suitable $D$ value is fixed. 
One can see from the figure that when the value of $M$ is large or small, the constant steady state is stable, on the other hand, when the value of $M$ is moderate, it is destabilized, so that we can expect that non-constant stationary solutions emerge in the range. 
Based on the linear stability analysis, we numerically compute global structures of stationary solutions with the aid of a numerical bifurcation software AUTO. 
Figure \ref{exa_fig1} shows the global structure of stationary solution when $D=0.15$ and $M$ is a bifurcation parameter. 
\begin{figure}[htbp]
\begin{center}
\includegraphics[width=100mm]{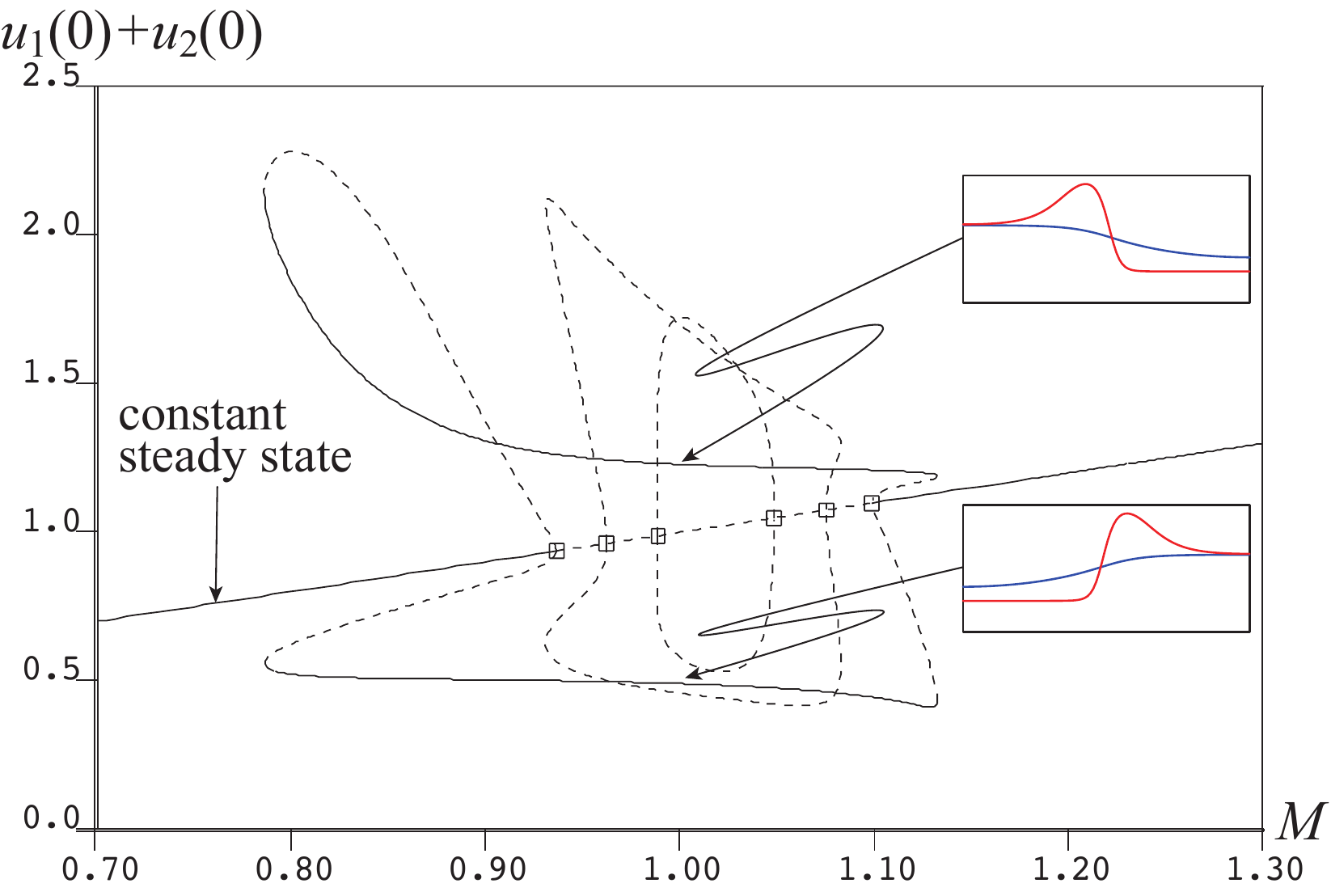}
\caption{Global structure of stationary solutions for \eqref{3RD1}. The horizontal and the vertical axes denote a bifurcation parameter $M$ and the value of $u_1+u_2$ at $x=0$, respectively. The solid and dashed curves respectively means stable solution branch and unstable one. The square $\square$ represents a pitchfork bifurcation point. Two profiles of stationary solution are displayed in the figure, where the red and blue curves means $u_1+u_2$ and $v$, respectively. Each profile corresponds to the stationary solution with $M=1$ on the stable branches. }
\label{exa_fig1}
\end{center}
\end{figure}
As the value of $M$ increases, the constant steady state loses its stability around $M=0.936341$ and a subcritical pitchfork bifurcation occurs. Although a pair of non-constant stationary solution branches unstably emerges from the pitchfork bifurcation point, it recovers the stability via fold bifurcations around $M=0.78621$ then stably non-constant stationary solutions are observed for $0.78621<M<1.13259$. 
These stable solution branches undergo fold bifurcations again around $M=1.13259$ and the unstable branches connect with the pitchfork bifurcation point around $M=1.09886$. 
For large values of $M$, only the constant solution is stable. 
Here, we note that the upper solution branch and the lower one can be identified because if $(u_1(x),u_2(x),v(x))$ is a stationary solution on the upper branch, $(u_1(L-x),u_2(L-x),v(L-x))$ corresponds to the lower one, so that the stability of the two solutions is also same(see also the solution profiles in Figure \ref{exa_fig1}). 
For $0.936341<M<1.09886$, non-constant stationary solution branches with 2-Fourier mode and with 3-Fourier mode are also observed. However, these are always unstable. 
The transition of stably non-constant stationary solution profiles is shown in Figure \ref{transition} when the value $M$ varies. 
\begin{figure}[htbp]
\begin{center}
\begin{tabular}{cc}
\includegraphics[width=55mm]{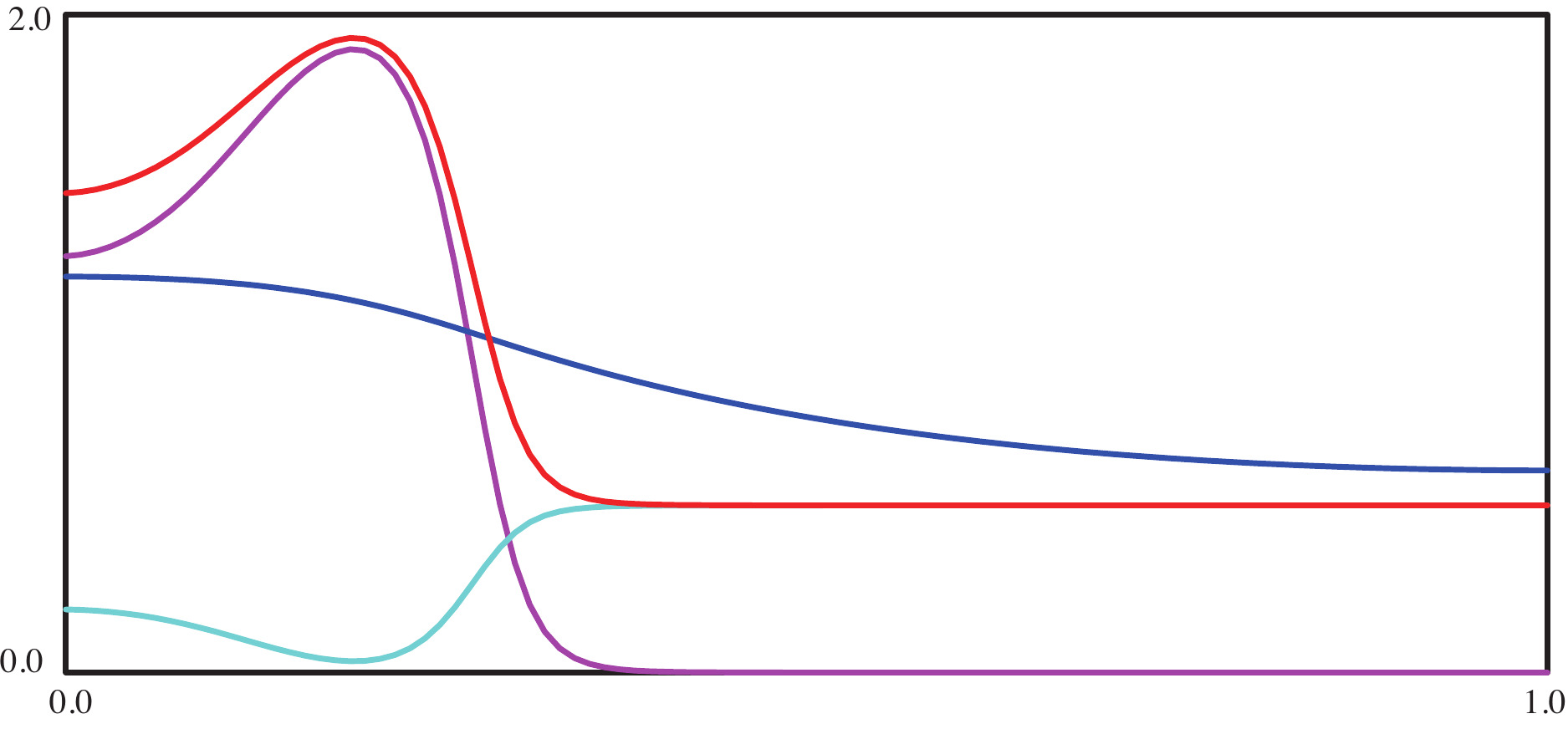}&
\includegraphics[width=55mm]{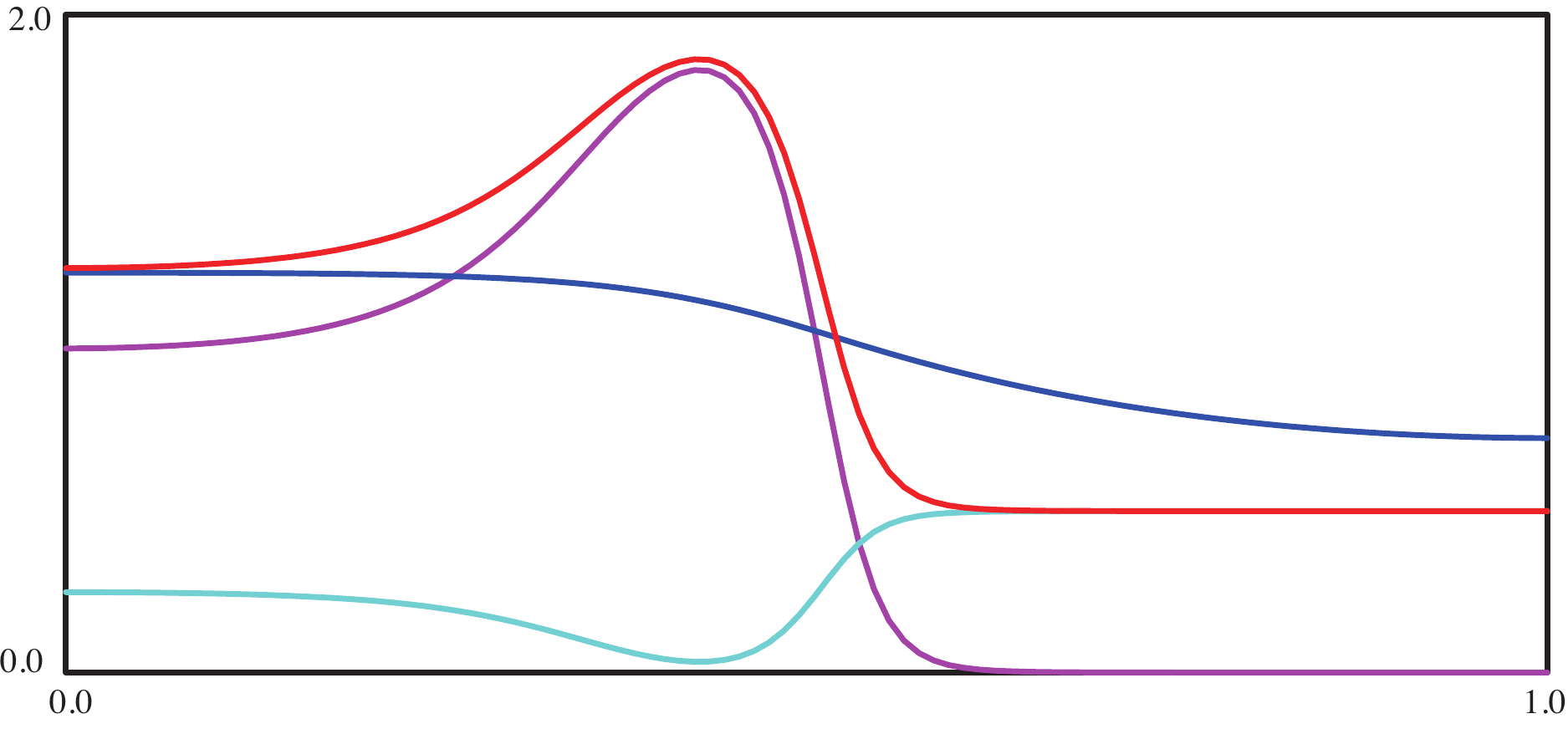}\\
$M=0.85$ & $M=1.0$\\
\includegraphics[width=55mm]{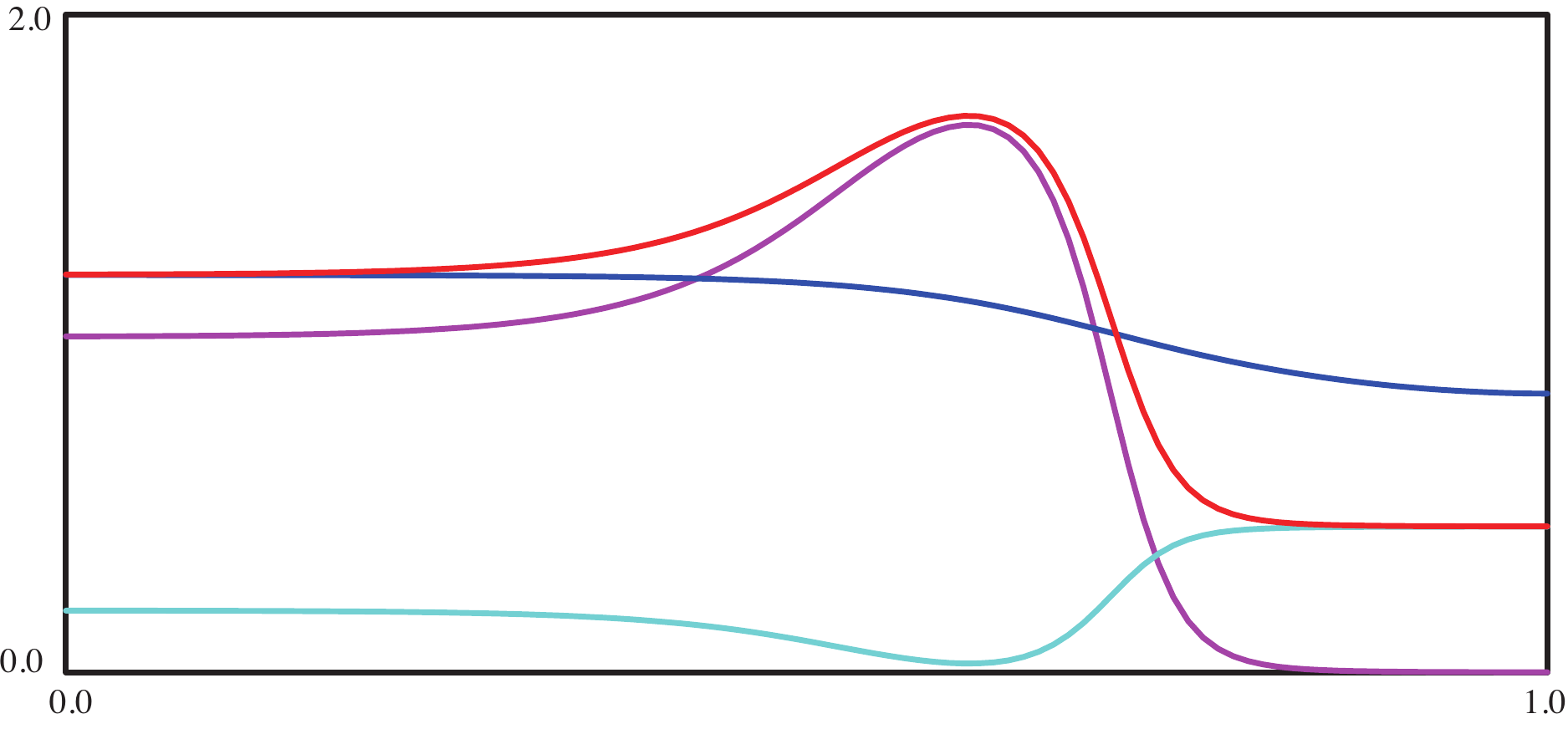} & \\
$M=1.1$ & 
\end{tabular}
\caption{Stably stationary solution profiles according to the value $M$. The meaning of colors is as follows: (red) $u_1+u_2$, (blue) $v$, (magenta) $u_1$ and (cyan) $u_2$. The parameter values are the same as in Figure \ref{exa_fig1}. }
\label{transition}
\end{center}
\end{figure}
One can see that the aggregation region expands as the average mass $M$ increases. 
Interestingly, the peak of the total population density $u_1+u_2$ is not the point where the pheromone concentration $v$ is a maximum. 
One can understand this phenomenon from the fact that each individual prefers to the moderate pheromone concentration because they move actively if it is low or high.  

In this section, the function $q(v)$ is set as \eqref{q and p}.
If the value of $v^\#$ is large enough in this expression, the function $q(v)$ is regarded as $q(v)=q_1(v)=\frac{1}{2} \left( 1-\tanh(\gamma_1(v-v^*))\right)$. This case is discussed in \cite{Funaki-2012}. A new mathematical modeling perspective in this paper is that if the population density is too high, each individual escapes from the overcrowded region. Here, we assume that each individual senses the population density through the pheromone concentration. Therefore, we add the latter term $q_2(v)=\frac{1}{2} \left( 1+\tanh(\gamma_2(v-v^\#))\right)$ to avoid an overcrowded situation which seems unrealistic as a biological situation, and the parameter $v^\#$ represents a kind of threshold such that each individual recognizes whether or not the current position is overcrowded. 
That is, if the value $v^\#$ is small, each individual is sensitive to the high population density and then tends to move to the moderate density region. We next investigate the influence of $v^\#$ on the pattern formation. 
Figure \ref{fig_ab} shows the transition of global structures of stationary solutions when the value of $v^\#$ is varied. 
\begin{figure}[htbp]
\begin{center}
\begin{tabular}{cc}
\includegraphics[width=65mm]{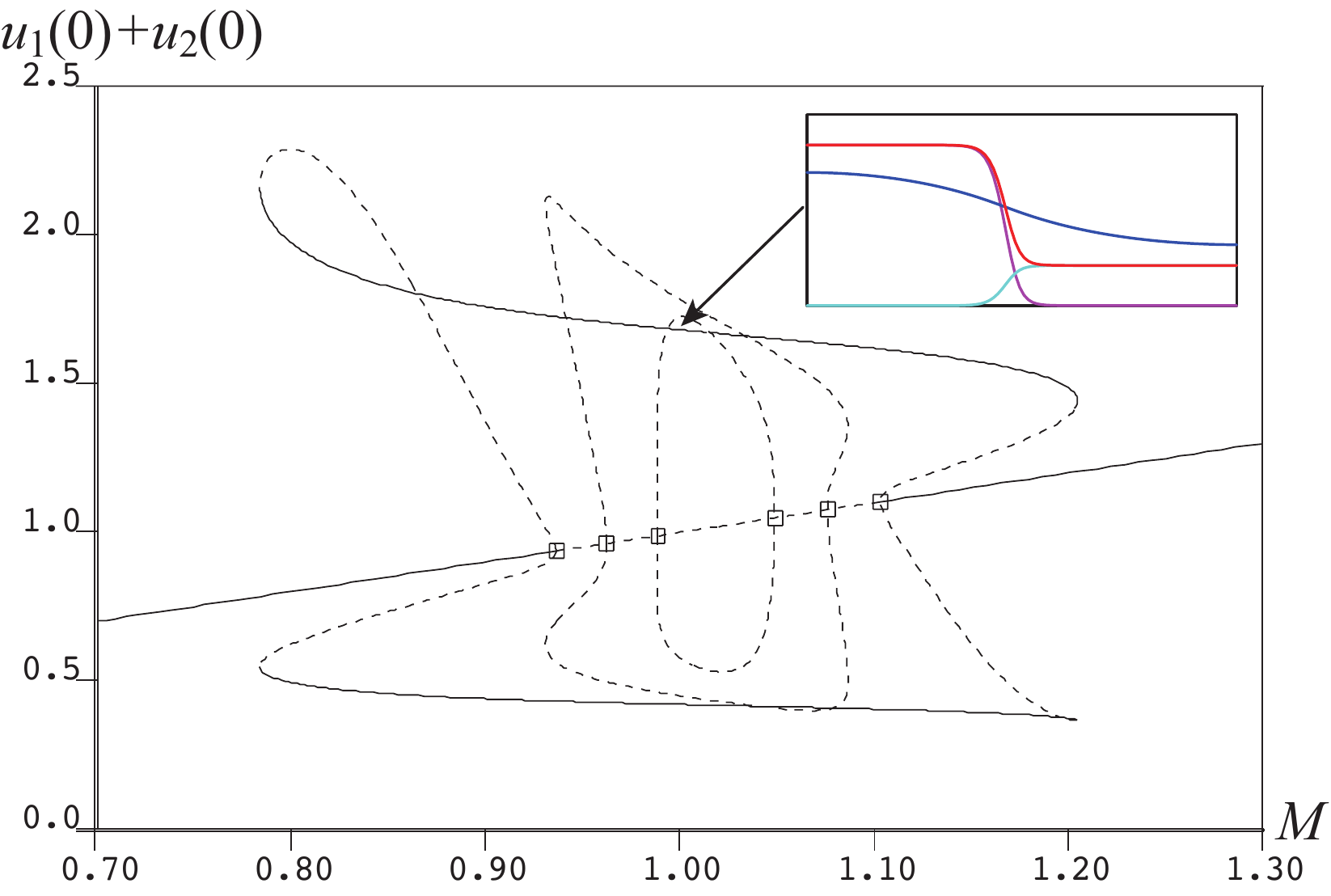} &
\includegraphics[width=65mm]{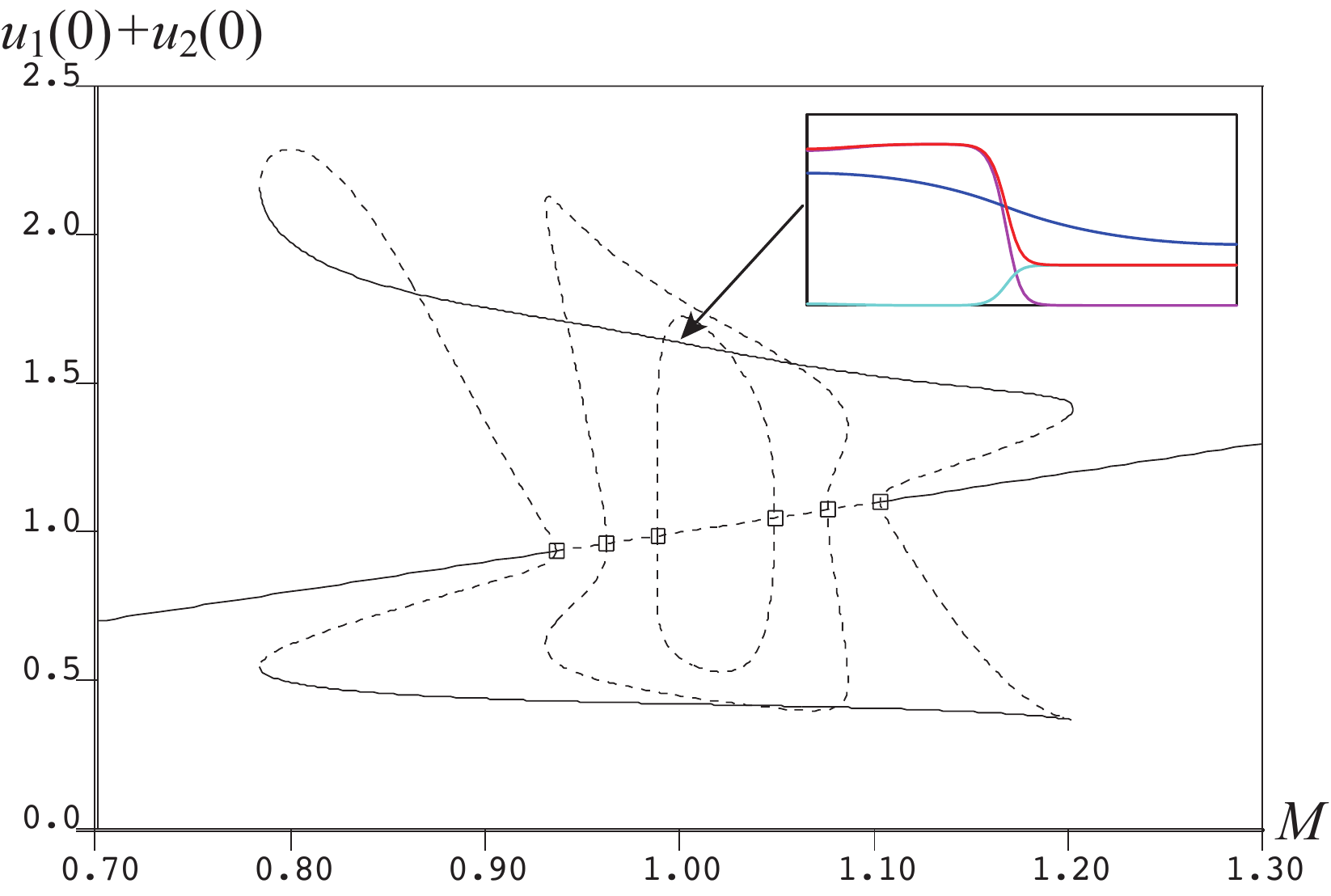}\\
(a) & (b)\\
\includegraphics[width=65mm]{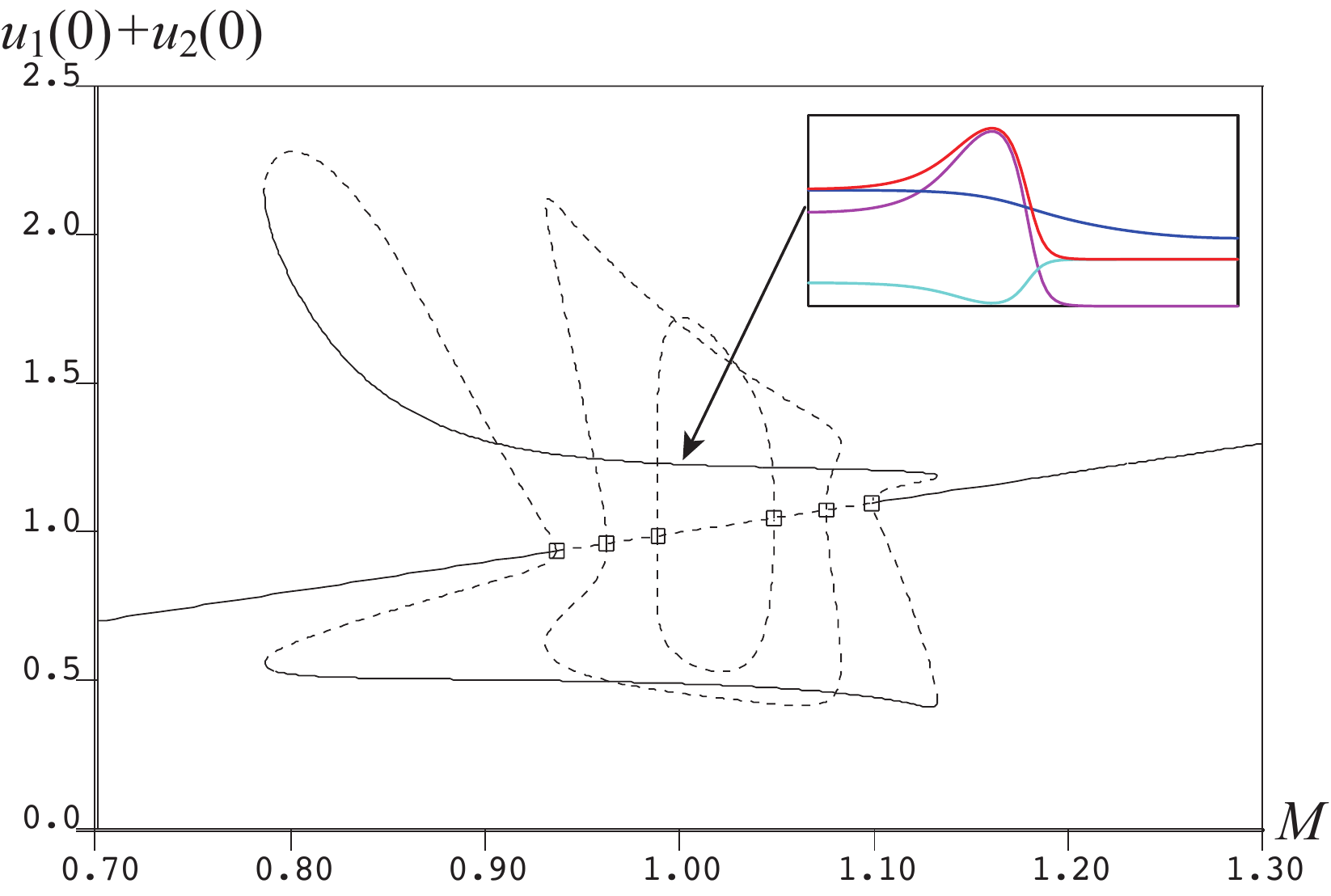} &
\includegraphics[width=65mm]{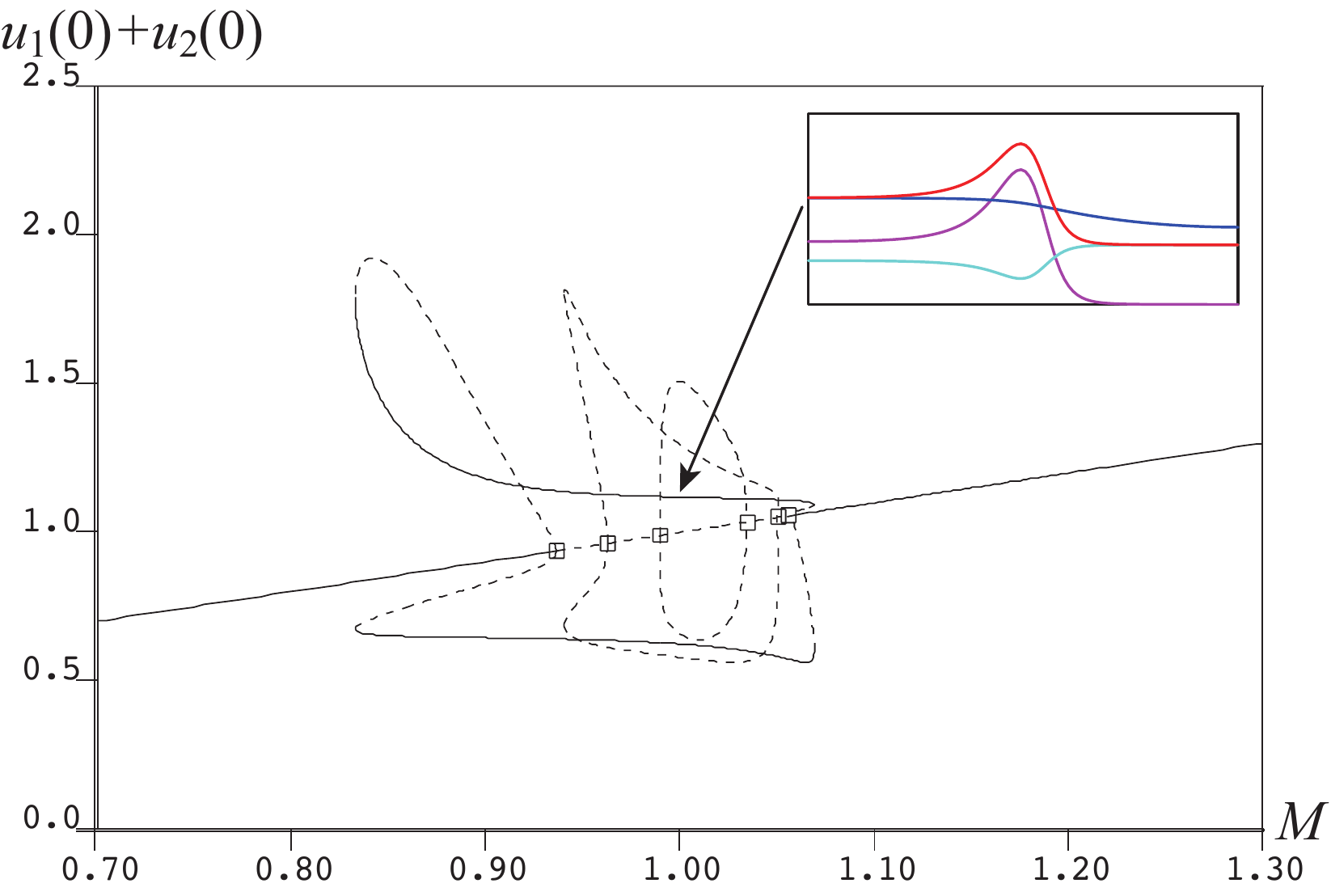}\\
(c) & (d)\\
\includegraphics[width=65mm]{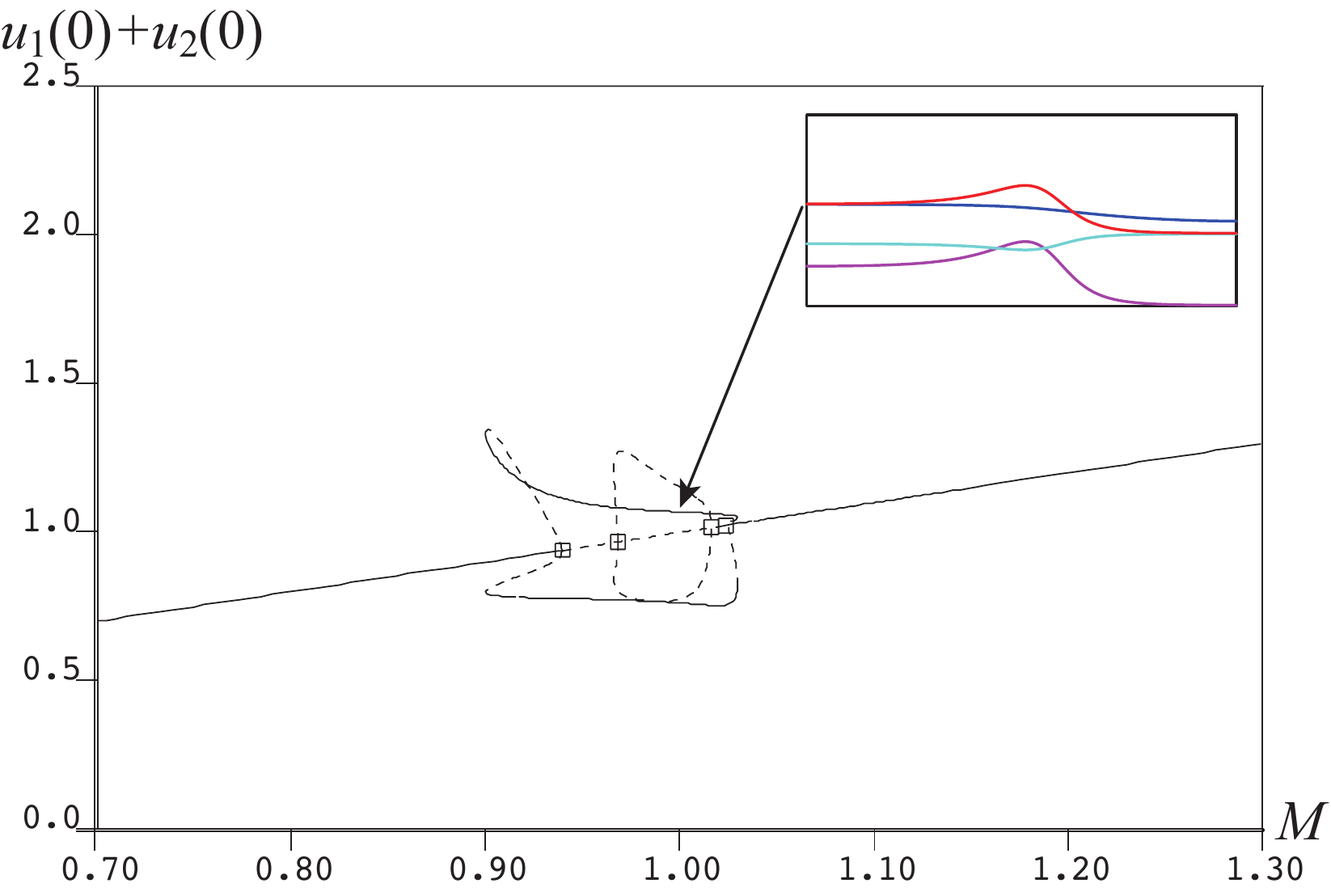} &
\includegraphics[width=65mm]{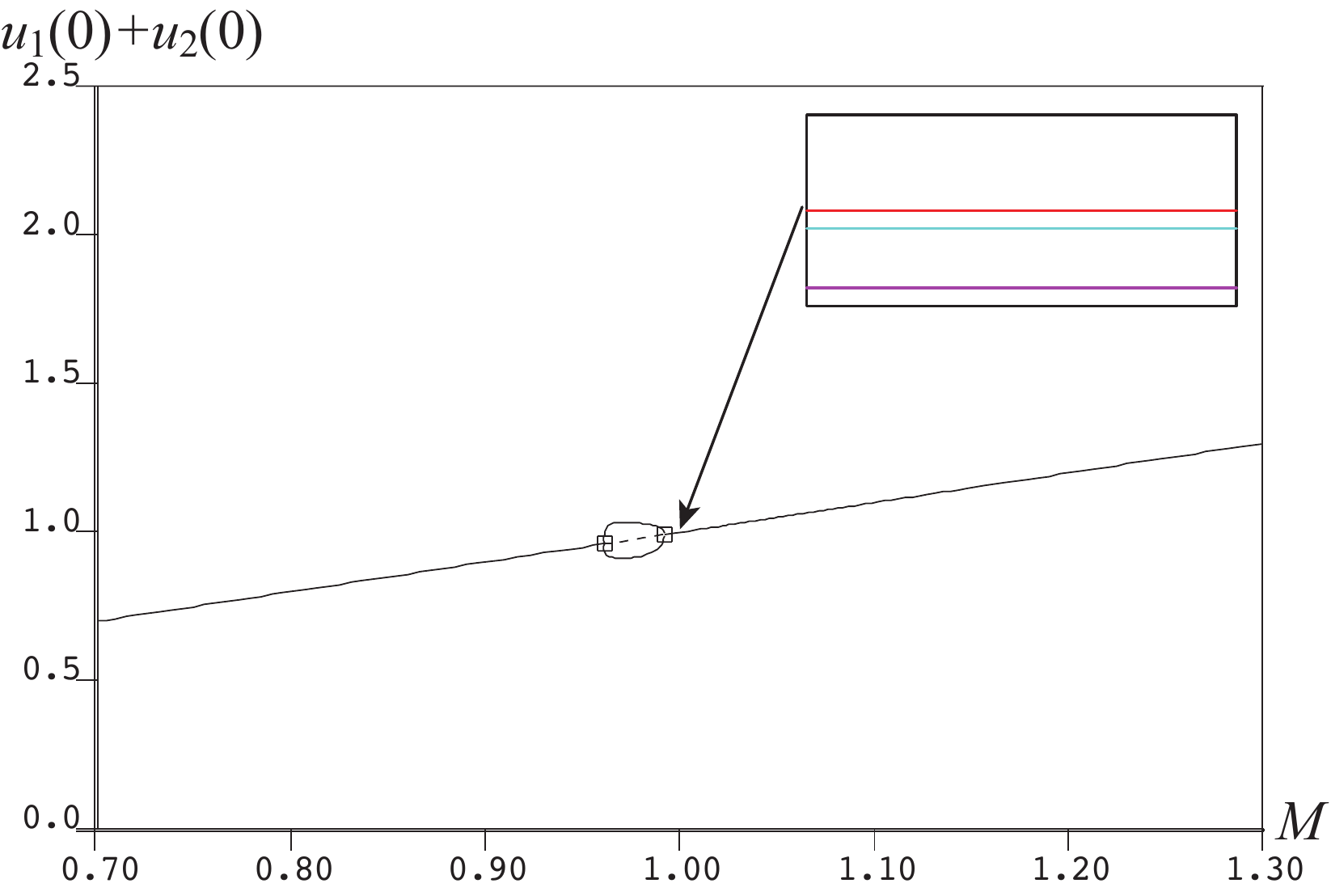}\\
(e) & (f)
\end{tabular}
\caption{Global structure of stationary solutions for \eqref{3RD1}. 
The parameter values are the same as in Figure \ref{exa_fig1} except for $v^\#$. 
(a) $q(v)=\frac{1}{2} \left( 1-\tanh(\gamma_1(v-v^*))\right)$.
(b) $v^\#=1.5$.
(c) $v^\#=1.25$
(d) $v^\#=1.125$. 
(e) $v^\#=1.06$. 
(f) $v^\#=1.02$. 
}
\label{fig_ab}
\end{center}
\end{figure}
Figure \ref{fig_ab} (a) is the global structure for $q(v)=\frac{1}{2} \left( 1-\tanh(\gamma_1(v-v^*))\right)$, which is discussed in \cite{Funaki-2012}. This case does not consider the repulsion for the high population density. 
Figures \ref{fig_ab} (b)-(f) are the cases with avoiding effect from crowded region. 
Note that the term $q_1(v)$ including $v^*$ contributes to generating an aggregation via an instability of the constant steady state (see also \cite{Funaki-2012}), on the other hand, the term $q_2(v)$ including $v^\#$ means preventing an overcrowded aggregation. 
Therefore, the global structure shrinks as the value of $v^\#$ becomes small, and eventually when $v^\#=v^*$, $q(v)=1$ and $p(v)=0$, that is, no pattern emerges because all $u_1$ changes into $u_2$ which is governed by a diffusion equation. 
In other words, if the value of $v^\#$ approaches $v^*$, then the ratio of $u_2$ which diffuses faster increases. Consequently, it is hard to form self-organized aggregation.  
Actually, in Figure \ref{fig_ab}, a stably stationary solution with $M=1$ is displayed in each figure. As $v^\#$ is close to $v^*$, one can see that the ratio of $u_2$ (cyan) (resp. $u_1$ (magenta)) relatively increases (resp. decreases). 



\subsection{System with growth term}
In this subsection, we discuss the system with growth term
\begin{equation}
\label{3RD2}
\left\{
\begin{aligned}
\partial_t u_1&=d\partial^2_x u_1 + a_1(1-u_1-u_2)u_1-\frac 1 \varepsilon (q(v)u_1-p(v)u_2),\\
\partial_t u_2&=(d+D)\partial^2_x u_2 + a_2(1-u_1-u_2)u_2+\frac 1 \varepsilon (q(v)u_1-p(v)u_2),\\
\partial_t v &= D_v \partial^2_x v +\alpha (u_1+u_2)-\beta v.  
\end{aligned}
\right.
\end{equation}
in one-space dimension under the Neumann boundary conditions. 
Unlike the system without growth term \eqref{3RD1}, \eqref{3RD2} possesses a unique constant steady state given by 
$(\overline{u}_1,\overline{u}_2,\overline{v})=(\frac{p(\alpha/\beta)}{p(\alpha/\beta)+q(\alpha/\beta)},\frac{q(\alpha/\beta)}{p(\alpha/\beta)+q(\alpha/\beta)},\frac{\alpha}{\beta})$. 
 Linearizing \eqref{3RD2} around the constant steady state, the linearized matrix for $n$-Fourier cosine mode is obtained as follows: 
\[
B_n=
\begin{pmatrix}
b_{11} & -a_1 \overline{u}_1+\frac 1 \varepsilon p(\overline{v}) & -\frac1 \varepsilon (q'(\overline{v})\overline{u}_1-p'(\overline{v})\overline{u}_2)\\
-a_2\overline{u}_2+\frac1 \varepsilon q(\overline{v}) & b_{22} & \frac 1 \varepsilon (q'(\overline{v})\overline{u}_1-p'(\overline{v})\overline{u}_2)\\
\alpha & \alpha & -D_v\left(\frac{n\pi}{L}\right)^2-\beta
\end{pmatrix}, 
\]
where $b_{11}=-d\left(\frac{n\pi}{L}\right)^2-a_1\overline{u}_1-\frac 1 \varepsilon q(\overline{v})$ and $b_{22}=-(d+D)\left(\frac{n\pi}{L}\right)^2-a_2\overline{u}_2-\frac 1 \varepsilon p(\overline{v})$. 
It follows from linear stability analysis that the neutral stability curve for $n$ ($n=1,2,3,\cdots$) is given by $\text{det}B_n=0$. 
Here, we assume that the growth rates $a_1$ and $a_2$ are equal for simplicity, namely $r=a_1=a_2$. 
Then, the stable and unstable regions of the constant steady state $(\overline{u}_1,\overline{u}_2,\overline{v})$, and the neutral stability curves in $(r,D)$-plane are shown in Figure \ref{NSC-3RDwG}. 
\begin{figure}[htbp]
\begin{center}
\includegraphics[width=80mm]{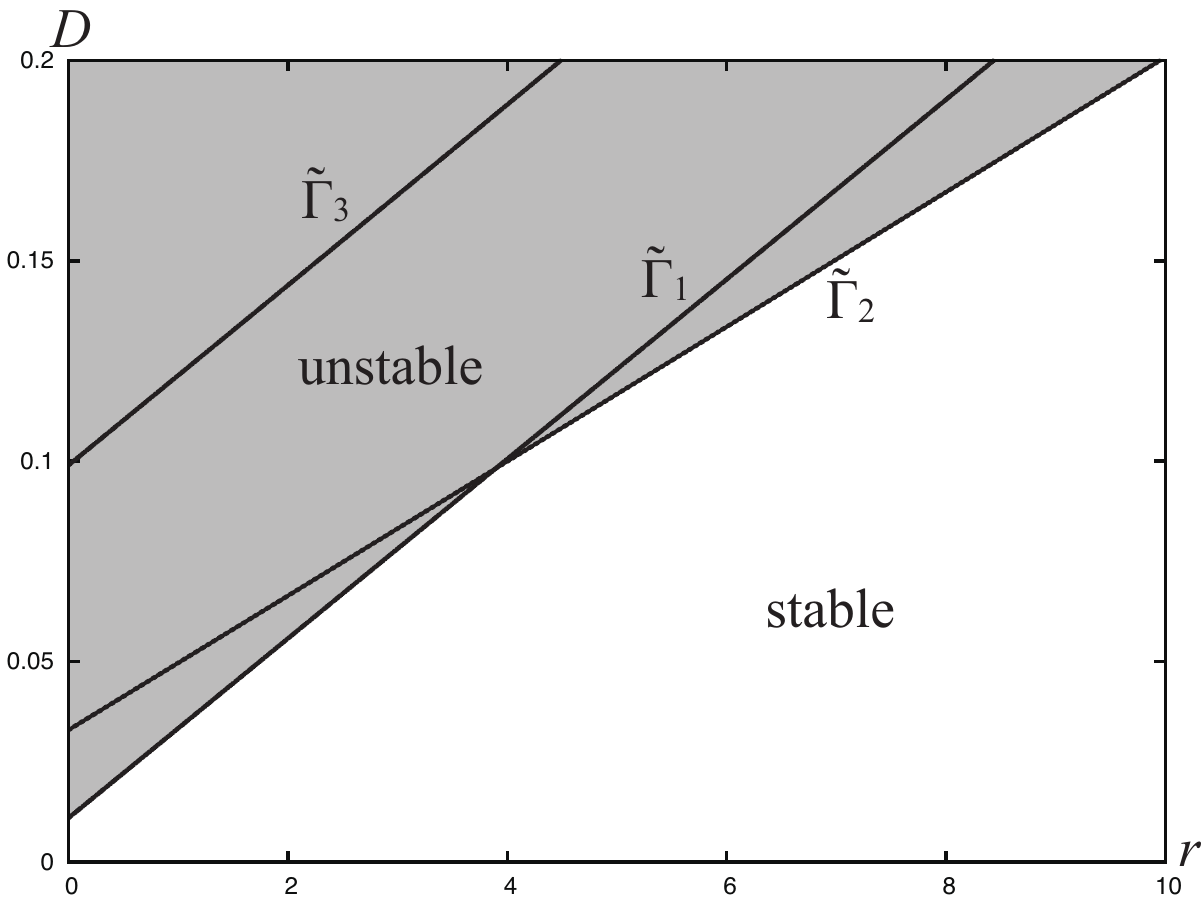}
\caption{The stable and unstable regions of the constant steady state $(\overline{u}_1,\overline{u}_2,\overline{v})$ and the neutral stability curves which defined by $\tilde{\Gamma}_n=\{(r,D)\in \mathbb{R}_+^2 : \text{det}B_n=0\}$ for each $n$. Here, assume that the growth rate $a_1$ and $a_2$ are equal, that is $r=a_1=a_2$. 
The parameter values are $v^\#=1.25$, $\varepsilon=0.001$ and $L=1$.}
\label{NSC-3RDwG}
\end{center}
\end{figure}
One can see from the figure that the constant steady state is stable for suitably large $r$ values or small $D$ values, on the other hand, it is unstable for suitably small $r$ values or large $D$ values. 
Thus, it is expected that there are non-constant stationary solutions in the unstable region due to the instability of the constant steady state. 
Similarly to the previous subsection, a numerical bifurcation software AUTO helps to unveil global structures of stationary solutions for \eqref{3RD2}. 
\begin{figure}[htbp]
\begin{center}
\includegraphics[width=100mm]{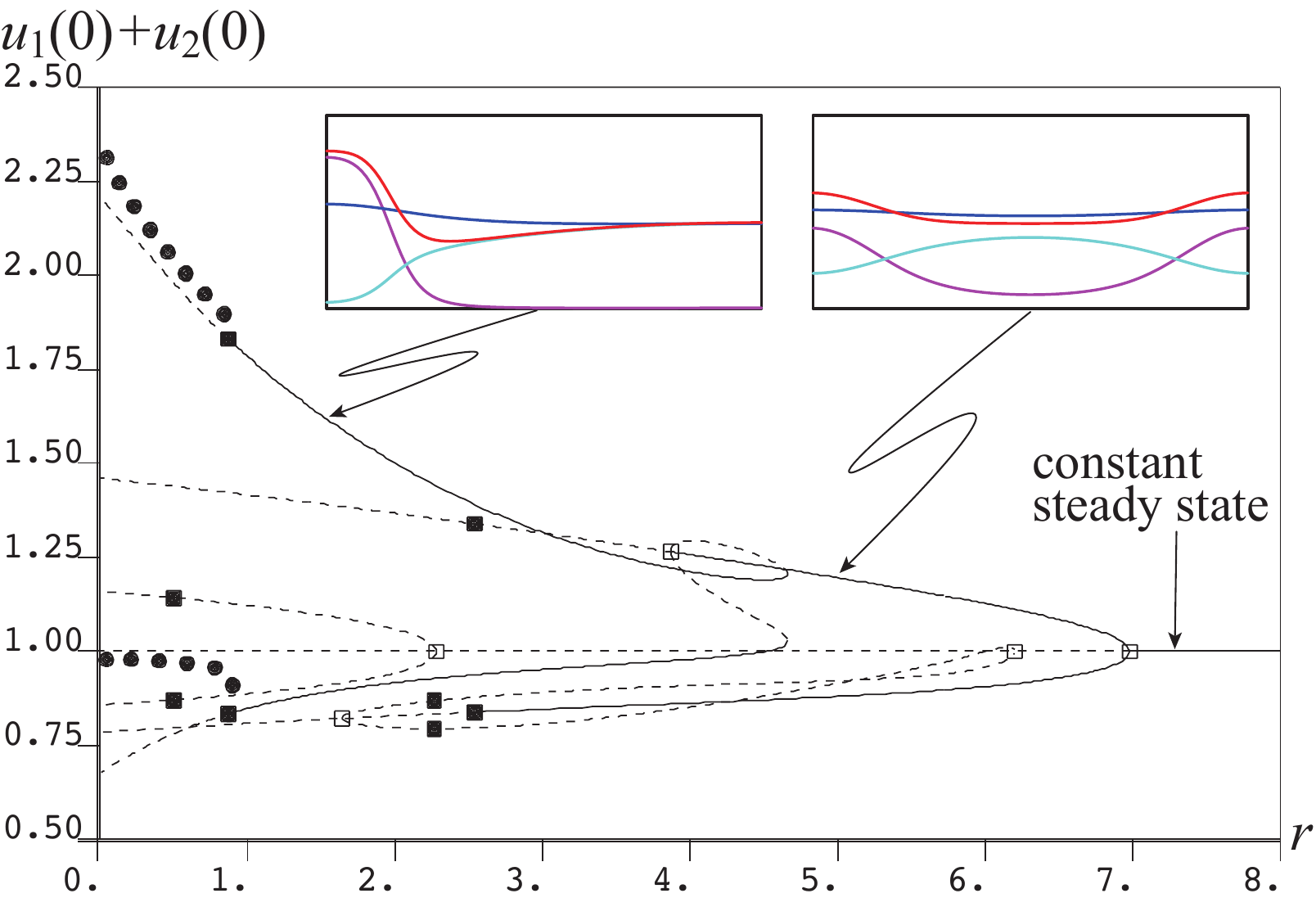}
\caption{Global structure of stationary solutions for \eqref{3RD2} when $D=0.15$. The horizontal and the vertical axes denote a bifurcation parameter $r=a_1=a_2$ and the value of $u_1+u_2$ at $x=0$, respectively. The solid and dashed curves respectively means stable solution branch and unstable one. The symbols $\square$ and $\blacksquare$ represent a pitchfork bifurcation point and a Hopf bifurcation point, respectively. The black circle $\bullet$ means a stable periodic solution branch. The other parameter values are the same as in Figure \ref{NSC-3RDwG}. }
\label{exa_fig2}
\end{center}
\end{figure}
Figure \ref{exa_fig2} exhibits a global bifurcation diagram of \eqref{3RD2} for $D=0.15$. 
One can see from the linear stability analysis in Figure \ref{NSC-3RDwG} that 2-mode, 1-mode and 3-mode are destabilized successively for $D=0.15$ as the value of $r$ gradually decreases. 
Figure \ref{exa_fig2} corresponds to the global bifurcation diagram for $D = 0.15$, where stably non-constant stationary solutions primarily bifurcate from the constant steady state via a supercritical pitchfork bifurcation when $r$ decreases. 
Therefore, stably 1-mode and 2-mode stationary solutions are observed for moderate $r$ values (see also the solution profiles in Figure \ref{exa_fig2}). 
However, there is no stable stationary solution for $r<0.875497$. 
Instead, stable periodic solution branches (black circles) are observed through a supercritical Hopf bifurcation at $r=0.875497$. 
\begin{figure}[htbp]
\begin{center}
\begin{tabular}{cc}
\includegraphics[width=65mm]{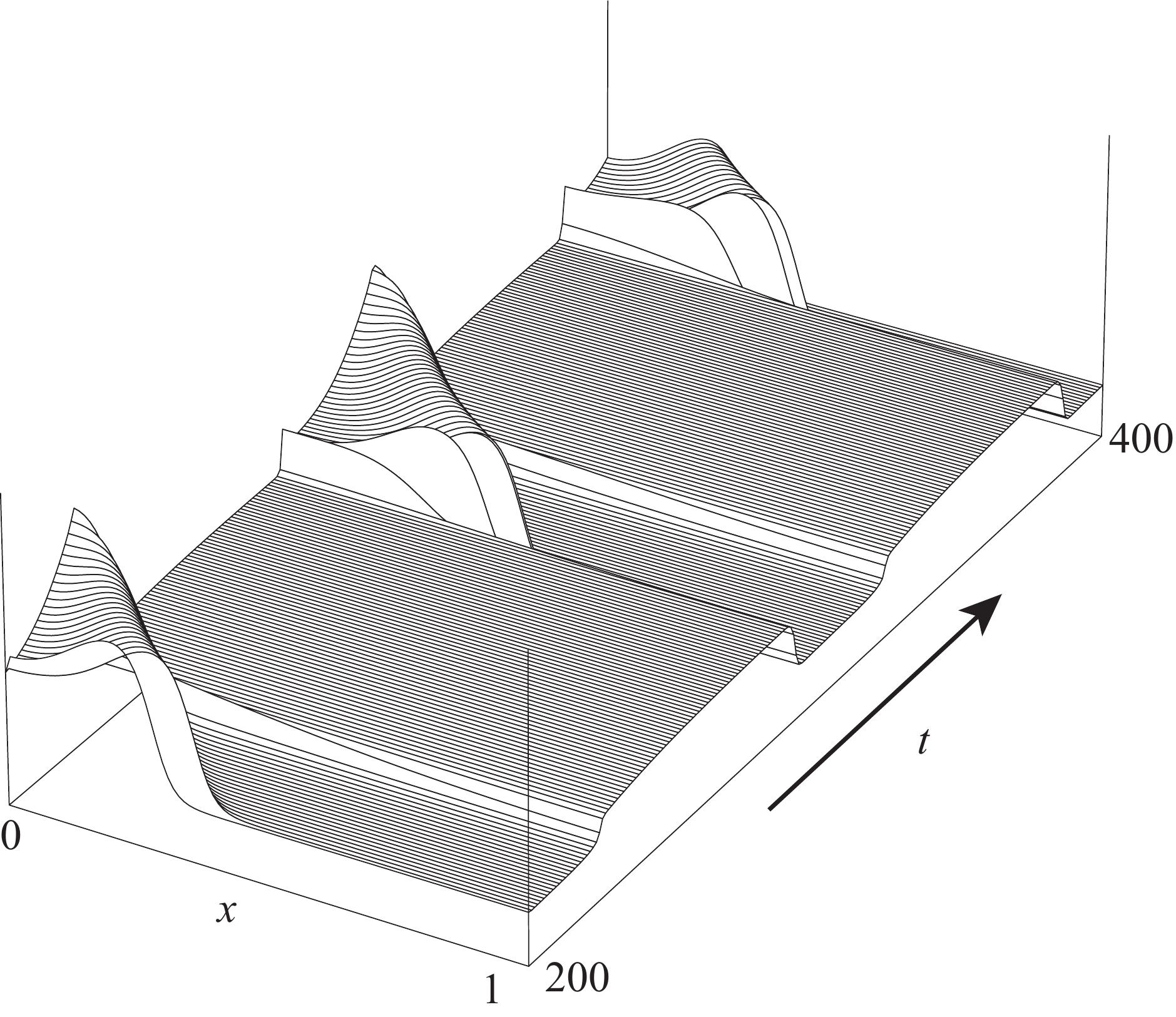} &
\includegraphics[width=65mm]{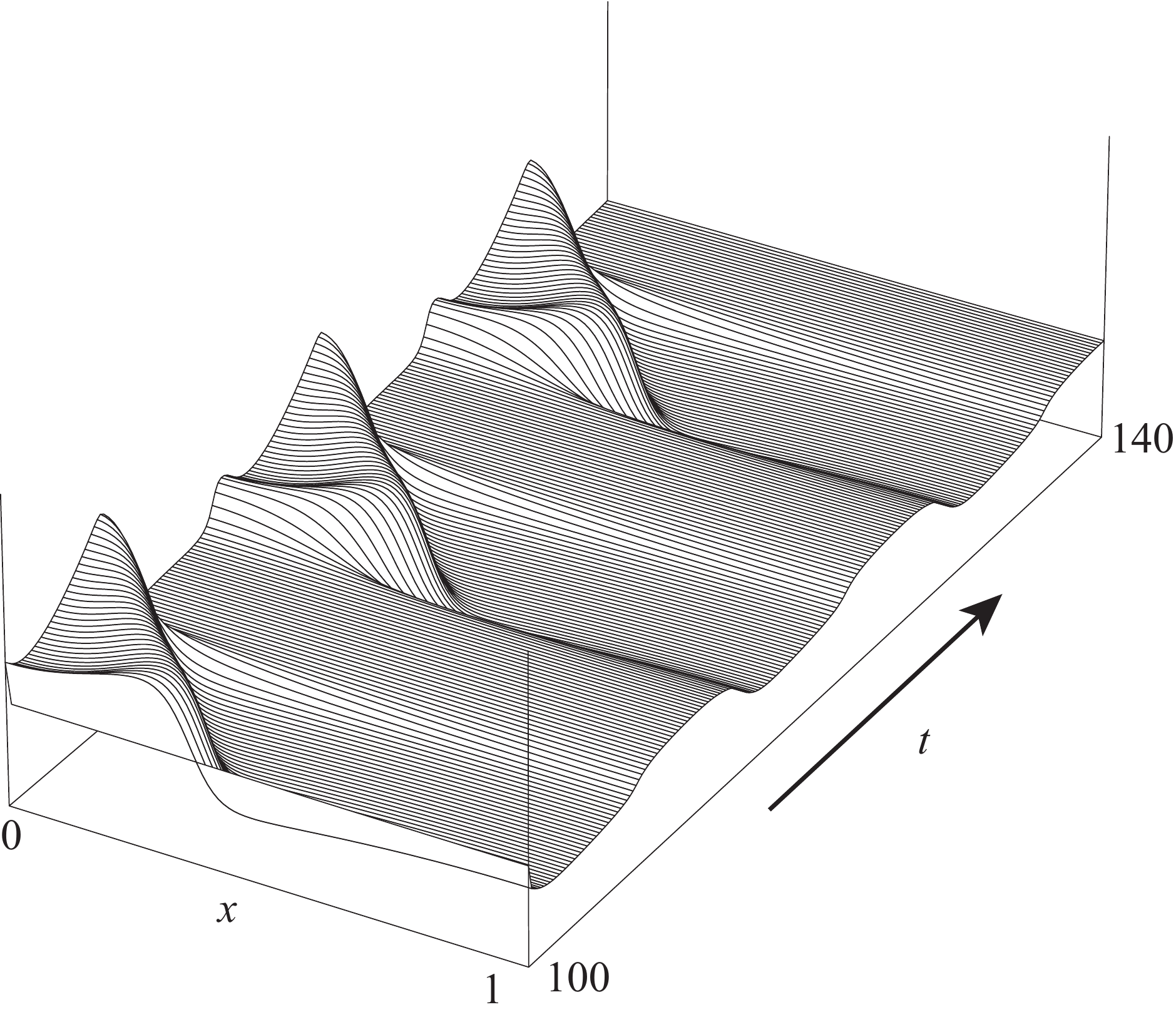}\\
$r=0.05$ & $r=0.5$
\end{tabular}
\caption{Stable periodic solutions arising in Figure \ref{exa_fig2}. Spatio-temporal dynamics of $u_1+u_2$ is displayed. 
}
\label{fig_periodic1}
\end{center}
\end{figure}
A feature of this type of periodic solutions is that the smaller the value $r$ is, the longer a period of the solution becomes, as shown in Figure \ref{fig_periodic1}. 
When $r$ value is quite small, a remarkable periodic solution so called a relaxation oscillation appears. (See also \cite{Ei-2012}.)
Actually, the relaxation oscillation which consists of a fast dynamics and a slow dynamics uses the global bifurcation diagram for \eqref{3RD1} ($r=0$). 
\begin{figure}[htbp]
\begin{center}
\includegraphics[width=90mm]{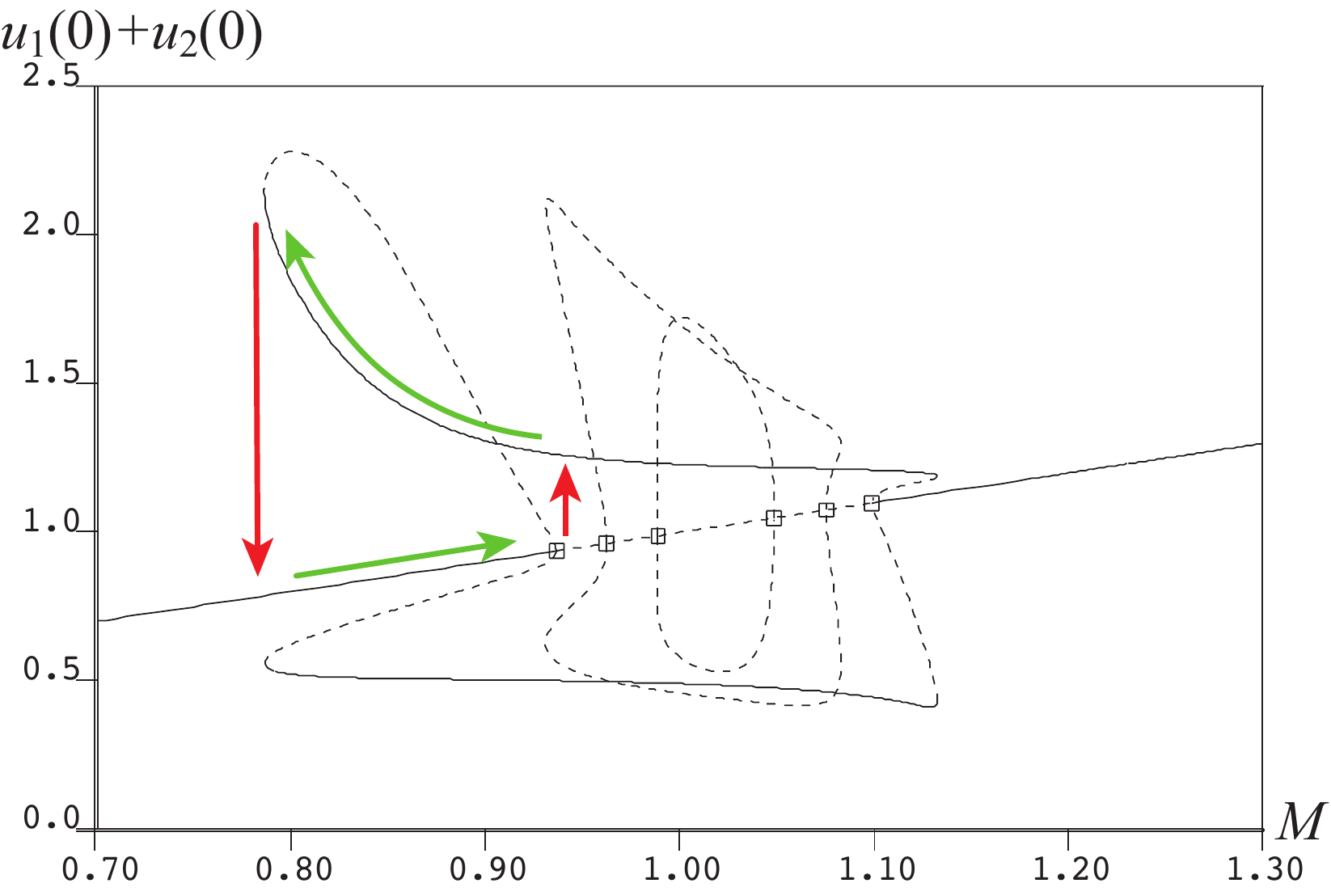}
\caption{The mechanism of relaxation oscillation using the bifurcation structure for $r=0$. The red arrow and the green one respectively mean a flow of fast dynamics and a flow of slow one. }
\label{RO}
\end{center}
\end{figure}
Figure \ref{RO} indicates a conceptual trajectory on the relaxation oscillation. 
For the detailed dynamics of relaxation oscillation, we refer to \cite{Ei-2012}. 

Next, we investigate the change of global structure of stationary solutions when the value $v^\#$ varies. 
As well as the case of \eqref{3RD1}, the small $v^\#$ values simplify global structure of stationary solutions as shown in Figure \ref{fig_ab-1}.  
\begin{figure}[htbp]
\begin{center}
\begin{tabular}{cc}
\includegraphics[width=65mm]{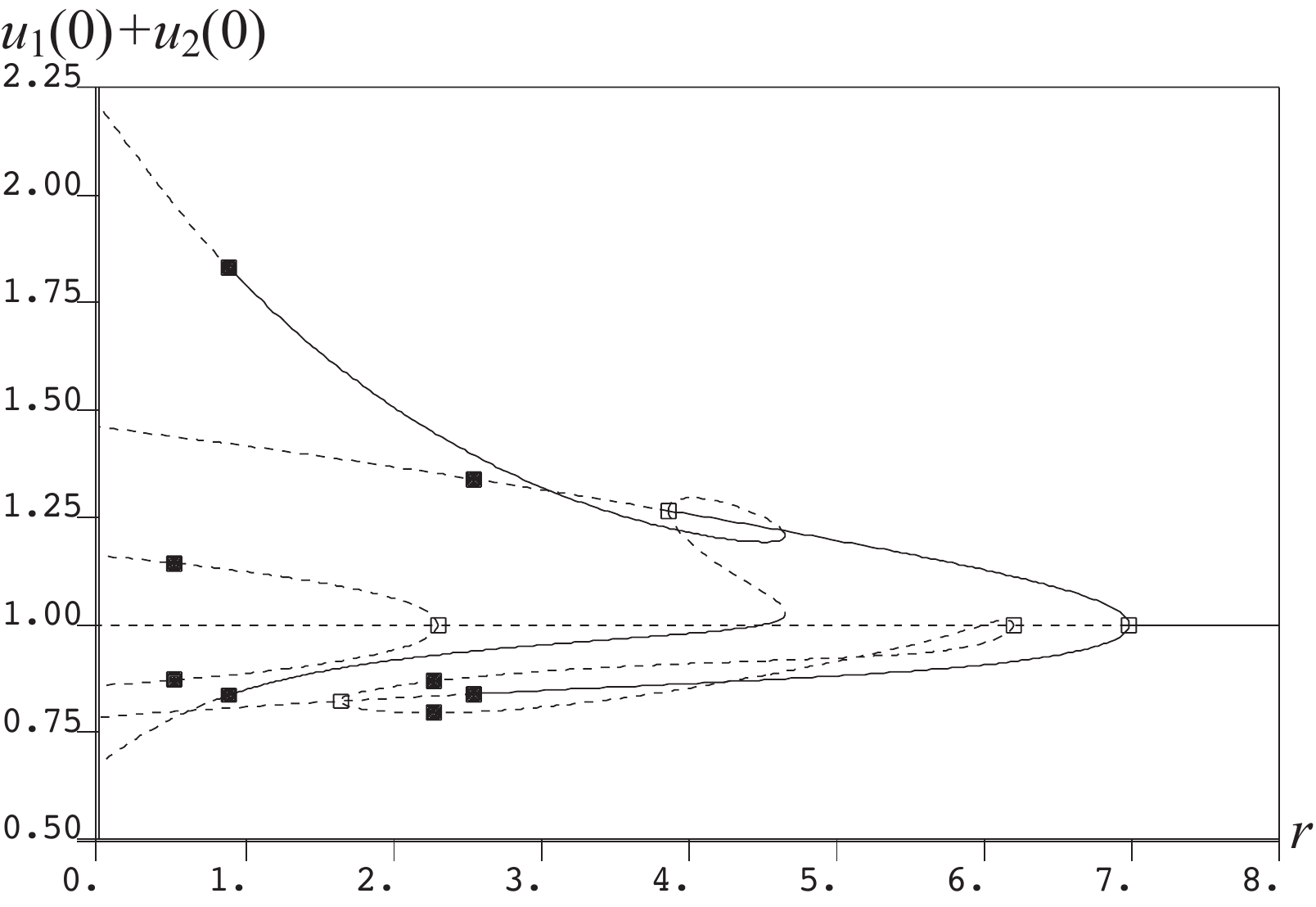} &
\includegraphics[width=65mm]{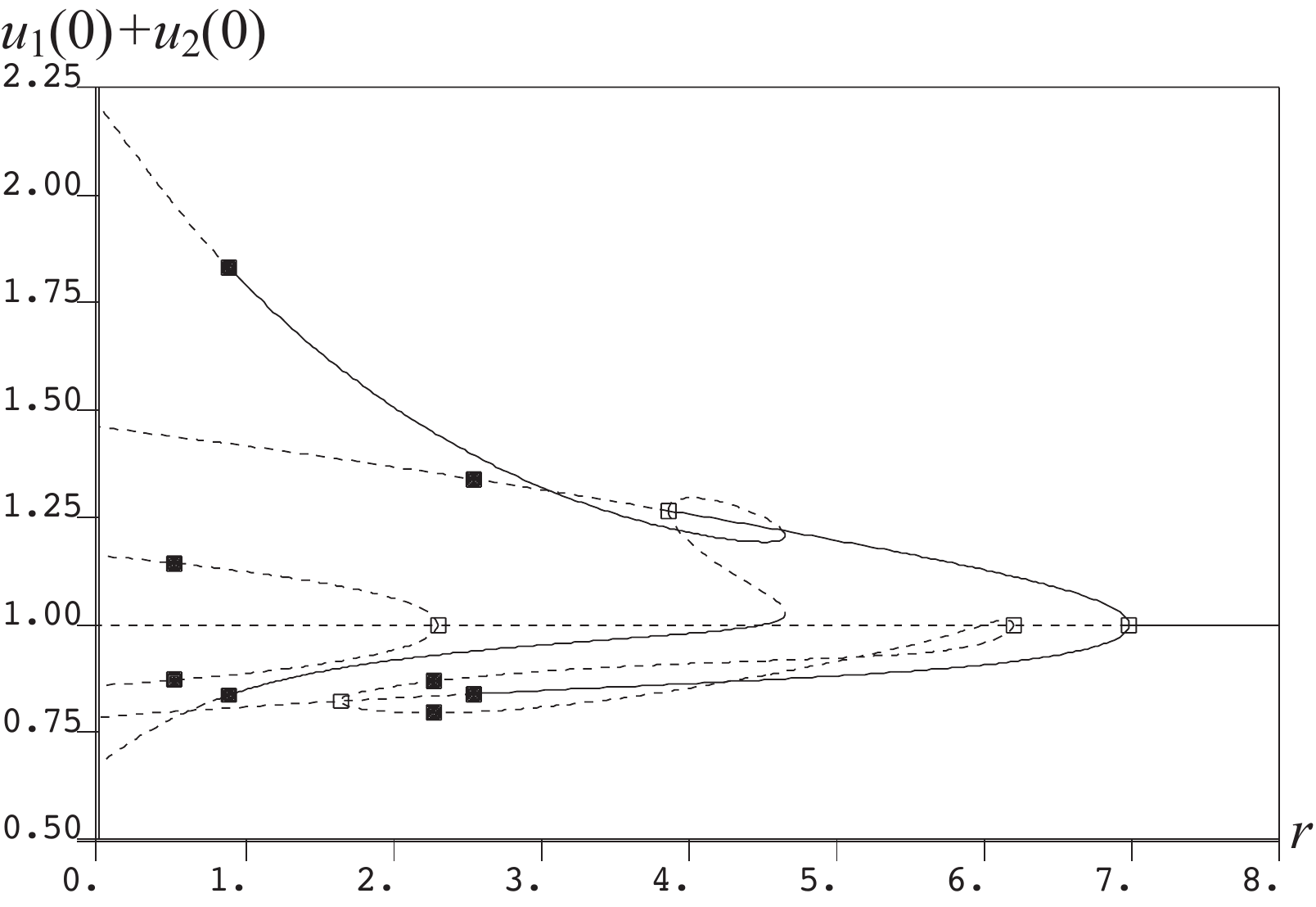}\\
(a) & (b)\\
\includegraphics[width=65mm]{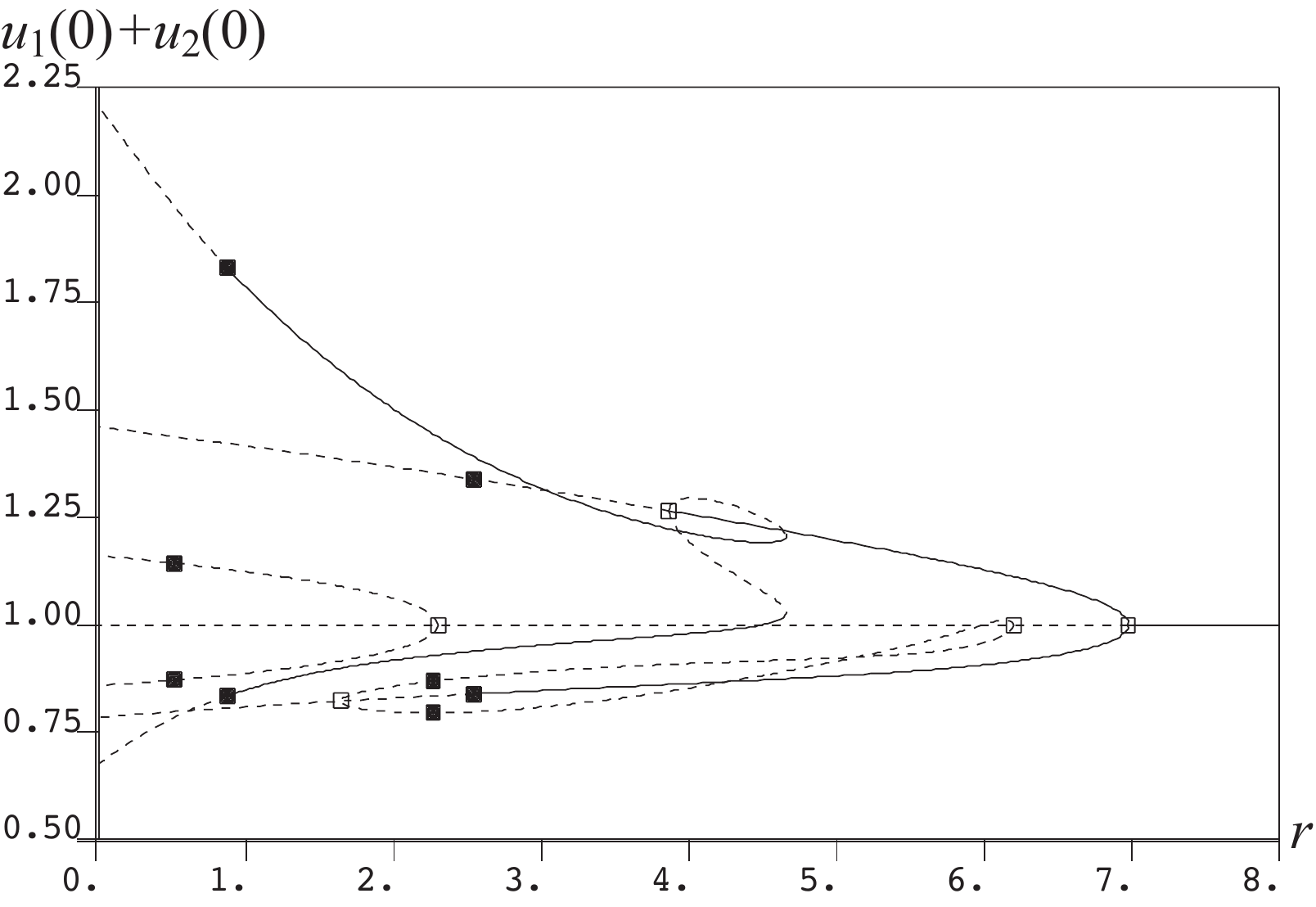} &
\includegraphics[width=65mm]{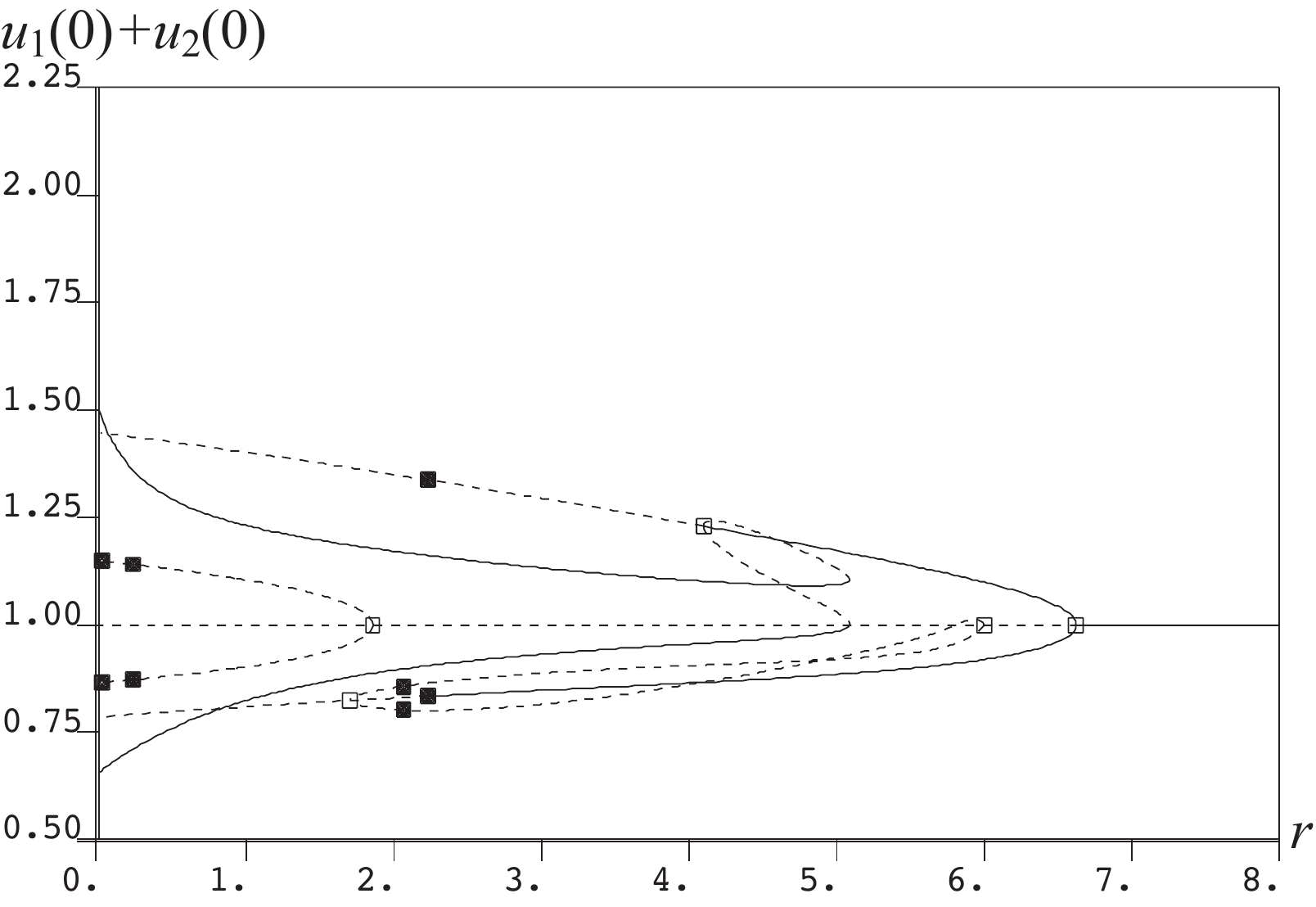}\\
(c) & (d)\\
\includegraphics[width=65mm]{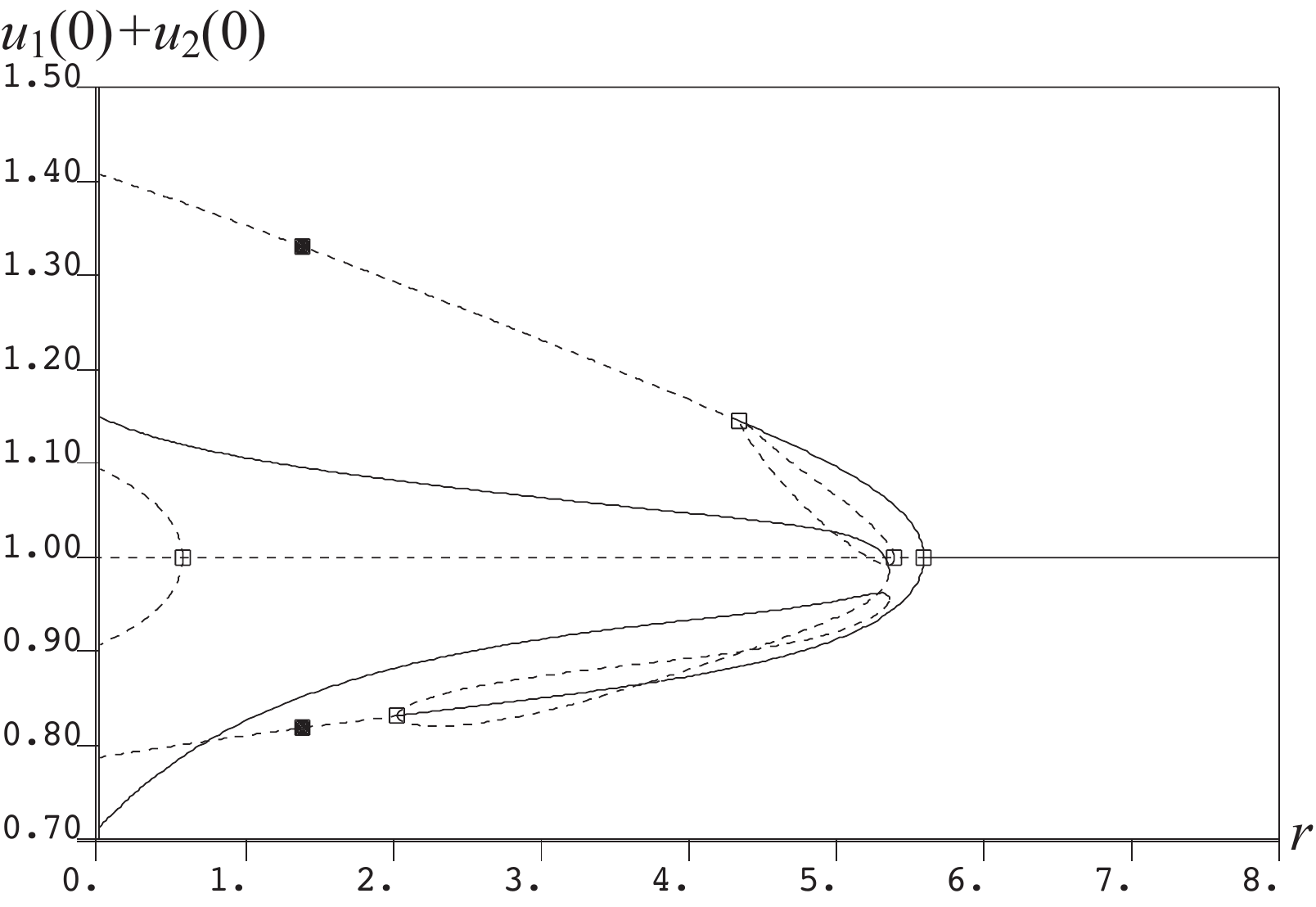} &
\includegraphics[width=65mm]{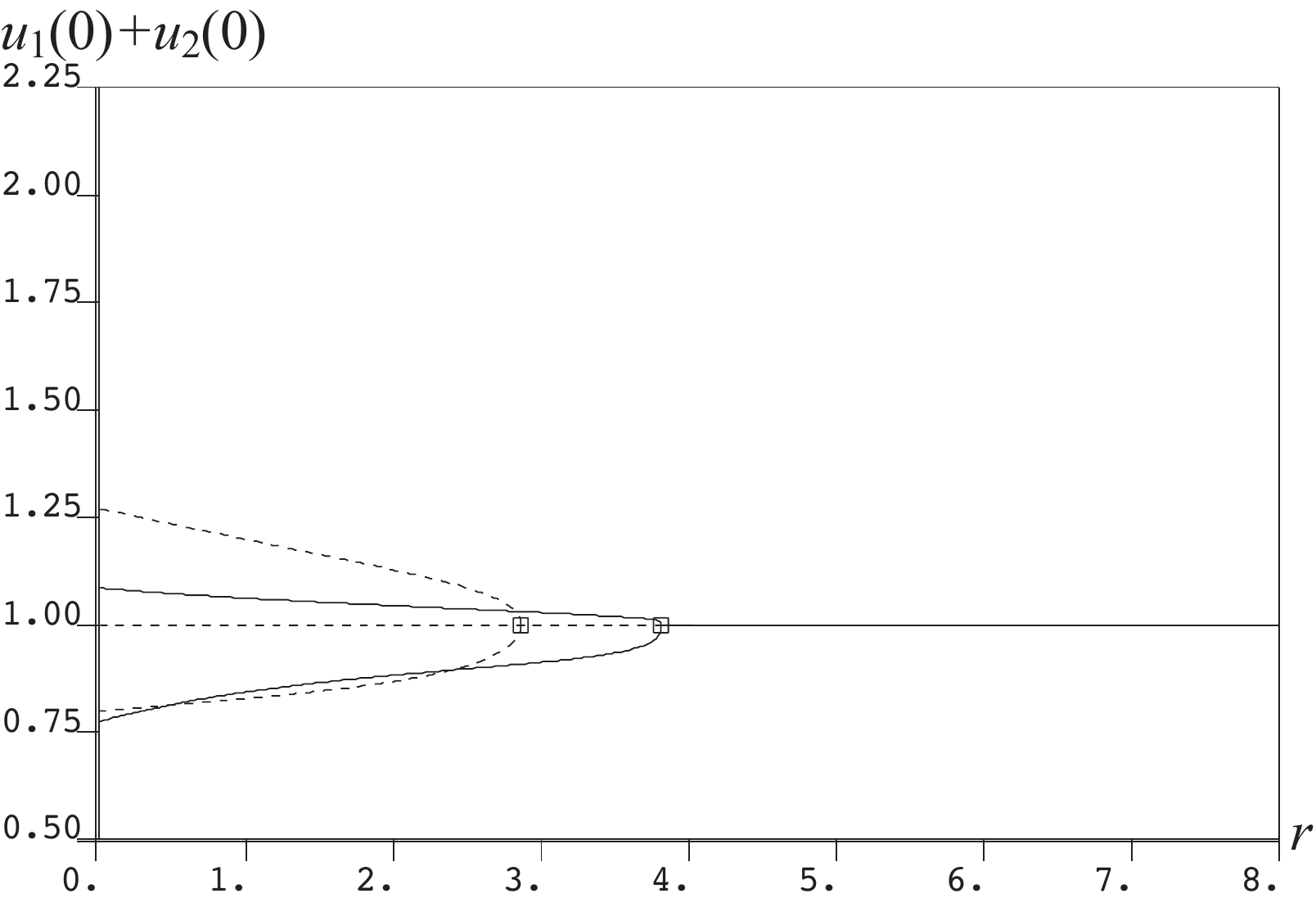}\\
(e) & (f)
\end{tabular}
\caption{Global structure of stationary solutions for \eqref{3RD2}. 
The parameter values are the same as in Figure \ref{exa_fig2} except for $v^\#$. 
(a) $q(v)=\frac{1}{2} \left( 1-\tanh(\gamma_1(v-v^*))\right)$.
(b) $v^\#=1.5$.
(c) $v^\#=1.25$
(d) $v^\#=1.125$. 
(e) $v^\#=1.09$. 
(f) $v^\#=1.06$. 
}
\label{fig_ab-1}
\end{center}
\end{figure}
In particular, when $v^\#$ is suitably small, the Hopf bifurcation points which are the onset of relaxation oscillation in Figure \ref{RO} disappear.

Generally, it seems that the intrinsic growth rate of inactive subpopulation is greater than that of active subpopulation because the regions where inactive individuals stay are a better environment. 
Therefore, it is more natural that the intrinsic growth rates $a_1$ and $a_2$ should be different. 
Here, we set $a_1$ and $a_2$ as $a_1=r$ and $a_2=ar$, and set $r$ as a bifurcation parameter, where $a$ means a ratio between growth rates, so that assume that $0\leq a \leq 1$. When $a=1$, the global structures of stationary solutions correspond to Figures \ref{exa_fig2} and \ref{fig_ab-1}. The global structures with $0\leq a<1$ are shown in Figure \ref{dependency on a}. 
\begin{figure}[htbp]
\begin{center}
\begin{tabular}{cc}
\includegraphics[width=65mm]{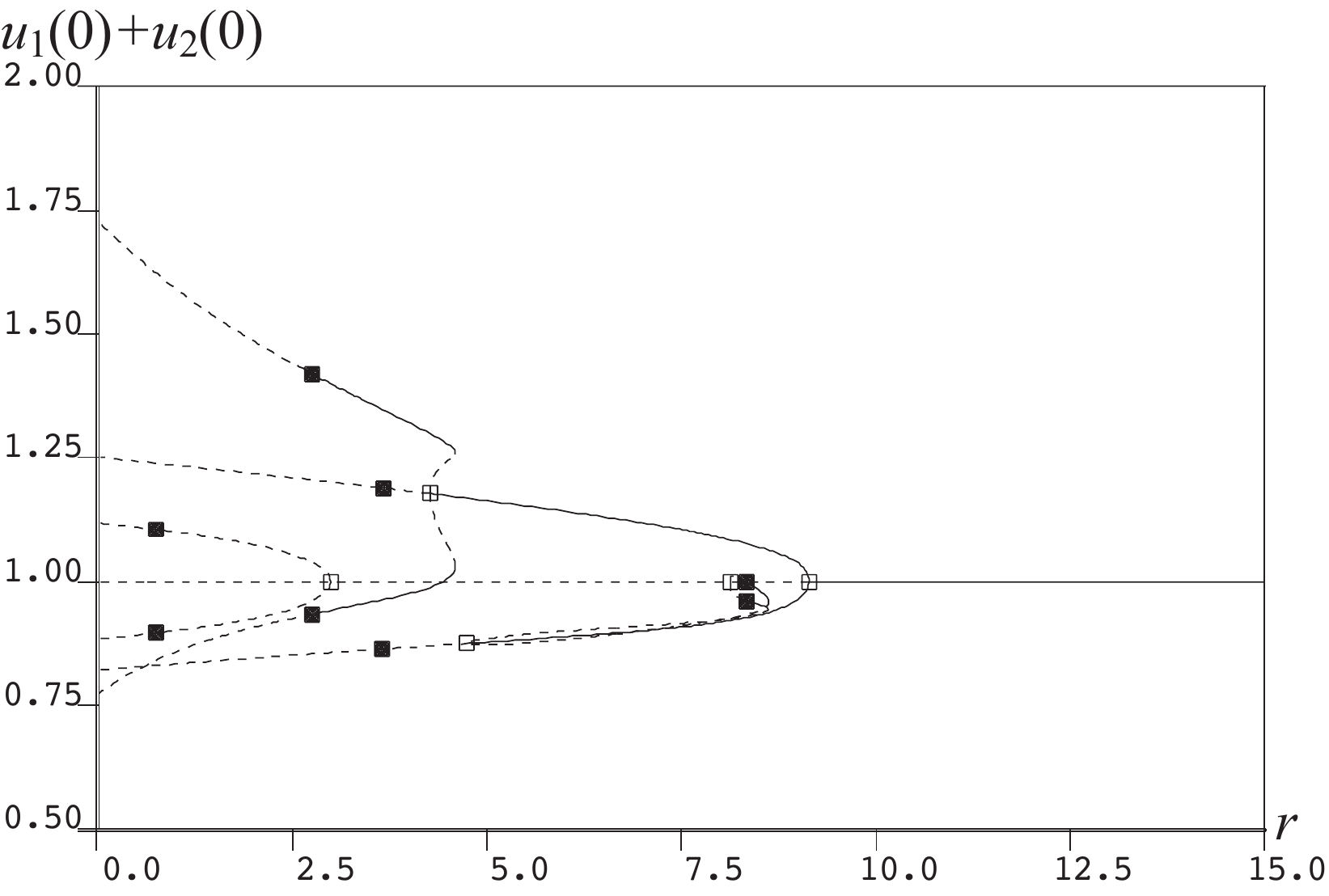}&
\includegraphics[width=65mm]{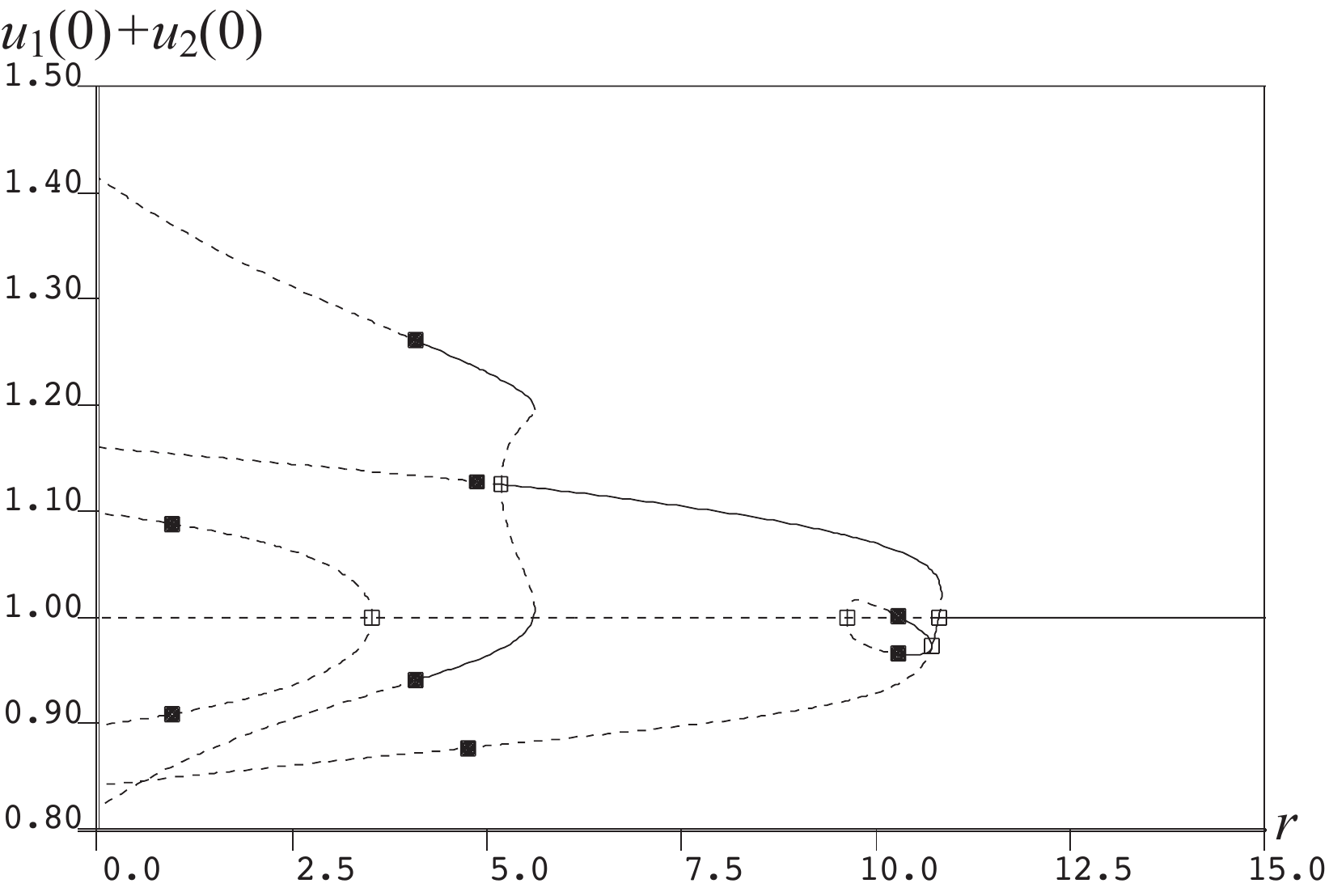}\\
$a=0.5$ & $a=0.25$\\
\includegraphics[width=65mm]{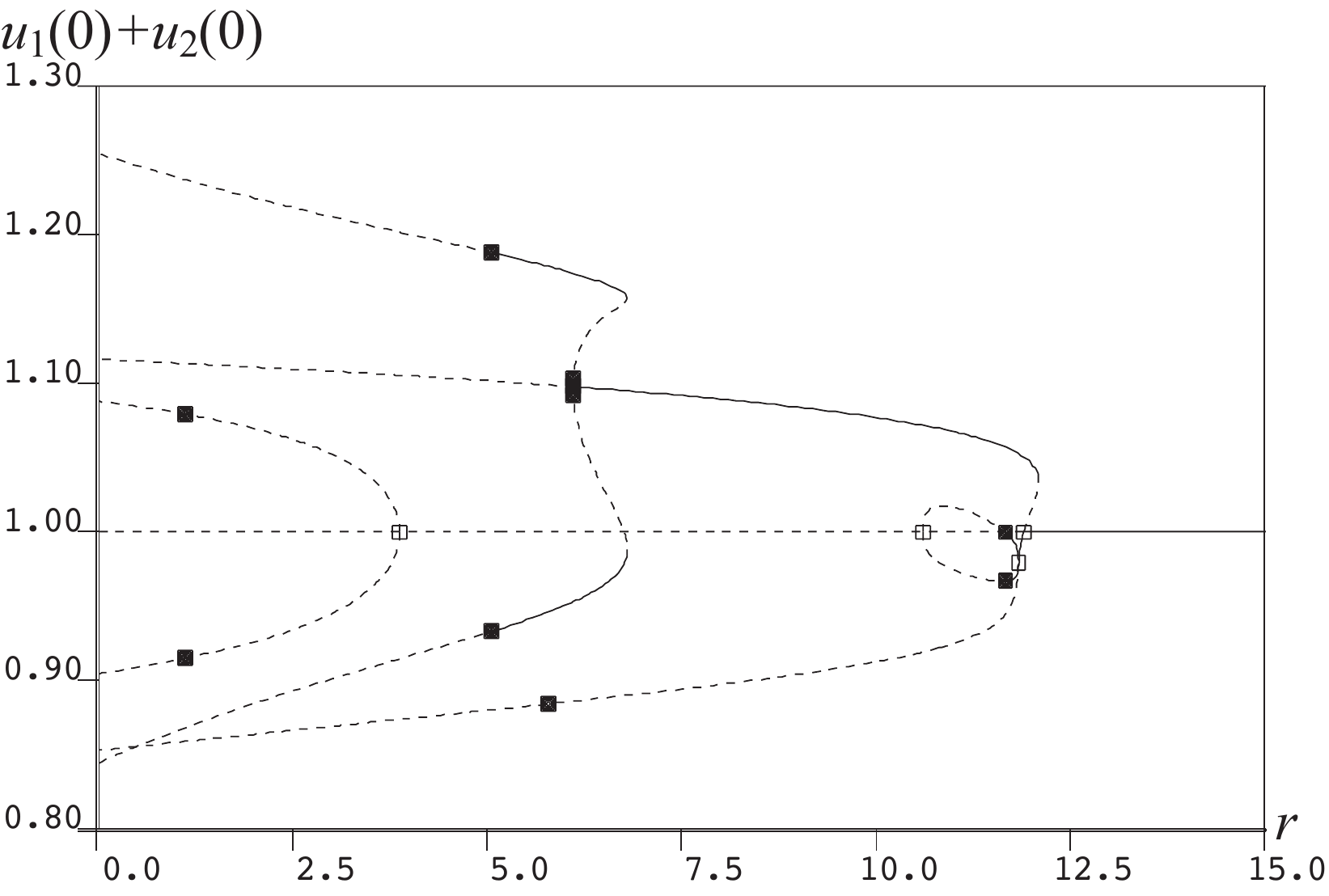}&
\includegraphics[width=65mm]{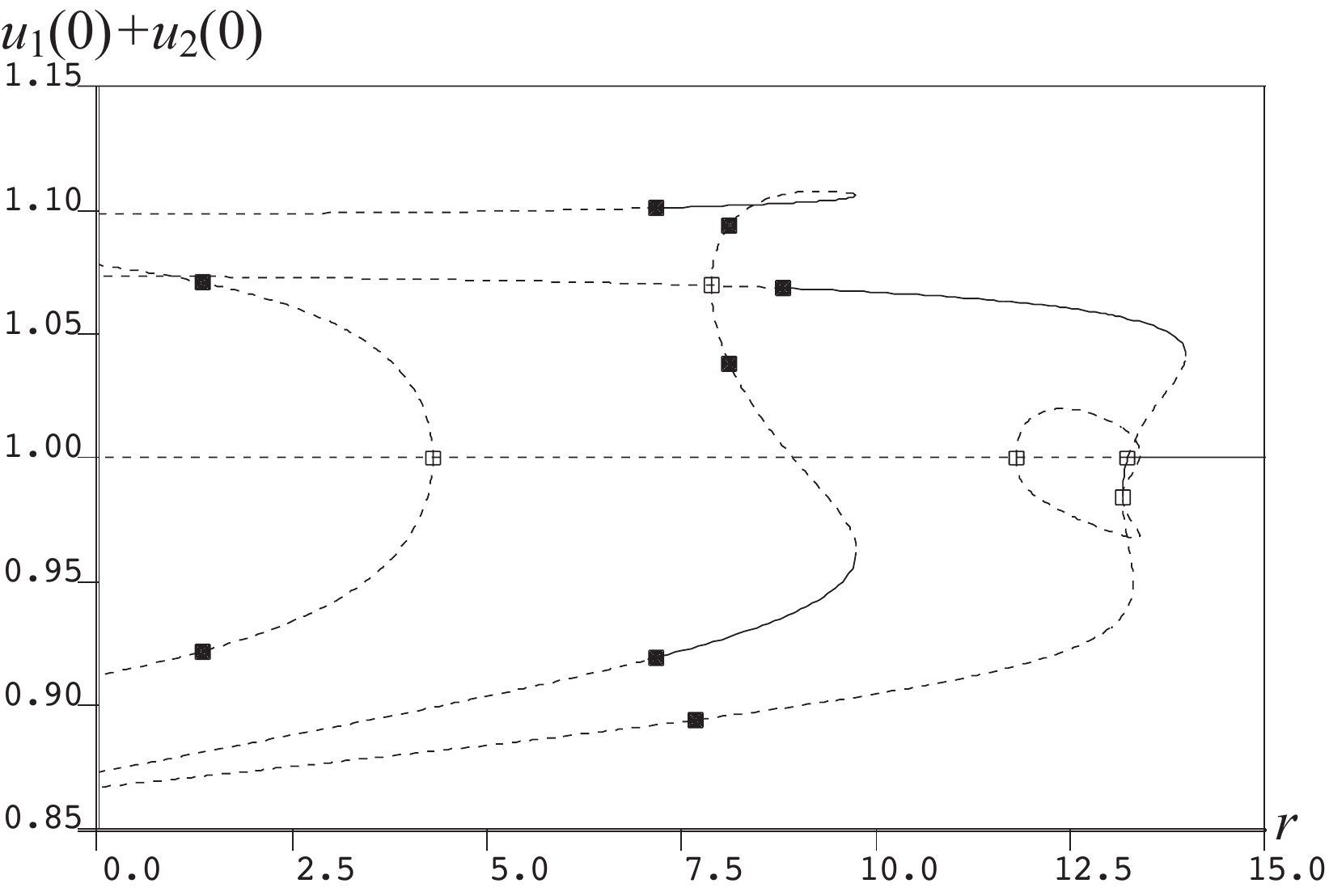}\\
$a=0.125$ & $a=0.0$\\
\end{tabular}
\caption{Global structures of stationary solutions when the value of $a$ varies, where $a_1=r$ and $a_2=ar$ and $r$ is set as a bifurcation parameter. The other parameter values are the same as in Figure \ref{exa_fig2}. }
\label{dependency on a}
\end{center}
\end{figure}
Interestingly, when the intrinsic growth rate  of inactive population density is greater than that of active population density, even the relatively large $r$ values give rise to aggregation patterns. 
Moreover, when the value of $a$ varies, the type of the primal bifurcation from the constant steady state changes from the pitchfork bifurcation ($a=0.5$) to the transcritical bifurcation ($a=0.25$, $a=0.125$, $a=0.0$). One can see that the difference on the intrinsic growth rates between active and inactive population densities influences the pattern dynamics drastically. 

\subsection{Convergence as $\varepsilon\to+0$}

Finally, we numerically discuss the convergence on stationary solutions as $\varepsilon\to+0$. 
In order to do that, the following question naturally arise: how do global structures of stationary solutions change according to the parameter value $\varepsilon$? 
\begin{figure}[htbp]
\begin{center}
\begin{tabular}{cc}
\includegraphics[width=65mm]{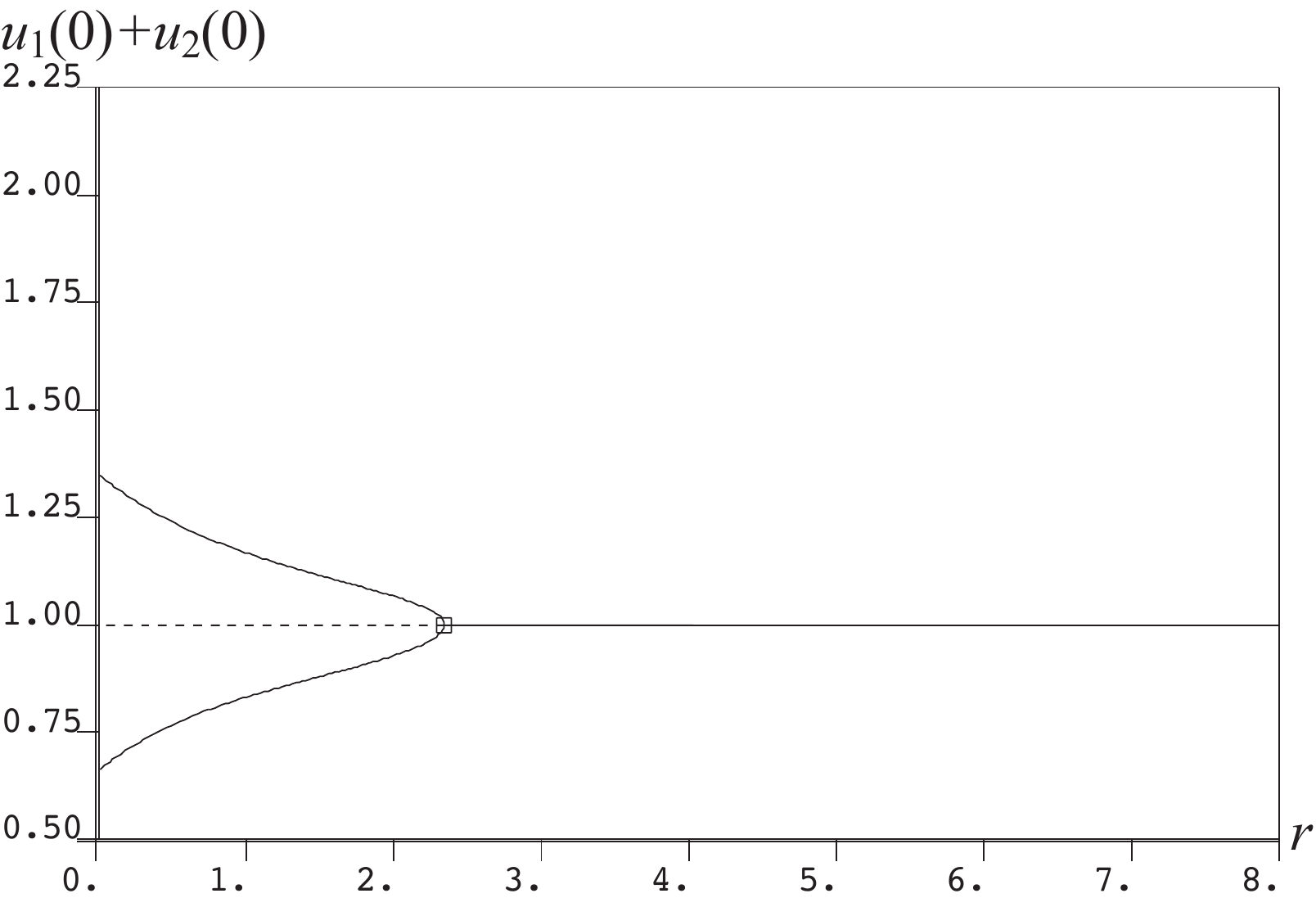}&
\includegraphics[width=65mm]{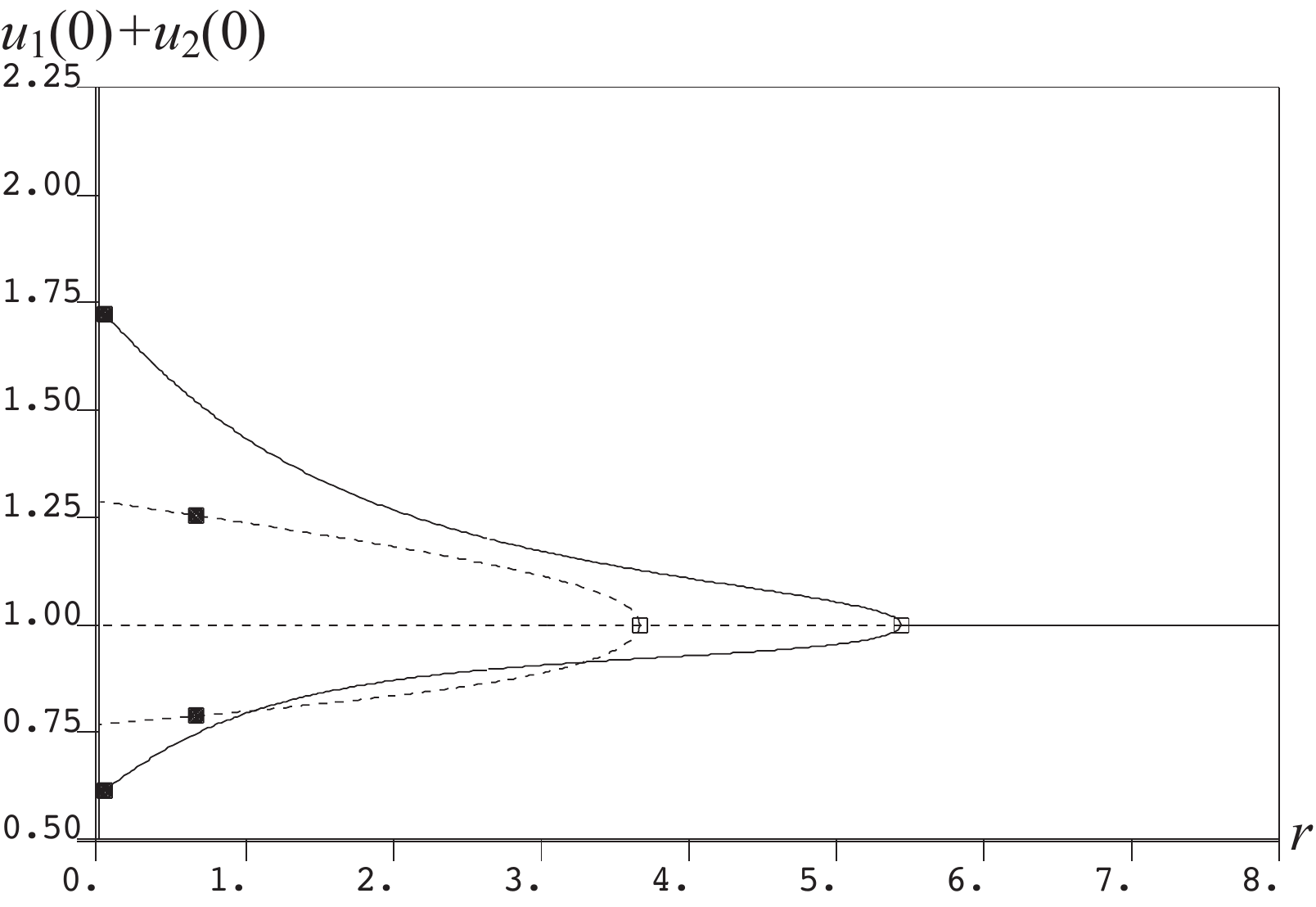}\\
$\varepsilon=1$ & $\varepsilon=0.1$\\
\includegraphics[width=65mm]{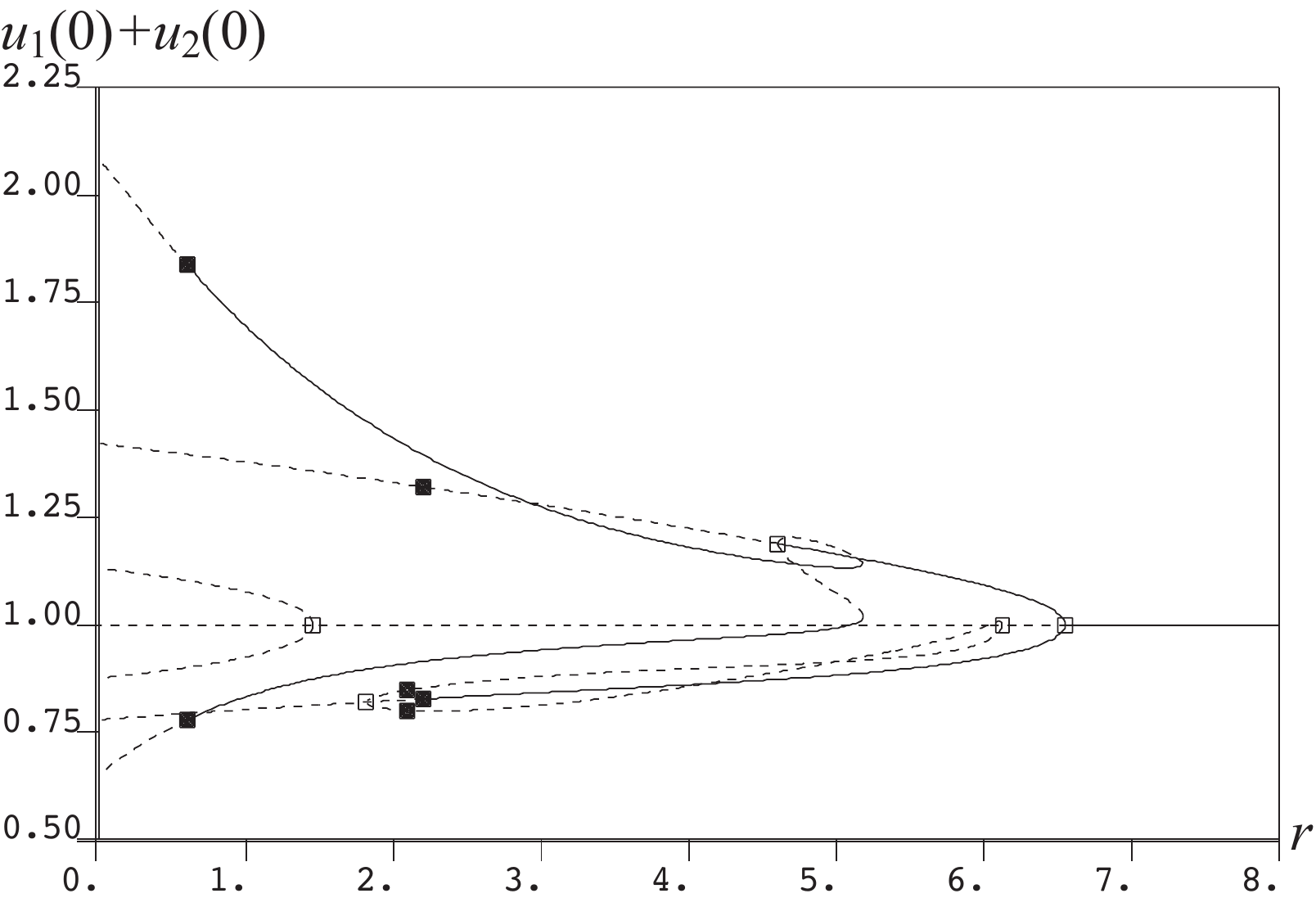}&
\includegraphics[width=65mm]{Figure10-3-eps-converted-to.pdf}\\
$\varepsilon=0.01$ & $\varepsilon=0.001$\\
\includegraphics[width=65mm]{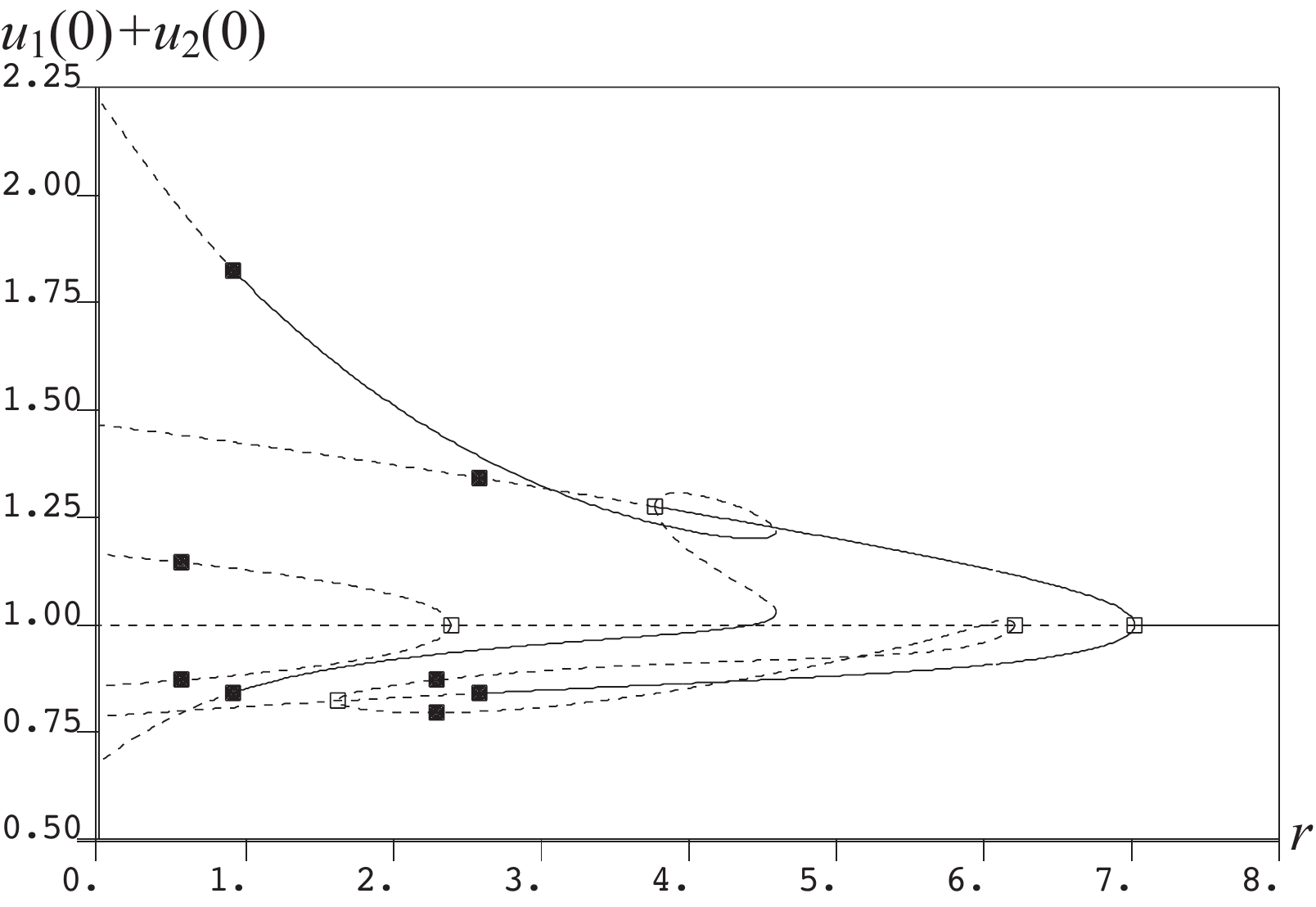} & 
\includegraphics[width=65mm]{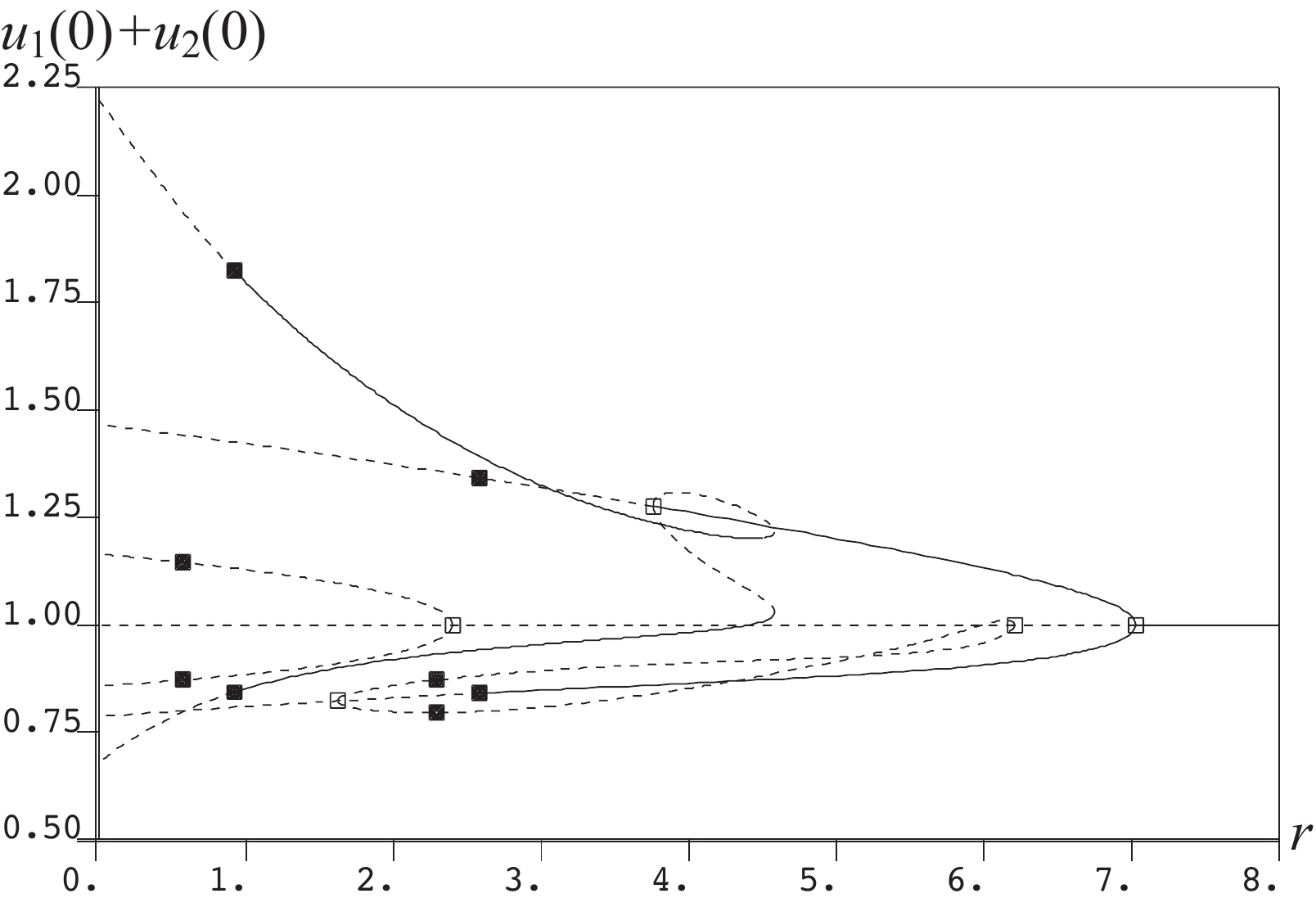}\\
$\varepsilon=0.0001$ & $\varepsilon\to +0$ (Limiting system) 
\end{tabular}
\caption{A transition of global structures of stationary solutions when the value of $\varepsilon$ varies. The parameter values are the same as in Figure \ref{exa_fig2} except for $\varepsilon$. }
\label{dependency on epsilon}
\end{center}
\end{figure}
Figure \ref{dependency on epsilon} gives an answer to the question. 
It seems that the global structures of stationary solutions for \eqref{3RD2} converge to that for the two-component limiting system as the value of $\varepsilon$ tends to zero. 
We numerically confirmed that each solution profiles of \eqref{3RD2}, including periodic solutions observed for the small $r$ values, also approach the corresponding solution profile of the two-component limiting system as $\varepsilon$ tends to zero. 
In addition, for the system without growth term \eqref{3RD1}, a similar result holds. 
These numerical results suggest that not only stationary solutions and periodic solutions but also global solution structure of the three-component reaction-diffusion system converge to those of the two-component limiting system. 
A rigorous proof of the convergence on stationary solutions is not given in this paper, which is our future work.


\section{Singular limit of the problem (\ref{model3}) when $\eps \to 0$} \label{analy}


From a theoretical point of view, we show a singular limit problem of Problem $(\mathcal{P})$, namely,
\begin{equation} \label{eq:01} \left\{ \begin{aligned}
\partial_t \uoneeps & =  d \Delta \uoneeps + a_1 (1 - \uoneeps - \utwoeps)\uoneeps - \frac{1}{\eps}\left( q(\veps)\uoneeps - p(\veps)\utwoeps\right), \\
\partial_t \utwoeps & = (d+D) \Delta \utwoeps  + a_2 (1 - \uoneeps - \utwoeps)\utwoeps + \frac{1}{\eps}\left( q(\veps)\uoneeps - p(\veps)\utwoeps\right), \\
\partial_t \veps & =  D_v\Delta \veps + \alpha(\uoneeps+\utwoeps) - \beta \veps, \\
\end{aligned} \right. \end{equation}
for two functions $p$ and $q$ such that 
\begin{enumerate}[label=(A\theenumi),ref=A\theenumi]
\item\label{A1} $p, q \in C^1(\re_{+},\re_{+})$ with $p(s) > 0$, $q(s) > 0$ for $s > 0$, and
\begin{equation}\label{bound_pq}
	\sup_{s>0}\bra{\left|\frac{p'(s)}{p(s)}\right| + \left|\frac{q'(s)}{q(s)}\right| } < +\infty.
\end{equation}
\end{enumerate}
Note that in general we do not require $p$ and $q$  to satisfy $p + q = 1$. Here the superscript $\eps$ of $\uoneeps, \utwoeps, \veps$ indicates that the solution to \eqref{eq:01} depends on $\eps>0$.

We consider the system \eqref{eq:01} in $Q_T = \Omega \times (0,T)$, where $\Omega\subset\mathbb R^N$ is an open, bounded domain with a sufficiently smooth boundary $\pa\Omega$ and $T>0$. Moreover, we assume the homogeneous Neumann boundary conditions
\begin{equation} \label{eq:02} \na \uoneeps \cdot \nu =\na \utwoeps \cdot \nu = \na \veps \cdot \nu =  0 \text{ on } \partial \Omega \times (0,T),
\end{equation}
where $\nu$ is the unit normal vector pointing outwards of the domain, and 
\begin{equation} \label{eq:03} \uoneeps(x,0) = u_{1,0}(x), \; \utwoeps(x,0) = u_{2,0}(x) \; \text{and} \;\veps(x,0) = v_{0}(x) \text{ for } x\in \Omega
\end{equation}
where $u_{1,0}, u_{2,0}$ and $v_{0}$ satisfy

\begin{enumerate}[label=(A\theenumi),ref=A\theenumi]
\setcounter{enumi}{1}
\item\label{A2} \begin{itemize}
\item $v_0 \in H^1(\Omega)\cap L^{\infty}(\Omega)$ with $v_0(x) \geq \underline{v} > 0$ for a.e. $x\in \Omega$,\\
\item $0\leq u_{1,0}, u_{2,0}\in L^{2+\delta_0}(\Omega)$ for some $\delta_0>0$.
\end{itemize}
\end{enumerate}

\noindent All parameters are positive constants. 

We consider a singular limit of problem \eqref{eq:01}-\eqref{eq:03} when $\eps$ tends to zero and prove that a weak solution $(\uoneeps,\utwoeps,\veps)$ to \eqref{eq:01}-\eqref{eq:03} converges to a limit $(u_1,u_2,v)$ in the sense of Definition~\ref{def:cross} as $\eps \to 0$ where $(u,v)$ for $u = u_1+u_2$ is a weak solution of
\begin{equation} \label{eq:04} \left\{ \begin{aligned}
\partial_t u & =  \Delta \left( d + D \frac{q(v)}{p(v)+q(v)} \right) u + \frac{a_1 p(v) + a_2 q(v)}{p(v) + q(v)} (1 - u)u & \text{ in } Q_T, \\
\partial_t v & =  D_v \Delta v + \alpha u - \beta v & \text{ in } Q_T, \\
 \partial_{\nu} u & = \partial_{\nu} v = 0 & \text{ on } \partial \Omega \times (0,T) \\
u(0) & = u_{1,0} + u_{2,0}, \; v(0) = v_{0} &\text{ in } \Omega.
\end{aligned} \right. \end{equation}

To state our results, we first impose an additional technical assumption.
\begin{enumerate}[label=(A\theenumi),ref=A\theenumi]
\setcounter{enumi}{2}
\item\label{A3} Assume either
\begin{enumerate}
\item\label{A3.1} $N\le 2$;
\item\label{A3.2} or $N \geq 3$ and $\sup_{s>0}(|p(s)| + |q(s)|) < +\infty$;
\item\label{A3.3} or $N\ge 3$ and there exists $r>(N+2)/2$ such that
\begin{equation*}
	C_{\mr,r} < \frac{2d+D}{D}
\end{equation*}
where $C_{\mr,r}$ is the maximal regularity constant in Lemma \ref{maximal_regularity}.
\end{enumerate}
\end{enumerate}

We define the notion of weak solutions to \eqref{eq:04}.
\begin{definition}\label{def:cross}
	Let $T>0$. A pair of functions $(u,v)$ is called a weak solution to \eqref{eq:04} on $(0,T)$ if for any test function $\varphi \in C^{2,1}(\overline{Q}_T)$ with $\varphi(\cdot,T)=0$ and $\nabla \varphi\cdot\nu = 0$ on $\pa\Omega\times(0,T)$, the following integral identities hold
	\begin{equation} \label{eq:w04} \begin{aligned}
		-\iint_{Q_T} u \partial_t \varphi \dx \dt & - \int_{\Omega} (u_{1,0} + u_{2,0}) \varphi(0,x) \dx
		- \iint_{Q_T} \left(d + D \frac{q(v)}{q(v)+p(v)} \right) u \Delta \varphi \dx \dt  \\
		& = \iint_{Q_T}  \frac{a_1 p(v) + a_2 q(v)}{q(v)+p(v)} (1 - u)u   \varphi \dx \dt,
	\end{aligned}\end{equation}
	and
	\begin{equation} \label{eq:w05} \begin{aligned}
		-\iint_{Q_T} v \partial_t \varphi \dx \dt - \int_{\Omega} v_{0} \varphi(0,x) \dx - D_v\iint_{Q_T} v \Delta \varphi \dx \dt = \iint_{Q_T} \left(  \alpha u - \beta v \right) \varphi \dx \dt
	\end{aligned}\end{equation}		
\end{definition}

The main result of this section is the following theorem.
\begin{theorem}\label{thm:01}
Assume \eqref{A1}, \eqref{A2} and \eqref{A3}. Let $T>0$. There exist $0<\delta \le \delta_0$ (with $\delta_0$ is in \eqref{A1}), $u_1, u_2 \in L^{2+\delta}(0,T;L^{2+\delta}(\Omega))$, $v \in L^2(0,T;L^2(\Omega))$ and subsequences (not relabeled) of $(\{\uoneeps, \utwoeps,\veps\}_{\eps > 0}$, which are weak solutions to \eqref{eq:01}, such that
\begin{align*} 
\veps \to v, \quad 
\uoneeps \to u_1, \quad \text{and} \quad  \utwoeps \to u_2 \quad  \text{strongly in } L^2(0,T;L^2(\Omega)),
\end{align*}
as $\eps \to 0$ where 
\begin{equation} \label{eq:21} u_1 = \frac{p(v)}{q(v)+p(v)} (u_1 + u_2) \qquad \text{and} \qquad u_2 = \frac{q(v)}{q(v)+p(v)} (u_1 + u_2)\end{equation}
hold a.e. in $ Q_T$ and $(u, v)$, where $u = u_1+u_2$, is a weak solution to Problem \eqref{eq:04} in the sense of Definition \ref{def:cross}.
\end{theorem}
\begin{remark} 
 The existence of global weak solution $(\uoneeps, \utwoeps, \veps)$ to \eqref{eq:01} is given in Proposition \ref{prop:01}. Note that for this result alone, the assumption \eqref{A3} can be considerably weakened.
\end{remark}



In the following, we start to show the existence of a global weak solution to the system \eqref{eq:01} for fixed $\eps>0$. The crucial estimates, which are uniform in $\eps>0$, are proved in the next subsection. Finally, the limit as $\eps\to 0$ is performed in the last subsection. For simplicity, we will use the following notation for the rest of this paper.
\begin{itemize}
	\item For fixed $T>0$ and $r, s\in [1,\infty]$,
	\begin{equation*}
		L^r_{t}L^s_x:= L^r(0,T;L^s(\Omega)), \; L^r_{x,t} := L^r(0,T;L^r(\Omega)), \; L^r_tW^{1,s}_x := L^r(0,T;W^{1,s}(\Omega)).
	\end{equation*}
	\item For any constant $r\in [1,\infty]$, we denote by $r'$ its H\"older conjugate exponent, i.e.
	\begin{equation*}
		r' = \frac{r}{r-1},
	\end{equation*}
	with the convention $r' = \infty$ when $r = 1$, and $r' = 1$ when $r = \infty$.
\end{itemize}

\subsection{The reaction-diffusion system for $\eps>0$}

We start with the definition of a weak solution to \eqref{eq:01}-\eqref{eq:03}.

\begin{definition}\label{def_weak}
For each $\eps >0$, a triple of non-negative functions $(u_1^{\eps},u_2^{\eps},v^{\eps})$ is called a weak solution on $(0,T)$, $T>0$, to system \eqref{eq:01}-\eqref{eq:03} if  $u_1^{\eps},u_2^{\eps},v^{\eps} \in L^{2}_{x,t}$ for all $T>0$ and, for any test function $\varphi \in C^{2,1}(\overline{Q_T})$ with $\varphi(x,T) = 0$ in $\Omega$, and $\pa_{\nu}\varphi = 0$ on $\pa\Omega\times(0,T)$,
\begin{equation} \label{eq:w01} \begin{aligned}
-\iint_{Q_T} \uoneeps \partial_t \varphi \dx \dt & - \int_{\Omega} u_{1,0} \varphi(0,x) \dx - d \iint_{Q_T} \uoneeps \Delta \varphi \dx \dt \\
& = a_1 \iint_{Q_T} (1 - \uoneeps - \utwoeps)\uoneeps \varphi \dx \dt - \frac{1}{\eps} \iint_{Q_T} (q(\veps) \uoneeps - p(\veps)\utwoeps) \varphi \dx \dt,
\end{aligned}\end{equation}
\begin{equation} \label{eq:w02} \begin{aligned}
-\iint_{Q_T} \utwoeps \partial_t \varphi \dx \dt & - \int_{\Omega} u_{2,0} \varphi(0,x) \dx - (d+D) \iint_{Q_T} \utwoeps \Delta \varphi \dx \dt \\
& = a_2\iint_{Q_T} (1 - \uoneeps - \utwoeps)\utwoeps \varphi \dx \dt + \frac{1}{\eps} \iint_{Q_T} (q(\veps) \uoneeps - p(\veps)\utwoeps) \varphi \dx \dt,
\end{aligned}\end{equation}
and
\begin{equation} \label{eq:w03} \begin{aligned}
-\iint_{Q_T} \veps \partial_t \varphi \dx \dt - \int_{\Omega} v_{0} \varphi(0,x) \dx - D_v\iint_{Q_T} \veps \Delta \varphi \dx \dt = \iint_{Q_T} \left(  \alpha (\uoneeps + \utwoeps) - \beta \veps \right) \varphi \dx \dt.
\end{aligned}\end{equation}
\end{definition}

\begin{proposition}[Weak solutions for fixed $\eps>0$]\label{prop:01}
	Assume \eqref{A1}, \eqref{A2} and either
	\begin{enumerate}[label={\normalfont(A\theenumi')},ref=A\theenumi']
		\setcounter{enumi}{2}
		\item\label{A3'} Assume either \eqref{A3.1} or \eqref{A3.3} or
		\begin{equation}\label{A3'-b}
		N\ge 3 \; \text{ and } \; |p(z)| + |q(z)| \leq C\bra{1+|v|^{1+\frac{4}{N-2}}} \; \forall z\in \mathbb R_+.
		\end{equation}
	\end{enumerate}
	Then there exists a global non-negative weak solution to \eqref{eq:01}-\eqref{eq:03} in the sense of Definition \ref{def_weak}.
\end{proposition}
\begin{remark}
	The global existence of a weak solution to \eqref{eq:01} requires the assumption \eqref{A3'-b} which is slightly weaker than \eqref{A3.2} in the case $N\geq 3$ and $p, q$ are bounded by polynomials of certain orders.
\end{remark}
\begin{proof}[Proof of Proposition \ref{prop:01}]
	The proof of this Proposition follows from recent advances in reaction-diffusion systems with control of mass, see e.g. \cite{canizo2014improved}. For the sake of completeness, we provide a full proof in Appendix \ref{appB}.
\end{proof}

\subsection{Uniform-in-$\eps$ a-priori estimates}
The following proposition contains the main a-priori estimates of the system \eqref{eq:01} which are uniform in $\eps>0$. These estimates are essential in studying the limit $\eps\to 0$. We emphasize that all constants in this subsection are \textit{independent of $\eps>0$} unless stated otherwise.
\begin{proposition}\label{main-estimates}
	Assume \eqref{A1}, \eqref{A2} and \eqref{A3'}, which implies that \eqref{eq:01} has a global weak solution.
	\begin{itemize}
		\item[(i)] There exists a constant $C_T$ depending on $T>0$,
	\begin{equation*}
		\|u_i^{\eps}\|_{L^{\infty}_tL^1_x} + \|u_i^{\eps}\|_{L^{2+\delta_0}_{x,t}} \leq C_T, \quad i = 1,2.
	\end{equation*}
	\item[(ii)] There exists $\delta>0$ such that
	\begin{equation}\label{estimata-v-2+delta}
		\|\veps\|_{L^{2+\delta}_{x,t}} + \|\partial_t\veps\|_{L^{2+\delta}_{x,t}} + \|\na \veps\|_{L^{2+\delta}_{x,t}} \leq C_T.
	\end{equation}
	Assume additionally either \eqref{A3.1} or \eqref{A3.3}. Then we have
	\begin{equation}\label{estimate-v-infty}
		\|\veps\|_{L^{\infty}_{x,t}} \leq C_T.
	\end{equation}
	\item[(iii)] Assume additionally \eqref{A3}. Then there exists $\xi>0$ such that
	\begin{equation}\label{u-gradient}
		\|\na \uoneeps\|_{L^{1+\xi}_{x,t}} + \|\na \utwoeps\|_{L^{1+\xi}_{x,t}} \leq C_T.
	\end{equation}
\end{itemize}
\end{proposition}
The rest of this subsection is devoted to prove Proposition \ref{main-estimates}. The three parts of this proposition are proved in the following sub-subsection respectively, where the first one is shown by using an improved duality method, the second one is obtained from the regularisation of the heat operator, while the last one is proved by using a concave energy function. These parts will be proved consecutively, that means the latter part uses the results of previous part(s) in its proof.

\subsubsection{Improved duality method}
We start with some estimates that can be immediately obtained from the structure of the system.
\begin{lemma}\label{lem:L1_bound}
 	There exists a positive constant $C_T$ independent of $\eps$ such that
\begin{equation}\label{LInfL1}
\norm{\uoneeps}_{L^\infty_tL^1_x} + \norm{\utwoeps}_{L^\infty_tL^1_x} \le C_T 
\end{equation}
and 
\begin{equation}\label{L2}
\norm{\uoneeps}_{L^{2}_{x,t}} + \norm{\utwoeps}_{L^{2}_{x,t}}  \le C_T. 
\end{equation}
\end{lemma}
\begin{proof} We obtain from \eqref{eq:01} and non-negativity of $\uoneeps$ and $\utwoeps$ that
\begin{gather*} \int_{\Omega}(\uoneeps(t)+\utwoeps(t)) \dx + a_1 \int_0^t \int_{\Omega} (\uoneeps)^2 \dx\ds + a_2 \int_0^t \int_{\Omega} (\utwoeps)^2 \dx\ds \\
 \le \int_{\Omega}(u_{1,0}+u_{2,0}) \dx + \max\{a_1,a_2\} \int_0^t \int_{\Omega} (\uoneeps(s)+\utwoeps(s)) \dx \ds
\end{gather*}
for $0 < t \le T$. We deduce \eqref{LInfL1} and \eqref{L2} from the nequality $x \leq \kappa x^2 + \frac{1}{4\kappa^2}$ for any $\kappa >0$.
\end{proof}

To get the estimate in $L^{2+\delta_0}_{x,t}$ we use an improved duality method which was first proved in \cite{canizo2014improved} and later in \cite{Einav-2020}. This improvement looks marginal, but turns out to be crucial in our analysis. We first need the following lemma.

\begin{lemma}\label{maximal_regularity}
Let $1<r<\infty$ and $m>0$. Let $\theta\in L^r_{x,t}$ and $\phi$ be the solution to
\begin{equation}\label{dual}
\begin{cases}
	\partial_t \phi + m\Delta \phi = -\theta, &x\in\Omega, \; t\in (0,T),\\
	\na \phi(x,t)\cdot\nu = 0, &x\in\partial\Omega, \; t\in (0,T),\\
	\phi(x,T) = 0, &x\in\Omega.
\end{cases}
\end{equation}
Then there exists a constant $C_{\mr, r}>0$ such that 
\begin{equation}\label{max}
	\|\Delta \phi\|_{L^r_{x,t}} \leq \dfrac{C_{\mr,r}}{m}\|\theta\|_{L^r_{x,t}}.
\end{equation}
In particular, for any positive constants $\omega_1, \omega_2$, there exists $r^*<2$ depending on $\omega_1, \omega_2$ such that
\begin{equation}\label{pstar}
	\frac{|\omega_1-\omega_2|}{\omega_1+\omega_2}C_{\mr, r^*} < 1.
\end{equation}
\end{lemma}
\begin{proof}
	The proof of this lemma can be found in \cite{Einav-2020}. For the sake of completeness, we provide a full proof in the Appendix \ref{appA}.
\end{proof}

\begin{proof}[Proof of Proposition \ref{main-estimates} part (i)]
	From \eqref{eq:01} and the fact that $\uoneeps, \utwoeps \geq 0$, we have 
	\begin{equation}\label{b1}
		\partial_t(\uoneeps + \utwoeps) - \Delta(d\uoneeps + (d+D)\utwoeps) \leq a_1\uoneeps + a_2\utwoeps.
	\end{equation}
	Let $a = \max\{a_1,a_2\}$. By defining $\yoneeps(x,t) = e^{-at}\uoneeps(x,t)$ and $\ytwoeps(x,t) = e^{-at}\utwoeps(x,t)$, it follows from \eqref{b1} that
	\begin{equation}\label{b2}
		\partial_t(\yoneeps + \ytwoeps) - \Delta(d\yoneeps + (d+D)\ytwoeps) \leq  (a_1-a)\uoneeps + (a_2 - a)\utwoeps \leq 0.
	\end{equation}
	Let $\omega_1 = d$ and $\omega_2 = d+D$, and let $r^*<2$ be the constant defined in \eqref{pstar} in Lemma \ref{maximal_regularity}. Assume $0\leq \theta\in L^{r^*}_{x,t}$ arbitrary, and let $\phi$ be the solution to \eqref{dual} with $m = d + D/2$. From the comparison principle, it follows that $\phi \geq 0$. Moreover, thanks to \eqref{max}, 
	\begin{equation}\label{b3}
		\|\Delta \phi\|_{L^{r^*}_{x,t}} \leq \frac{C_{\mr,r^*}}{m}\|\theta\|_{L^{r^*}_{x,t}}.
	\end{equation}
	It also follows from this that
	\begin{equation}\label{b3_1}
		\|\partial_t\phi\|_{L^{r^*}_{x,t}} \leq m\|\Delta\phi\|_{L^{r^*}_{x,t}} + \|\theta\|_{L^{r^*}_{x,t}} \leq (C_{\mr,r^*}+1)\|\theta\|_{L^{r^*}_{x,t}}.
	\end{equation}
	Therefore,
	\begin{equation}\label{b4}
		\|\phi(\cdot,0)\|_{L^{r^*}_{x}} = \left(\int_{\Omega}\left|\int_0^T\partial_t\phi \dt\right|^{r^*}\dx \right)^{\frac{1}{{r^*}}} \leq T^{\frac{1}{{r^*}'}} \|\partial_t\phi\|_{L^{r^*}_{x,t}} \leq T^{\frac{1}{{r^*}'}} (C_{\mr,r^*}+1)\|\theta\|_{L^{r^*}_{x,t}}.
	\end{equation}
	Using integration by parts we have
	\begin{equation}\label{b4_1}
	\begin{aligned}
		\iint_{Q_T}(\yoneeps + \ytwoeps)\theta \dx\dt = -\iint_{Q_T}(\yoneeps + \ytwoeps)(\partial_t \phi + m\Delta \phi)\dx\dt\\
		= \int_{\Omega}(\yoneeps(\cdot,0) + \ytwoeps(\cdot,0))\phi(\cdot,0)\dx + \iint_{Q_T}\phi(\partial_t(\yoneeps + \ytwoeps) - \Delta(d\yoneeps + (d+D)\ytwoeps)\dx\dt\\ 
		+ (d-m)\iint_{Q_T}\yoneeps \Delta \phi \dx\dt + (d+D-m)\iint_{Q_T}\ytwoeps \Delta \phi \dx\dt\\
		\leq \|\uoneeps(\cdot,0)+\utwoeps(\cdot,0)\|_{L^{{r^*}'}_x}\|\phi(\cdot,0)\|_{L^{r^*}_x} + \dfrac{D}{2}\|\yoneeps + \ytwoeps\|_{L^{{r^*}'}_{x,t}}\|\Delta \phi\|_{L^{r^*}_{x,t}}\\
		\leq \left(\|\uoneeps(\cdot,0)+\utwoeps(\cdot,0)\|_{L^{{r^*}'}_x}T^{\frac{1}{{r^*}'}} + \frac{C_{\mr,r^*}D/2}{m}\|\yoneeps + \ytwoeps\|_{L^{{r^*}'}_{x,t}} \right)\|\theta\|_{L^{r^*}_{x,t}}.
	\end{aligned}
	\end{equation}
	Since $0\leq \theta\in L^{r^*}_{x,t}$ arbitrary, by duality we have
	\begin{align*}
		\|\yoneeps + \ytwoeps\|_{L^{{r^*}'}_{x,t}} &\leq \|\uoneeps(\cdot,0)+\utwoeps(\cdot,0)\|_{L^{{r^*}'}_x}T^{\frac{1}{{r^*}'}} + \frac{C_{\mr,r^*}D/2}{m}\|\yoneeps + \ytwoeps\|_{L^{{r^*}'}_{x,t}}.
	\end{align*}
	Now thanks to \eqref{pstar}, we have
	\begin{equation*}
		\frac{C_{\mr,r^*}D/2}{m} = \frac{C_{\mr,r^*}D}{2d+D} = \frac{\omega_2 - \omega_1}{\omega_1+\omega_2}C_{\mr,r^*} < 1
	\end{equation*}
	recalling $\omega_1 = d$, $\omega_2 = d+D$, and $m = d + D/2$. Therefore,
	\begin{equation*}
		\|\yoneeps + \ytwoeps\|_{L^{{r^*}'}_{x,t}} \leq \left(1- \frac{C_{\mr,r^*}D}{2d+D}\right)^{-1} \|\uoneeps(\cdot,0)+\utwoeps(\cdot,0)\|_{L^{{pr^*}'}_x}T^{\frac{1}{{r^*}'}}.
	\end{equation*}
	Thanks to the non-negativity of $\yoneeps$ and $\ytwoeps$ we have
	\begin{equation*}
		\|\yoneeps\|_{L^{{r^*}'}_{x,t}} + \|\ytwoeps\|_{L^{{r^*}'}_{x,t}} \leq C_T
	\end{equation*}
	and therefore,
	\begin{equation*}
		\|\uoneeps\|_{L^{{r^*}'}_{x,t}} + \|\utwoeps\|_{L^{{r^*}'}_{x,t}} \leq C_T
	\end{equation*}
	which are the desired estimates since ${r^*}' = \frac{r^*}{r^* - 1} > 2$.
\end{proof}

\subsubsection{Heat regularisation}

We will use the following lemma.
\begin{lemma}[Heat regularisation, \cite{ladyvzenskaja1988linear}]\label{lem:heat_gradient}
	Let $0<d$, $1<r<+\infty$ and $f\in L^r_{x,t}$. Let $u$ be the solution of the following linear parabolic equation
	\begin{equation}
	\begin{cases}
		\partial_t u - d\Delta u = f, &(x,t)\in Q_T,\\
		\nabla u \cdot \nu = 0, &(x,t)\in \partial\Omega\times(0,T),\\
		u(x,0) = u_0(x).
	\end{cases}
	\end{equation}
	Then
	\begin{equation*}
		\|\partial_t u\|_{L^{r}_{x,t}} + \|u\|_{L^{r_1}_{x,t}} + \|\nabla u\|_{L^{r_2}_{x,t}} \leq C_T\|f\|_{L^r_{x,t}}
	\end{equation*}
	where the constant $C_T$ depends on $T$ but is independent of $f$, and
	\begin{equation}\label{r1}
	r_1=\begin{cases}
	\frac{(N+2)r}{N+2-2r} &\text{if } r<(N+2)/2,\\
	<+\infty \text{ arbitrary } &\text{if } r = (N+2)/2,\\
	+\infty &\text{if } r>(N+2)/2.
	\end{cases}
	\end{equation}
	and
	\begin{equation}\label{r2}
		r_2=\begin{cases}
			\frac{(N+2)r}{N+2-r} &\text{if } r<N+2,\\
			<+\infty \text{ arbitrary } &\text{if } r = N+2,\\
			+\infty &\text{if } r>N+2.
		\end{cases}
	\end{equation}
\end{lemma}

\begin{proof}[Proof of Proposition \ref{main-estimates} part (ii)]
	By defining $\vepshat(x,t) = e^{\beta t}v(x,t)$ we have
	\begin{equation}\label{eq_v_hat}
	\begin{cases}
		\partial_t \vepshat - D_v\Delta\vepshat = e^{\beta t}\alpha(\ueps_1 + \ueps_2), &(x,t)\in Q_T,\\
		\nabla \vepshat\cdot \nu = 0, &(x,t)\in \partial\Omega\times(0,T),\\
		\vepshat(x,t) = v_0(x), &x\in\Omega.
	\end{cases}
	\end{equation}
	By part (i) of Proposition \ref{main-estimates} which was proved in the previous sub-subsection, $f:= e^{\beta t}\alpha(\ueps_1 + \ueps_2)\in L^{2+\delta_0}_{x,t}$ for $\delta_0>0$ in Assumption \eqref{A2}. Therefore, by applying Lemma \ref{lem:heat_gradient} we obtain that $$\|\partial_t \vepshat\|_{L^{2+\delta_0}_{x,t}}+\|\nabla \vepshat\|_{L^{r_2}_{x,t}}\leq C_T$$
	where $r_2$ is computed as in \eqref{r2} with $r = 2+ \delta_0$. Note that, in any case, we have $r_2 > 2+ \delta_0 > 2$. Therefore, switching back to $\veps$ we obtain the desired estimates in \eqref{estimata-v-2+delta}. 
	
	It remains to show the $L^\infty_{x,t}$-bound of $\veps$ in \eqref{estimate-v-infty}. In case $N\leq 2$, we have immediately
	\begin{equation*}
		\|\widehat{v}^\eps\|_{L^\infty_{x,t}} \leq C_T
	\end{equation*} 
	thanks to the heat regularisation in Lemma \ref{lem:heat_gradient} since $f = e^{\beta t}\alpha (\uoneeps + \utwoeps) \in L^{2 + \delta_0}_{x,t}$ and $2 + \delta_0 > (N+2)/2$. For the case \eqref{A3.3} holds, it is sufficient to show that 
	\begin{equation}\label{bound-u-Lr}
		\|\uoneeps\|_{L^r_{x,t}} + \|\utwoeps\|_{L^r_{x,t}} \le C_T
	\end{equation}
	where, recalling that, $r>(N+2)/2$.
	We use again the improved duality method. Let $0\leq \theta \in L^{r'}_{x,t}$ arbitrary, where $r' = \frac{r}{r-1}$ is the H\"older conjugate exponent of $r$. Let $\phi$ be the solution to \eqref{dual} with $m = d + D/2$. Lemma \ref{maximal_regularity} gives
	\begin{equation*}
		\|\Delta \phi\|_{L^{r'}_{x,t}} \leq \frac{C_{mr,r'}}{m}\|\theta\|_{L^{r'}_{x,t}}.
	\end{equation*}
	Similarly to \eqref{b3_1} and \eqref{b4} we have
	\begin{equation*}
		\|\phi(\cdot,0)\|_{L^{r'}_x} \leq C_T\|\theta\|_{L^{r'}_{x,t}}.
	\end{equation*}
	We denote by $a = \max\{a_1, a_2\}$ and define $\yoneeps(x,t) = e^{-a t}\uoneeps(x,t)$ and $\ytwoeps(x,t) = e^{-at}\utwoeps(x,t)$. This change of variables gives the inequality \eqref{b2}. 	By using integration by parts similarly to that of \eqref{b4_1} we have
	\begin{equation*}
		\begin{aligned}
			\intQT (\yoneeps + \ytwoeps)\theta \dx\dt &\leq \|\uoneeps(\cdot,0) + \utwoeps(\cdot,0)\|_{L^{r}_x}\|\phi(\cdot,0)\|_{L^{r'}_x} + \frac{D}{2}\|\yoneeps + \ytwoeps\|_{L^{r}_{x,t}}\|\Delta \phi\|_{L^{r'}_{x,t}}\\
			&\leq \left(\|\uoneeps(\cdot,0) + \utwoeps(\cdot,0)\|_{L^{r}_x}C_T + \frac{DC_{\mr,r}}{2d + D}\|\yoneeps + \ytwoeps\|_{L^r_{x,t}} \right)\|\theta\|_{L^{r'}_{x,t}}.
		\end{aligned}
	\end{equation*}
	Since $0\le \theta \in L^{r'}_{x,t}$ is arbitrary, duality yields
	\begin{equation*}
		\|\yoneeps + \ytwoeps\|_{L^r_{x,t}} \le \|\uoneeps(\cdot,0) + \utwoeps(\cdot,0)\|_{L^{r}_x}C_T + \frac{DC_{\mr,r}}{2d + D}\|\yoneeps + \ytwoeps\|_{L^r_{x,t}}.
	\end{equation*}
	From this, we can use the condition \eqref{A3.3} and the non-negativity of $\yoneeps$ and $\ytwoeps$ to get
	\begin{equation*}
		\|\yoneeps\|_{L^r_{x,t}} + \|\ytwoeps\|_{L^r_{x,t}} \leq C_T
	\end{equation*}
	which gives the desired estimates \eqref{bound-u-Lr}.
\end{proof}

\subsubsection{Energy estimates}

We start with a lower bound for $\veps$ on any finite interval $(0,T)$.
\begin{lemma}\label{lem:lower_bound}
	Assume \eqref{A1}, \eqref{A2} and \eqref{A3}. Then there exists $\underline{v}>0$ such that $\veps(x,t) \ge 
	\underline{v}>0$ for a.e. $(x,t)\in Q_T$ for each $T>0$. Consequently, there's some constant $\eta_T$ depending on $T$ such that
	\begin{equation} \label{eq:15} 
	\inf_{(x,t) \in Q_T} p(\veps(x,t)) \ge \eta_T \quad \text{and} \quad \inf_{(x,t) \in Q_T} q(\veps(x,t)) \ge \eta_T. \end{equation}
\end{lemma}
\begin{proof}
	Since 
	\[ \partial_t \veps - D_v\Delta \veps = \alpha(\uoneeps + \utwoeps) - \beta \veps \ge - \beta \veps,\]
	every solution $w$ to problem
	\begin{equation}\label{eq:16} \left\{ \begin{aligned}
	\partial_t w - D_v \Delta w & = - \beta w, \\
	\partial_{\nu} w &= 0, \\
	w(0,x) &= \inf_{x \in \Omega} v_0(x) =: \eta,
	\end{aligned} \right. \end{equation}
	is a lower solution to the equation of $\veps$, and therefore, thanks to the comparison principle,
	\begin{equation*}
		\veps(x,t) \geq w(x,t) \quad \text{ a.e. } (x,t)\in Q_T.
	\end{equation*}
	Moreover, every solution $w$ to Problem \eqref{eq:16} is bounded from below by a solution to the equation $v' = -\beta v$ with, for example, $v(0) =  \eta/2$, that is by $v(t) = \eta/2 e^{-\beta t}$. Thus, for each $T > 0$, there is $\veps \ge C e^{-\beta T} > 0$ in $Q_T$.  In view of Assumption \eqref{A1} we deduce \eqref{eq:15}.
\end{proof}

In order to obtain estimates on the gradients of $\uoneeps$ and $\utwoeps$, we use an energy method with a \textit{concave} energy function. This idea was utilized in \cite{desvillettes2015new}.

\begin{proof}[Proof of Proposition \ref{main-estimates} part (iii)]
	For simplicity, in this lemma we drop the superscript $\eps$ and write simply $u_1, u_2, u, v$ instead of $\uoneeps, \utwoeps, \ueps, \veps$. Note that all the constants in this Lemma are independent of $\eps>0$. 
	We first show that 
	\begin{equation}\label{e1}
		\iint_{Q_T}\bra{\left|\na \bra{u_1^{\theta/2}}\right|^2+\left|\na\bra{ u_2^{\theta/2}}\right|^2}\dx\dt \le C_T
	\end{equation}
	for all parameter $\theta$ such that 
	\begin{equation*}
		0 < \theta < \min\{1,\delta_0\}
	\end{equation*}
	where $\delta_0$ is in Assumption \eqref{A2}. In order to do that, we multiply the equation of $u_1$ and $u_2$ in \eqref{eq:01} by $[q(v)u_1]^{\theta-1}$ and $[p(v)u_2]^{\theta-1}$, respectively, then sum the resultants to have
	\begin{equation}\label{e2}
	\begin{aligned}
		&\intQT\sbra{\pa_t u_1 - d\Delta u_1 - a_1(1-u)u_1}\sbra{q(v)u_1}^{\theta-1}\dx\dt\\
		&\quad + \intQT \sbra{\pa_t u_2 - (d+D)\Delta u_2 - a_2(1-u)u_2}\sbra{p(v)u_2}^{\theta-1}\dx\dt\\
		&\quad \quad = -\frac{1}{\eps}\intQT\sbra{q(v)u_1 - p(v)u_2}\sbra{(q(v)u_1)^{\theta-1} - (p(v)u_2)^{\theta-1}} \dx\dt \geq 0
	\end{aligned}
	\end{equation}
	where we used at the last step the elementary inequality
	\begin{equation}\label{e3}
		-(x-y)(x^{\theta-1}-y^{\theta-1}) = \frac{(x-y)(x^{1-\theta} - y^{1-\theta})}{(xy)^{1-\theta}} \geq 0 \quad \text{ for all } x, y > 0 \text{ and } \theta \in (0,1).
	\end{equation}
	Denote by $(H_1)$ and $(H_2)$ the two integrals on the left hand side of \eqref{e2}. We have
	\begin{gather*}
		(H_1) = \intQT \pa_t u_1 [q(v)u_1]^{\theta-1} \dx\dt - d\intQT \Delta u_1 [q(v)u_1]^{\theta-1} \dx\dt \\ 
		- a_1\intQT (1-u)u_1[q(v)u_1]^{\theta-1} \dx\dt =: (H_{11}) + (H_{12}) + (H_{13}).
	\end{gather*}
	Explicit calculations on the first and third term give
	\begin{align*}
			(H_{11}) &= \frac{1}{\theta}\intQT \sbra{\pa_t\sbra{u_1^{\theta}q(v)^{\theta-1}} - u_1^{\theta}\pa_t[q(v)^{\theta-1}]} \dx\dt\\
			&= \frac{1}{\theta}\intO\sbra{u_1(T)^{\theta}q(v(T))^{\theta-1} - u_{1,0}^{\theta}q(v_0)^{\theta-1}} \dx + \frac{1-\theta}{\theta}\intQT u_1^{\theta}q(v)^{\theta-2}q'(v)\pa_t v \dx\dt,
	\end{align*}
	and
	\begin{align*}
		(H_{13})&= -a_1\intQT (1-u)u_1^{\theta}(q(v))^{\theta-1} \dx\dt.
	\end{align*}
Integration by parts in the second integral gives
	\begin{align*}
		&(H_{12}) = d_1\intQT \na u_1 \cdot \na[q(v)u_1]^{\theta-1} \dx\dt \\
		&= d\intQT \na u_1\cdot (\theta-1)[q(v)u_1]^{\theta-2}\sbra{u_1q'(v)\na v + q(v)\na u_1} \dx\dt \\
		&= d(\theta-1)\intQT (q(v))^{\theta-2}u_1^{\theta-1}q'(v)\na u_1 \cdot \na v \dx\dt + d(\theta-1)\intQT (q(v))^{\theta-1}u_1^{\theta-2}|\na u_1|^2 \dx\dt\\
		&= d(\theta-1)\intQT (q(v))^{\theta-2}u_1^{\theta-1}q'(v)\na u_1 \cdot \na v \dx\dt + \frac{4d(\theta-1)}{\theta^2}\intQT (q(v))^{\theta-1} \left|\na(u_1^{\theta/2}) \right|^2 \dx\dt.
	\end{align*}
	Thus, we have
	\begin{align*}
		(H_1) = &\frac{1}{\theta}\intO \sbra{u_1(T)^{\theta}q(v(T))^{\theta-1} - u_{1,0}^{\theta}q(v_0)^{\theta-1}} \dx + \frac{1-\theta}{\theta}\intQT u_1^{\theta}(q(v))^{\theta-2}q'(v)\pa_tv \dx\dt\\
		&- a_1\intQT (1-u)u_1^{\theta}(q(v))^{\theta-1} \dx\dt + d(\theta-1)\intQT (q(v))^{\theta-2}q'(v)u_1^{\theta-1}\na u_1\cdot \na v \dx\dt\\
		&+ \frac{4d(\theta-1)}{\theta^2}\intQT (q(v))^{\theta-1}\left|\na(u_1^{\theta/2})\right|^2 \dx\dt.
	\end{align*}	
	and, similarly,
	\begin{align*}
		&(H_2) = \frac{1}{\theta}\intO \sbra{u_2(T)^{\theta}p(v(T))^{\theta-1} - u_{2,0}^{\theta}p(v_0)^{\theta-1}} \dx + \frac{1-\theta}{\theta}\intQT u_2^{\theta}(p(v))^{\theta-2}p'(v)\pa_tv \dx\dt\\
		&- a_2\intQT (1-u)u_2^{\theta}(p(v))^{\theta-1} \dx\dt + (d+D)(\theta-1)\intQT (p(v))^{\theta-2}p'(v)u_2^{\theta-1}\na u_2\cdot \na v \dx\dt\\
		&+ \frac{4(d+D)(\theta-1)}{\theta^2}\intQT (p(v))^{\theta-1}\left|\na(u_2^{\theta/2})\right|^2 \dx\dt.
	\end{align*}
	By inserting the estimates for $(H_1)$ and $(H_2)$ into \eqref{e2}, we obtain
	{\allowdisplaybreaks
	\begin{align*}
		\frac{1}{\theta}\intO \sbra{u_{1,0}^{\theta}q(v_0)^{\theta-1} + u_{2,0}^{\theta}p(v_0)^{\theta-1}} \dx + \frac{4d(1-\theta)}{\theta^2}\intQT q(v)^{\theta-1}\left|\na (u_1^{\theta/2})\right|^2 \dx\dt\\
		\frac{4(d+D)(1-\theta)}{\theta^2}\intQT p(v)^{\theta-1}\left|\na (u_2^{\theta/2})\right|^2 \dx\dt\\
		\leq \frac{1}{\theta}\intO \sbra{u_1(T)^{\theta}q(v(T))^{\theta-1} + u_2(T)^{\theta}p(v(T))^{\theta-1}}\dx\\
		+ \frac{1-\theta}{\theta}\intQT \frac{q'(v)}{q(v)}q(v)^{\theta-1}u_1^{\theta}\pa_t v \dx\dt + \frac{1-\theta}{\theta}\intQT \frac{p'(v)}{p(v)}p(v)^{\theta-1}u_2^{\theta}\pa_t v\dx\dt\\
		-a_1 \intQT (1-u)u_1^{\theta}(q(v))^{\theta-1} \dx\dt - a_2\intQT (1-u)u_2^{\theta}(p(v))^{\theta-1}\dx\dt\\
		+ d(\theta-1)\intQT \frac{q'(v)}{q(v)}q(v)^{\theta-1}u_1^{\theta-1}\na u_1\cdot \na v \dx\dt\\
		+ (d+D)(\theta-1)\intQT \frac{p'(v)}{p(v)}p(v)^{\theta-1}u_2^{\theta-1}\na u_2 \cdot \na v \dx\dt\\
		=: \sum_{i=1}^{7}(A_{i}).
	\end{align*}}
	Thanks to the lower bounds in Lemma \ref{lem:lower_bound} and $\theta\in (0,1)$, there exists $\eta_T>0$ such that
	\begin{equation}\label{e5}
	\begin{gathered}
		q(v(x,T))^{\theta-1} +p(v(x,T))^{\theta-1} \leq \eta_T \text{ a.e. }x\in\Omega,\\
		 q(v(x,t))^{\theta-1}+ p(v(x,t))^{\theta-1} \leq \eta_T \text{ a. e } (x,t)\in Q_T.
		\end{gathered}
	\end{equation}
	We estimates the terms $(A_i), i=1,\ldots, 7$ separately. First,
	\begin{equation*}
		(A_1) \leq \frac{\eta_T^{\theta}}{\theta}\intO \bra{ u_1(T)^{\theta} + u_2(T)^{\theta}} \dx \leq C_{\theta,T}\intO \bra{1+u_1(T) + u_2(T)} \dx \leq C_{\theta,T,M},
	\end{equation*}
	thanks to the $L^1(\Omega)$-bounds in Lemma \ref{lem:L1_bound}. For $(A_2)$, we use the assumption \eqref{A1} and the estimate \eqref{e5} to calculate
	\begin{equation*}
	\begin{aligned}
		(A_2) &\leq \frac{1-\theta}{\theta}\sup_{s>0}\left|\frac{q'(s)}{q(s)} \right|\eta_T^{\theta-1}\intQT |u_1|^{\theta}|\pa_tv| \dx\dt\\
		&\leq C_{\theta,T,q}\intQT(1+|u_1|)|\pa_tv|\dx\dt \quad (\text{since } \theta \in (0,1))\\ 
		&\leq C_{\theta,T,q}\bra{1+\|u_1\|_{L^2_{x,t}}^2 + \|\pa_tv\|_{L^2_{x,t}}^2} \leq C_{\theta,T,q}
	\end{aligned}
	\end{equation*}
	thanks to parts (i) and (ii) of Proposition \ref{main-estimates}. Similarly, 
	\begin{equation*}
		(A_3) \leq C_{\theta,T,p}\bra{1 + \|u_2\|_{L^2_{x,t}}^{2} + \|\pa_tv\|_{L^2_{x,t}}^2} \leq C_{\theta,T,p}.
	\end{equation*}
	The estimates of $(A_4)$ and $(A_5)$ are immediate due to the bounds \eqref{e5} and part (i) of Proposition \ref{main-estimates},
	\begin{equation*}
		(A_4) + (A_5) \leq C_{a_1,a_2,p,q,T,\theta}.
	\end{equation*}
	For $(A_6)$, we use \eqref{bound_pq} in Assumption \eqref{A1}, \eqref{e5}, H\"older's and Young's inequality to estimate
	\begin{align*}
		(A_6) &\leq \frac{2d(1-\theta)}{\theta}\intQT \left|\frac{q'(v)}{q(v)}\right||q(v)|^{\frac{\theta-1}{2}}\left|q(v)^{\frac{\theta-1}{2}}|\na(u_1^{\theta/2})| \right||u_1|^{\theta/2}|\na v| \dx\dt\\
		&\leq C_{d,\theta,T,q}\sbra{\intQT q(v)^{\theta-1}\left|\na(u_1^{\theta/2}) \right|^2 \dx\dt}^{1/2}\sbra{\intQT |u_1|^{\theta}|\na v|^2\dx\dt}^{1/2} \\
		&\leq \frac{2d(1-\theta)}{\theta^2}\intQT q(v)^{\theta-1}\left|\na(u_1^{\theta/2})\right|^2 \dx\dt + C_{d,\theta,T,q}\intQT |u_1|^{\theta}|\na v|^2 \dx\dt.
	\end{align*}
	Further, by H\"older's inequality we have
	\begin{align*}
		\intQT |u_1|^{\theta}|\na v|^2 \dx\dt &\leq \sbra{\intQT |\na v|^{2\cdot \frac{2+\delta_0}{2}}\dx\dt}^{\frac{2}{2+\delta_0}}\sbra{\intQT |u_1|^{\theta\cdot\frac{2+\delta_0}{\delta_0}}\dx\dt}^{\frac{\delta_0}{2+\delta_0}}\\
		&= \|\na v\|_{L^{2+\delta_0}_{x,t}}^2\|u_1\|_{L^{\frac{\theta}{\delta_0}(2+\delta_0)}_{x,t}}^{\theta}\\
		&\leq C_T\|\na v\|_{L^{2+\delta_0}_{x,t}}^2\|u_1\|_{L^{2+\delta_0}_{x,t}}^{\theta}
	\end{align*}
	since $\theta \in (0,\delta_0)$. Therefore, thanks to parts (i) and (ii) of Proposition \ref{main-estimates} we obtain 
	\begin{equation*}
		(A_6) \leq \frac{2d(1-\theta)}{\theta^2}\intQT q(v)^{\theta-1}\left|\na(u_1^{\theta/2}) \right|^2\dx\dt + C_{d,\theta,T,q}.
	\end{equation*}
	Similarly, we have the estimate for $(A_7)$,
	\begin{equation*}
		(A_7) \leq \frac{2(d+D)(1-\theta)}{\theta^2}\intQT p(v)^{\theta-1}\left|\na(u_2^{\theta/2})\right|^2 \dx\dt + C_{d,D,\theta,T,p}.
	\end{equation*}
	Inserting all these estimates of $(A_1)-(A_7)$ 
	we obtain
	\begin{equation}\label{e7}
		\intQT q(v)^{\theta-1}\left|\na(u_1^{\theta/2}) \right|^2 \dx\dt + \intQT p(v)^{\theta-1}\left|\na(u_2^{\theta/2}) \right|^2 \dx\dt \leq C_T
	\end{equation}
	where the constant $C_T$ depends on $T$ and other parameters, $\theta$, $d$, $D$, $p$, $q$, $a_1$, $a_2$, but {\it not on $\eps$}.  Now we show that there exists $\gamma_T>0$ depending on $T$, but not on $\eps$, such that
	\begin{equation}\label{e6}
		q(v(x,t))^{\theta-1} \geq \gamma_T \quad \text{ and } \quad p(v(x,t))^{\theta-1} \geq \gamma_T \quad \text{ for } a.e. (x,t)\in Q_T.
	\end{equation}	
	Since $0<\theta<1$, it is sufficient to show that $q(v(x,t))$ and $p(v(x,t))$ are essentially bounded in $Q_T$, which follows from \eqref{A3}. Indeed, it's immediate in the case of \eqref{A3.2}. When either \eqref{A3.1} or \eqref{A3.3} holds, we have $\|v\|_{L^{\infty}_{x,t}} \leq C_T$ (part (ii) of Proposition \ref{main-estimates}), and therefore the boundedness of $p(v(x,t))$ and $q(v(x,t))$ follows straightforwardly from the smoothness of $p$ and $q$.

	It remains to show \eqref{u-gradient} using \eqref{e1}. Let $1>\xi>0$ to be determined. By direct computations, we have
	\begin{equation*}
		\left|\na (u_1^{\theta/2})\right|^{1+\xi} = \bra{\frac \theta 2}^{1+\xi}|u_1|^{\frac{(\theta-1)(1+\xi)}{2}}\left|\na u_1\right|^{1+\xi}.
	\end{equation*}
	Therefore,
	\begin{gather}
		\intQT |\na u_1|^{1+\xi} \dx\dt = \bra{\frac 2\theta}^{1+\xi}\intQT|u_1|^{\frac{(1-\theta)(1+\xi)}{2}}\left|\na (u_1^{\theta/2})\right|^{1+\xi} \dx\dt \notag \\
		\leq \bra{\frac 2\theta}^{1+\xi}\bra{\intQT \left|\na(u_1^{\theta/2})\right|^{(1+\xi)\frac{2}{1+\xi}}\dx\dt}^{\frac{1+\xi}{2}}\bra{\intQT |u_1|^{\frac{(1-\theta)(1+\xi)}{2} \cdot \frac{2}{1-\xi} }\dx\dt}^{\frac{1-\xi}{2}} \label{e8}\\
		= \bra{\frac 2\theta}^{1+\xi}\bra{\intQT \left|\na(u_1^{\theta/2})\right|^{2}\dx\dt}^{\frac{1+\xi}{2}}\bra{\intQT |u_1|^{\frac{(1-\theta)(1+\xi)}{1-\xi} }\dx\dt}^{\frac{1-\xi}{2}}. \notag
	\end{gather}
	For arbitrary $\theta \in (0,\min\{1,\delta_0\})$ we take $\xi = (1+\theta+\delta_0)/(3-\theta+\delta_0) \in (0,1)$ so that
	\begin{equation*}
		\frac{(1-\theta)(1+\xi)}{1-\xi} = 2 + \delta_0.
	\end{equation*}
	Thus, by applying part (i) of Proposition \ref{main-estimates} and \eqref{e1} into \eqref{e8} we get
	\begin{equation*}
		\intQT |\na u_1|^{1+\xi} \dx\dt \leq C_T.
	\end{equation*}
	The proof of 
	\begin{equation*}
			\intQT |\na u_2|^{1+\xi} \dx\dt \leq C_T.
		\end{equation*}
		follows exactly the same, and thus the proof of Proposition \ref{main-estimates} part (iii) is complete.
\end{proof}

The following lemma is the consequence of the proof of part (iii) above, and will be important in our subsequent analysis.
\begin{lemma}\label{lem:new_estimate_2}
	There exist constants $\theta \in (0,1)$ and $C_T>0$ independent of $\eps>0$ such that
	\begin{equation}\label{e10}
		\norm{(q(\veps)\uoneeps)^{\theta/2} - (p(\veps)\utwoeps)^{\theta/2}}_{L^{2}_{x,t}} \leq C_T \sqrt{\eps}.
	\end{equation}
	Consequently,
	\begin{equation}\label{bb2}
			q(\veps)\uoneeps - p(\veps)\utwoeps \xrightarrow{\eps\to 0} 0 \quad \text{a.e. in } \; Q_T.
		\end{equation}
\end{lemma}
\begin{proof}
We obtain \eqref{e10} by repeating the proof of part (iii) of Proposition \ref{main-estimates} in the previous section, however, rather than \eqref{e3} we use the inequality
	\begin{equation}\label{e11}
		-(x-y)(x^{\theta-1}-y^{\theta-1}) \geq C_{\theta} \abs{x^{\theta/2}-y^{\theta/2}}^2 \quad \text{ for all } x, y > 0 \text{ and } \theta \in (0,1)
	\end{equation}
	where $C_{\theta}$ is independent of $\eps$. Define $h(z) = z^{2/\theta}$ we can estimate
	\begin{align*}
		\intQT|q(\veps)\uoneeps - p(\veps)\utwoeps|\dx\dt &= \intQT\left|h\bra{(q(\veps)\uoneeps)^{\theta/2}} - h\bra{(p(\veps)\utwoeps)^{\theta/2}}\right|\dx\dt\\
		&= \intQT\left|(q(\veps)\uoneeps)^{\theta/2} - (p(\veps)\utwoeps)^{\theta/2}\right||h'(\Theta(x,t))|\dx\dt
	\end{align*}
	where $\Theta(x,t)$ satisfies
	\begin{equation*}
		|\Theta(x,t)| \leq |q(\veps)\uoneeps|^{\theta/2} + |p(\veps)\utwoeps|^{\theta/2} \leq C_T\left(|\uoneeps|^{\theta/2} + |\utwoeps|^{\theta/2} \right).
	\end{equation*}
	Thus,
	\begin{equation*}
		|h'(\Theta(x,t))| = (2/\theta)|\Theta(x,t)|^{(2-\theta)/\theta} \leq C_{T,\theta}\bra{|\uoneeps|^{(2-\theta)/2}+|\utwoeps|^{(2-\theta)/2}}.
	\end{equation*}
	We then can estimate further with H\"older's inequality
	\begin{align*}
		&\intQT|q(\veps)\uoneeps - p(\veps)\utwoeps|\dx\dt\\
		&\leq C_{T,\theta}\intQT \left|(q(\veps)\uoneeps)^{\theta/2} - (p(\veps)\utwoeps)^{\theta/2}\right|\bra{|\uoneeps|^{(2-\theta)/2}+|\utwoeps|^{(2-\theta)/2}}\dx\dt\\
		&\leq C_{T,\theta}\|(q(\veps)\uoneeps)^{\theta/2} - (p(\veps)\utwoeps)^{\theta/2}\|_{L^2_{x,t}}\left(\|\uoneeps\|_{L^{2-\theta}_{x,t}}^{(2-\theta)/2}+\|\utwoeps\|_{L^{2-\theta}_{x,t}}^{(2-\theta)/2} \right)\\
		&\leq C_{T,\theta}\|(q(\veps)\uoneeps)^{\theta/2} - (p(\veps)\utwoeps)^{\theta/2}\|_{L^2_{x,t}}
	\end{align*}
	due to the bounds of $\uoneeps$ and $\utwoeps$ in part (i) of Proposition \ref{main-estimates}. From this, it follows from \eqref{e10} that $\|q(\veps)\uoneeps - p(\veps)\utwoeps\|_{L^1_{x,t}} \xrightarrow{\eps\to 0} 0$ and consequently \eqref{bb2}.
\end{proof}

\subsection{Passing to the limit $\eps\to 0$}
We start with the strong convergence of $\{\veps\}_{\eps>0}$.
\begin{proposition} \label{prop:03}
	Assume \eqref{A1}, \eqref{A2}, \eqref{A3}. Then there exists $v \in L^2_{x,t}$ and a subsequence of $\{\veps\}_{\eps>0}$ (denoted by $\{\veps\}_{\eps>0}$) such that
	\begin{equation} \label{con:02}
		\veps \to v \quad \text{strongly in } L^{2}_{x,t}.
	\end{equation}
\end{proposition}
\begin{proof}
It follows from \eqref{estimata-v-2+delta} that $\veps \in L^{2+\delta_0}_tH^1_x$ and $\partial_t \veps \in L^{2+\delta_0}_{x,t}$.
By the Aubin-Lions compactness lemma \cite{Roubicek-2013}, p.~208,
the sequence $\{\veps\}_{\eps>0}$ is relatively compact in $L^2_{x,t}$. Thus there exists $v \in L^2_{x,t}$ and a subsequence of $\{\veps\}_{\eps > 0}$ (denoted again by $\{\veps\}_{\eps > 0}$) such that $\veps \to v$ strongly in $L^{2}(Q_T)$.
\end{proof}

The next lemma shows the strong convergence of $\uoneeps$ and $\utwoeps$.
\begin{lemma}\label{lem2}
	Assume \eqref{A1}, \eqref{A2}, \eqref{A3}. Then there exist $u_1, u_2 \in L^{2+\delta}_{x,t}$ such that, 
	\begin{equation*}
		u_1 = \frac{p(v)(u_1+u_2)}{q(v) + p(v)} \quad \text{ and } \quad u_2 = \frac{q(v)(u_1 + u_2)}{q(v) + p(v)},
	\end{equation*}
	and, up to a subsequence,
	\begin{equation*}
		\|\uoneeps - u_1\|_{L^2_{x,t}} + \|\utwoeps-u_2\|_{L^2_{x,t}} \xrightarrow{\eps\to 0} 0.
	\end{equation*}
\end{lemma}
\begin{proof}
	Thanks to part (i) of Proposition \eqref{main-estimates}, there exist $u_1, u_2\in L^{2+\delta_0}_{x,t}$ such that
	\begin{equation*}
		\uoneeps \rightharpoonup u_1 \;\text{ and } \; \utwoeps \rightharpoonup u_2 \quad \text{ weakly in } \quad L^{2+\delta_0}_{x,t}.
	\end{equation*}
	Unlike the case of $\veps$, it seems not possible to apply the Aubin-Lions lemma directly to $\uoneeps$ and $\utwoeps$ to get the strong convergence since there is no control (which is uniform in $\eps>0$) of their time derivative. We overcome this difficulty by first showing
	\begin{equation}\label{convergence-sum-u}
		\uoneeps + \utwoeps \to u_1 + u_2 \; \text{ strongly in } \; L^2_{x,t}.
	\end{equation}
	Denote by 
	\begin{equation*}
		\varkappa = \max\left\{\frac{1+\xi}{\xi}; \frac{2+\delta_0}{\delta_0} \right\}
	\end{equation*}
	with $\delta_0$ in \eqref{A2} and $\xi$ in \eqref{u-gradient}. Summing the first and second equations of \eqref{eq:01}, then testing the resultant with $\varphi \in L^{\varkappa}_tW^{1,\varkappa}_x$ give
	\begin{equation}\label{bb1}
	\begin{aligned}
		\intQT \pa_t(\uoneeps + \utwoeps)\varphi \dx\dt &= -d\intQT \nabla \uoneeps \cdot \nabla \varphi \dx\dt - (d+D)\intQT \nabla \utwoeps \cdot \na \varphi \dx\dt\\
		&\; +  \intQT (a_1\uoneeps + a_2\utwoeps)(1-\ueps)\varphi \dx\dt.
	\end{aligned}
	\end{equation}
	By using H\"older inequality, it holds for $i=1,2$,
	\begin{equation*}
		\left| \intQT \na u_i^{\eps} \cdot \na\varphi\dx\dt\right| \leq \|\na u_i^{\eps}\|_{L^{1+\xi}_{x,t}}\|\na \varphi\|_{L^{(1+\xi)/{\xi}}_{x,t}} \leq C_T\|\na \varphi\|_{L^{(1+\xi)/\xi}_{x,t}} \leq C_T\|\varphi\|_{L^{\varkappa}_tW^{1,\varkappa}_x}
	\end{equation*}
	thanks to part (iii) of Proposition \ref{main-estimates}. For the third integral on the right hand side of \eqref{bb1}, we use part (i)	 of Proposition \ref{main-estimates} and H\"older's inequality to estimate
	\begin{equation*}
	\begin{aligned}
		\left|\intQT(a_1\uoneeps + a_2\utwoeps)(1-\ueps)\varphi\dx\dt \right| &\leq C\intQT \left(1+|\uoneeps|^2 + |\utwoeps|^2\right)\varphi \dx\dt\\
		&\leq C\left(\|\varphi\|_{L^1_{x,t}} + \bra{\|\uoneeps\|_{L^{2+\delta_0}_{x,t}}^2+\|\utwoeps\|_{L^{2+\delta_0}_{x,t}}^2}\|\varphi\|_{L^{(2+\delta_0)/\delta_0}_{x,t}} \right)\\
		&\leq C_T\|\varphi\|_{L^{(2+\delta_0)/\delta_0}_{x,t}} \leq C_T\|\varphi\|_{L^{\varkappa}_tW^{1,\varkappa}_x}.
	\end{aligned}		
	\end{equation*}
	Therefore, by duality, we obtain from \eqref{bb1} that
	\begin{equation*}
		\{\pa_t(\uoneeps+\utwoeps) \}_{\eps>0} \; \text{ are bounded in } \; L^{\varkappa'}_t(W^{1,\varkappa}_x)' \; \text{ uniformly in } \; \eps>0,
	\end{equation*}
	where $\varkappa' = \varkappa/(\varkappa-1)$ and $(W^{1,\varkappa}_x)'$ is the dual space of $W^{1,\varkappa}_x$. Now we can apply the Aubin-Lions lemma to obtain some $\chi \in L^2_{x,t}$ such that, up to a subsequence,
	\begin{equation*}
		\uoneeps + \utwoeps \xrightarrow{\eps\to 0} \chi \quad \text{ in } \quad L^{2}_{x,t}.
	\end{equation*}
	By uniqueness of weak limit, it is clear that $\chi = u_1 + u_2$. To prove the strong convergence for $\uoneeps$ and $\utwoeps$ separately, we rewrite
	\begin{equation*}
	\uoneeps = \frac{p(\veps)(\uoneeps + \utwoeps) + [q(\veps) \uoneeps - p(\veps) \utwoeps]}{q(\veps) + p(\veps)}.
	\end{equation*}
	Thanks to \eqref{convergence-sum-u} and \eqref{bb2}, it yields
	\begin{equation}\label{bb3}
		\uoneeps \xrightarrow{\eps\to 0} \frac{p(v)(u_1 + u_2)}{q(v)+p(v)} \quad \text{ a.e. in } \; Q_T.
	\end{equation}
	Note that $\{\uoneeps\}$ is bounded in $L^{2+\delta_0}_{x,t}$. Therefore, the convergence in \eqref{bb3} is in fact strong in $L^2_{x,t}$. Thus, we have
	\begin{equation*}
		u_1 = \frac{p(v)(u_1 + u_2)}{q(v)+p(v)} \quad \text{ and } \quad \lim_{\eps\to 0}\|\uoneeps - u_1\|_{L^2_{x,t}} = 0.
	\end{equation*}
	The relations
	\begin{equation*}
		u_2 = \frac{q(v)(u_1+u_2)}{q(v) + p(v)} \quad \text{ and } \quad \lim_{\eps\to 0}\|\utwoeps-u_2\|_{L^2_{x,t}} = 0
	\end{equation*}
	follows the same way by rewriting
	\begin{equation*}
	\utwoeps = \frac{q(\veps)(\uoneeps + \utwoeps) - [q(\veps) \uoneeps - p(\veps) \utwoeps]}{q(\veps) + p(\veps)}.
	\end{equation*}
\end{proof}
We are now ready to prove the main results in Theorem \ref{thm:01}.
\begin{proof}
Adding \eqref{eq:w01} and \eqref{eq:w02} together, we obtain
\begin{equation} \label{eq:22} \begin{aligned}
-\iint_{Q_T} (\uoneeps+\utwoeps) \partial_t \varphi \dx \dt & - \int_{\Omega} (u_{1,0}+u_{1,0}) \varphi(0,x) \dx - \iint_{Q_T} (d\uoneeps + (d+D)\utwoeps)\Delta \varphi \dx \dt \\
& = \iint_{Q_T} (1 - (\uoneeps + \utwoeps))(a_1 \uoneeps + a_2 \utwoeps) \varphi \dx \dt
\end{aligned}\end{equation}
for $\varphi \in C^{2,1}(\overline{Q}_T)$ which satisfies $\varphi(x,T) = 0$ and $\partial_{\nu} \varphi = 0$ on $\partial \Omega \times [0,T]$.
Due to Lemma \ref{lem2}, the following convergences as $\eps\to 0$ are immediate
\begin{equation*}
	-\intQT (\uoneeps + \utwoeps)\pa_t\varphi\dx\dt \to -\intQT u\pa_t\varphi \dx\dt,
\end{equation*}
\begin{equation*}
	-\intQT (d\uoneeps + (d+D)\utwoeps)\Delta \varphi\dx\dt \to -\intQT \bra{d + D\frac{q(v)}{q(v)+p(v)}}u\Delta \varphi\dx\dt.
\end{equation*}
The convergence
\begin{equation*}
	\iint_{Q_T} \bra{1-(\uoneeps + \utwoeps)}(a_1 \uoneeps + a_2 \utwoeps) \varphi \dx \dt \to \iint_{Q_T} \bra{1-(u_1 + u_2)}(a_1 u_1 + a_2 u_2) \varphi \dx \dt.
\end{equation*}
follows from the strong convergence $\uoneeps \to u_1$ and $\utwoeps \to u_2$ in Lemma \ref{lem2}. By substituting $u_1 = q(v)u/(q(v) + p(v))$ and $u_2 = p(v)u/(q(v)+p(v))$ we obtain in fact
\[ \begin{aligned} 
\iint_{Q_T} (1 - (\uoneeps + \utwoeps))(a_1 \uoneeps + a_2 \utwoeps) \varphi \dx \dt \to  \iint_{Q_T}  \frac{a_1 p(v) + a_2 q(v)}{q(v)+p(v)} (1 - u)u   \varphi \dx \dt
\end{aligned}\]
as $\eps \to 0$.
Inserting these convergences into \eqref{eq:22} gives that $u = u_1+u_2$ satisfies \eqref{eq:w04}, which concludes the proof of the theorem.
\end{proof}

\appendix
\section{Proof of Lemma \ref{maximal_regularity}}\label{appA}
\newcommand{\wh}{\widehat}
By \cite[Theorem 1]{lamberton1987equations}, if $f \in L^r_{x,t}$ and $u$ is the solution to
\begin{equation}\label{eq_u}
\begin{cases}
	\pa_t u - \Delta u = f, &(x,t)\in Q_T,\\
	\nabla u \cdot \nu = 0, &(x,t)\in \partial\Omega\times(0,T),\\
	u(x,0) = 0, &x\in\Omega,
\end{cases}
\end{equation}
then there exists a constant $C_{\mr,q}>0$ independent of $T$ such that 
\begin{equation}\label{mr}
	\norm{\Delta u}_{L^r_{x,t}} \leq C_{\mr,q}\norm{f}_{L^r_{x,t}}.
\end{equation}
At the first glance, \eqref{dual} looks like a backward heat equation, but by defining $\wh{\phi}(x,t) = \phi(x,t/m)$ and $\tau = T-t$, it becomes the usual forward parabolic equation
\begin{equation*}
\begin{cases}
\pa_\tau \wh{\phi}(x,\tau) - \Delta \wh\phi(x,\tau) = \frac{1}{m}\wh{\theta}(x,\tau),  & x\in\Omega,  \tau\in (0,mT),\\
\nabla \wh{\phi}(x,\tau)\cdot \nu = 0, &x\in\pa\Omega, \tau\in (0,mT),\\
\wh{\phi}(x,0) = 0
\end{cases}
\end{equation*}
where $\wh{\theta}(x,\tau) = \theta(x,\tau/m) = \theta(x,(T-t)/m)$.  By applying \eqref{mr}, it follows that
\begin{equation*}
\norm{\Delta \wh{\phi}}_{L^r(0,mT;L^r(\Omega))} \leq C_{\mr,q}\norm{\frac 1m\wh{\theta}}_{L^r(0,mT;L^r(\Omega))} = \frac{C_{\mr,q}}{m}\norm{\wh{\theta}}_{L^r(0,mT;L^r(\Omega))}.
\end{equation*}
Therefore,
\begin{equation*}
	\iint_{Q_{mT}}|\Delta \wh\phi(x,\tau)|^qdxd\tau \leq \left(\frac{C_{\mr,q}}{m}\right)^q\iint_{Q_{mT}}|\wh\theta(x,\tau)|^qdxd\tau.
\end{equation*}
Making the change of variable $s = \tau/m$ yields
\begin{equation*}
	m\norm{\Delta \phi}_{L^r_{x,t}}^r \leq \frac{C_{\mr,q}^q}{m^{q-1}}\norm{\theta}_{L^r_{x,t}}^r
\end{equation*}
which implies immediately the desired estimate \eqref{max}.

To show \eqref{pstar}, we first prove that $C_{\mr,2} \leq 1$. Indeed, by multiplying \eqref{eq_u} with $-\Delta u$ then integrating on $Q_T$, we get
\begin{equation*}
	\frac 12\int_{\Omega}|\nabla u(\cdot,T)|^2dx + \iint_{Q_T}|\Delta u|^2dxdt  = -\iint_{Q_T}f\Delta u dxdt \leq \frac 12\norm{f}_{L^2_{x,t}}^2 + \frac 12\norm{\Delta u}_{L^2_{x,t}}^2
\end{equation*}
and thus
\begin{equation*}
	\norm{\Delta u}_{L^2_{x,t}} \leq \norm{f}_{L^2_{x,t}}
\end{equation*}
which implies $C_{\mr,2} \leq 1$. It is obvious that for any $\omega_1,\omega_2>0$,
\begin{equation*}
	\frac{|\omega_1-\omega_2|}{\omega_1+\omega_2}C_{\mr,2} < 1.
\end{equation*}
Therefore, it is sufficient to show that 
\begin{equation*}
	C_{\mr,2}^{-}:= \liminf_{\eta\to 0^+}C_{\mr,2-\eta} \leq C_{\mr,2}.
\end{equation*}
By the Riesz-Thorin interpolation theorem, cf. \cite[Chapter 2]{lunardi2018interpolation}, we have for small enough $\eta>0$
\begin{equation*}
	C_{\mr,2_\eta} \leq C_{\mr,2}^{1/2}C_{\mr,2-\eta}^{1/2}
\end{equation*}
where
\begin{equation*}
	2_\eta = 2 - \frac{2\eta}{4-\eta}.
\end{equation*}
Letting $\eta \to 0$ yields
\begin{equation*}
	C_{\mr,2}^{-} \leq C_{\mr,2}^{1/2}\left(C_{\mr,2}^{-}\right)^{1/2},
\end{equation*}
which gives consequently $C_{\mr,2}^{-} \leq C_{\mr,2}$ and the proof of Lemma \ref{maximal_regularity} is therefore complete.

\section{Global weak solutions for $\eps>0$}\label{appB}
\newcommand{\udel}{u_1^{\delta}}
\newcommand{\uudel}{u_2^{\delta}}
\newcommand{\vdel}{v^{\delta}}
\begin{proof}[Proof of Proposition \ref{prop:01}]
We provide in this section the proof of global weak solutions in Proposition \ref{prop:01}. For simplicity, we drop the superscript $\eps$ as it is fixed in the whole proof. Denote by 
\begin{align*}
	R(u_1,u_2,v):= \frac{1}{\eps}q(v)u_1 - p(v)u_2.
\end{align*}
Let $\delta>0$, we consider the following approximation system of \eqref{eq:01},
\begin{equation}\label{approx_sys}
\begin{cases}
	\pa_t \udel - D\Delta \udel = a_1(1-\udel - \uudel)\udel - R(\udel,\uudel,\vdel)\big[1+\delta|R(\udel,\uudel,\vdel)|\big]^{-1},\\
	\pa_t \uudel - (d+D)\Delta \uudel = a_2(1 - \udel - \uudel)\uudel + R(\udel,\uudel,\vdel)\big[1+\delta|R(\udel,\uudel,\vdel)|\big]^{-1},\\
	\pa_t \vdel - D_v\Delta \vdel = \alpha(\udel + \uudel) - \beta \vdel
\end{cases}
\end{equation}
subject to homogeneous Neumann boundary condition
\begin{equation*}
	\nabla \udel \cdot \nu(x) = \na \uudel \cdot \nu(x) = \na \vdel \cdot \nu(x) = 0, \text{ on } \pa\Omega\times(0,T),
\end{equation*}
and non-negative bounded initial data
\begin{equation*}
	\udel(x,0) = u_{1,0}^{\delta}(x), \quad \uudel(x,0) = u_{2,0}^{\delta}(x), \quad \vdel (x,0) = \vdel_0(x), \quad \text{ for } x\in\Omega,
\end{equation*}
We require that
\begin{equation*}
	\|u_{1,0}^{\delta} - u_{1,0}\|_{L^{2+\kappa}_x} + \|u_{2,0}^{\delta} - u_{2,0}\|_{L^{2+\kappa}_x} + \|\vdel_0 - v_0\|_{L^\infty_x} \xrightarrow{\delta \to 0} 0.
\end{equation*}
Denote by $f_1(\udel, \uudel, \vdel), f_2(\udel, \uudel, \vdel), f_3(\udel, \uudel, \vdel)$ the right-hand sides of \eqref{approx_sys}. It's immediate to check that these nonlinearities are at most polynomial, locally Lipschitz continuous, quasi-positive, i.e. for any $i\in \{1,2,3\}$,
\begin{equation*}
	f_i(z_1,z_2,z_3) \geq 0 \quad \text{ for all } \; z\in \mathbb R_+^3 \; \text{with} \; z_i = 0,
\end{equation*}
and are bounded above by linear functions, i.e. for any $i\in \{1,2,3\}$,
\begin{equation*}
	f_i(z_1,z_2,z_3) \leq C\bra{1+z_1+z_2+ \frac{1}{\delta}} \quad \text{ for all } \; z\in \mathbb R_+^3.
\end{equation*}
The classical theory of reaction-diffusion systems, see e.g. \cite{morgan1989global}, shows that there exists a unique global non-negative bounded solutions.

We now aim to pass to the limit as $\delta\to 0$. By summing the first two equations of \eqref{approx_sys}, we have
\begin{equation*}
	\pa_t(\udel + \uudel) - \Delta(D\udel + (D+d)\uudel) \leq \max\{a_1,a_2\}(\udel + \uudel).
\end{equation*}
Thanks to the improved duality as in the proof of part (i) of Proposition \ref{main-estimates} (or from \cite{canizo2014improved}), we have the following estimates
\begin{equation}\label{bounduu}
	\|\udel\|_{L^{2+\kappa}_{x,t}} + \|\uudel\|_{L^{2+\kappa}_{x,t}} \leq C_T
\end{equation}
where the constant $C_T$ might depend on $T>0$, $\|u_{1,0}^{\delta}\|_{L^{2+\kappa}_x}$, $\|u_{2,0}^{\delta}\|_{L^{2+\kappa}_x}$, $D$, $d$, $a_1, a_2$, \textit{but not on $\delta$}. From the heat regularisation from Lemma \ref{lem:heat_gradient}, it follows that when either \eqref{A3.1} or \eqref{A3.3} holds
\begin{equation}\label{boundv}
	\|\vdel\|_{L^{\infty}_{x,t}} \leq C_T,
\end{equation}
and when \eqref{A3'-b} holds,
\begin{equation}\label{boundv1}
	\|\vdel\|_{L^{\frac{(N+2)(2+\kappa)}{N-2(1+\kappa)}}_{x,t}} \leq C_T.
\end{equation}
From this and the growth condition in \eqref{A3'-b}, we have
\begin{equation}\label{bound_pandq}
	\|p(\vdel)\|_{\LQ{\lambda}} + \|q(\vdel)\|_{\LQ{\lambda}} \leq C_T
\end{equation}
where
\begin{equation*}
	\lambda:= \frac{(N-2)(2+\kappa)}{N-2(1+\kappa)} \geq 2 + \kappa.
\end{equation*}
From \eqref{bounduu} and \eqref{bound_pandq} and
\begin{align*}
	|f_1(\udel,\uudel,\vdel)| &\leq a_1(\udel + (\udel)^2 + \udel \uudel) + |R(\udel,\uudel,\vdel)|\\
	&\leq a_1(\udel + (\udel)^2 + \udel \uudel) + \frac 1\eps |q(\vdel)| \udel + \frac 1\eps |p(\vdel)|\uudel
\end{align*}
we conclude that
\begin{equation*}
	\|\pa_t \udel - D\Delta \udel\|_{\LQ{1+\frac{\kappa}{2}}} \leq C_{T,\eps}
\end{equation*}
where the constant $C_{T,\eps}$ \textit{does not depend on} $\delta$. Similarly,
\begin{equation*}
	\|\pa_t \uudel - (D+d)\Delta \uudel\|_{\LQ{1+\frac{\kappa}{2}}} \leq C_{T,\eps}
\end{equation*}
and
\begin{equation*}
	\|\pa_t \vdel - D_v\Delta \vdel\|_{\LQ{2+\kappa}} \leq C_{T,\eps}.
\end{equation*}
Classical Aubin-Lions lemma shows that, there exist a subsequence of $\{\udel,\uudel,\vdel\}$ (which we will not relabel) and a non-negative triple $(u_1,u_2,v)\in \LQ{1+\kappa/2}\times\LQ{1+\kappa/2} \times \LQ{2+\kappa}$ such that
\begin{equation*}
	\|\udel - u_1\|_{\LQ{1+\kappa/2}} + \|\uudel - u_2\|_{\LQ{1+\kappa/2}} + \|\vdel - v\|_{\LQ{2+\kappa}} \xrightarrow{\delta \to 0} 0.
\end{equation*}
Thanks to these strong convergence and the bounds \eqref{bounduu}, \eqref{boundv} or \eqref{boundv1} (depending $N$), and \eqref{bound_pandq}, we can deduce the following strong convergence in $\LQ{1}$ as $\delta \to 0$
\begin{equation}\label{strongL1}
	(\udel)^2 \to u_1^2, \quad (\uudel)^2 \to u_2^2, \quad \udel \uudel \to u_1u_2, \quad q(\vdel)\udel \to q(v)u_1, \quad p(\vdel)\uudel \to p(v)u_2.
\end{equation}
Consequently, we also have
\begin{equation}\label{strongL1_1}
	R(\udel,\uudel,\vdel)\big[1+\delta|R(\udel,\uudel,\vdel)|\big]^{-1} \longrightarrow R(u_1,u_2,v) \quad \text{ strongly in } \; \LQ{1}.
\end{equation}
Now we can test the system \eqref{approx_sys} against a test function $\varphi \in C^{2,1}(\overline{Q_T})$ satisfying $\varphi(x,T) = 0$ and $\na \varphi \cdot \nu = 0$ on $\pa\Omega\times(0,T)$, and let $\delta \to 0$. The strong convergence in \eqref{strongL1} and \eqref{strongL1_1} allows us to conclude that $(u_1,u_2,v)$ is a global weak solution to \eqref{eq:01}, which finishes the proof of Propositions \ref{prop:01}.
\end{proof}

\section*{Acknowledgements}
We would like to thank Dr. Cinzia Soresina for the fruitful discussion concerning the part of fast reaction limit.

\newcommand{\etalchar}[1]{$^{#1}$}


\medskip
\hspace*{0.85in}\textit{E-mail adress:} \href{mailto:jan.elias@uni-graz.at}{jan.elias@uni-graz.at}

\hspace*{0.85in}\textit{E-mail adress:} \href{mailto:izuhara@cc.miyazaki-u.ac.jp}{izuhara@cc.miyazaki-u.ac.jp}

\hspace*{0.85in}\textit{E-mail adress:} \href{mailto:mimura.masayasu@gmail.com}{mimura.masayasu@gmail.com}

\hspace*{0.85in}\textit{E-mail adress:} \href{mailto:quoc.tang@uni-graz.at}{quoc.tang@uni-graz.at}

\end{document}